\documentclass[11pt,reqno,twoside]{article}
\usepackage{mysty}

\usepackage{wrapfig}
\usepackage{mdwlist}

\usepackage{fullpage}
\usepackage{patchcmd,pifont}
\newcommand{\cmark}{\ding{51}}%
\newcommand{\xmark}{\ding{55}}%


\usepackage{enumitem}
\SetLabelAlign{parright}{\parbox[t]{\labelwidth}{\raggedleft{#1}}}
\setlist[description]{style=multiline,topsep=4pt,align=parright}%,font=\normalfont

\makeatletter
\let\reftagform@=\tagform@
\def\tagform@#1{\maketag@@@{(\ignorespaces\textcolor{black}{#1}\unskip\@@italiccorr)}}
\newcommand{\iref}[1]{\textup{\reftagform@{\tcr{\ref{#1}}}}}
\makeatother

\begin{document}
%%%%%%%%%%%%%%%%%%%%%%%%%%%%%%%%%%%
\title{Inertial Methods with Viscous and Hessian driven Damping for Non-Convex Optimization}
\author{Rodrigo Maulen-Soto, Jalal Fadili, Peter Ochs}
\date{\today\footnote{The insight and motivation for the study of inertial methods with viscous and Hessian driven damping came from the inspiring collaboration with our beloved friend and colleague Hedy Attouch before his unfortunate recent departure. We hope this paper is a valuable step in honoring his legacy.
}}
\maketitle

\begin{abstract}
In this paper, we aim to study non-convex minimization problems via second-order (in-time) dynamics, including a non-vanishing viscous damping and a geometric Hessian-driven damping. Second-order systems that only rely on a viscous damping may suffer from oscillation problems towards the minima, while the inclusion of a Hessian-driven damping term is known to reduce this effect without explicit construction of the Hessian in practice. There are essentially two ways to introduce the Hessian-driven damping term: explicitly or implicitly. For each setting, we provide conditions on the damping coefficients to ensure convergence of the gradient towards zero. Moreover, if the objective function is definable, we show global convergence of the trajectory towards a critical point as well as convergence rates. Besides, in the autonomous case, if the objective function is Morse, we conclude that the trajectory converges to a local minimum of the objective for almost all initializations. We also study algorithmic schemes for both dynamics and prove discrete analogues of the previous properties under appropriate stepsize conditions. In particular, we consider the case where the objective is only locally Lipschitz smooth and propose a backtracking strategy for which we establish convergence guarantees. Our work is the first one that handles this situation.
\end{abstract}
%\begin{flushleft}\end{flushleft}
%%%%%%%%%%%%%%%%%%%%%%%%%%%%%%%%%%%

%%%%%%%%%%%%%%%%%%%%%%%%%%%%%%%%%%%

%\tableofcontents

%%%%%%%%%%%%%%%%%%%%%%%%%%%%%%%%%%%
%\input{tex/sec-Intro-correction}

%%%%%%%%%%%%%%%%%%%%%%%%%%%%%%%%%%%%%%%%%%%%%%%%%%%%%%
\section{Introduction}\label{sec:intro}
%%%%%%%%%%%%%%%%%%%%%%%%%%%%%%%%%%%%%%%%%%%%%%%%%%%%%%

%%%%%%%%%%%%%%%%%%%%%%%%%%%%%%
\subsection{Problem statement}

Let us consider the minimization problem
\begin{equation}\label{P}\tag{P}
    \min_{x\in \R^d} f(x),
\end{equation}
where the objective function $f:\R^d\rightarrow\R$ satisfies the following standing assumptions:
\begin{equation}\label{H0}\tag{$\mathrm{H_0}$}
    \begin{cases}
        f\in C^2(\R^d);\\
        %\text{$\nabla f$ is locally Lipschitz continuous;}\\
        \inf f> -\infty.
    \end{cases}
\end{equation}
%\begin{align}
%\begin{cases}
%\text{$f\in C^2(\R^d)$}; \\
%%\text{$\nabla f$ is locally Lipschitz continuous;}\\
%%\mathrm{crit}(f)\neq\emptyset;\\
%\text{$\inf f >-\infty$.}
%\end{cases}
%\end{align} 

\begin{comment}
    Since the objective function is potentially non-convex, the problem \eqref{P} is NP-Hard. Therefore, until now, we cannot propose an algorithm that takes polynomial time to ensure convergence to the global minimizer of \eqref{P} in general. However, there are methods to ensure theoretical convergence to a local minimizer; in practice, there are techniques to (heuristically) converge to a global minimizer in some cases.
\end{comment}
    
    Since the objective function is potentially non-convex, the problem \eqref{P} is NP-Hard. However, there are tractable methods to ensure theoretical convergence to a critical point, or even to a local minimizer. In this regard, a fundamental dynamic to consider is the gradient flow system:

\begin{equation}
\begin{cases}\label{gf}\tag{GF}
\begin{aligned}
\dot{x}(t)+\nabla f(x(t))&= 0,\quad t>0;\\
x(0)&=x_0.
\end{aligned}
\end{cases}
\end{equation}

 For any bounded solution of \eqref{gf}, using LaSalle's invariance principle, we check that $\lim_{t\rightarrow +\infty}\nabla f(x(t))=0$. If $d=1$, any bounded solution of \eqref{gf} tends to a critical point. For $d\geq 2$ this becomes false in general, as shown in the counterexample by \cite{palis}. In order to avoid such behaviors, it is necessary to work with functions that present a certain structure. An assumption that will be central in our paper for the study of our dynamics and algorithms is that the function $f$ satisfies the Kurdyka-{\L}ojasiewicz (K{\L}) inequality \cite{loj1,loj2,kur}, which means, roughly speaking, that $f$ is sharp up to a reparametrization. The K{\L} inequality, including in its nonsmooth version, has been successfully used to analyze the asymptotic behavior of various types of dynamical systems \cite{chill, aris,boltearis,boltequasi} and algorithms \cite{bolte,frankel,noll,aujol,pol, absil, nonsm,tam, ipiano}\footnote{This list is by no means exhaustive.}. The importance of the K{\L} inequality comes from the fact that many problems encountered in optimization involve functions satisfying such an inequality, and it is often elementary to check that the latter is satisfied; \eg real semialgebraic/analytic functions \cite{loj1,loj2}, functions definable in an o-minimal structure and more generally tame functions \cite{kur,clarke}. 

\smallskip

The dynamic \eqref{gf} is known to yield a convergence rate of $\mathcal{O}(t^{-1})$ of the values in the convex setting (in fact even $o(t^{-1})$). 
Second-order inertial dynamical systems have been introduced to provably accelerate the convergence behavior in the convex case. They typically take the form
\begin{equation}\tag{$\mathrm{IGS}_{\gamma}$}\label{AVD}
\ddot{x}(t)+  \gamma (t)\dot{x}(t)  +\nabla f(x(t))=0,\quad t>t_0,
\end{equation}
where $t_0>0$, $\gamma: [t_0,+\infty[ \to \R_+$ is a time-dependent viscous damping coefficient. An abundant literature has been devoted to the study of the inertial dynamics \eqref{AVD}. The importance of working with a time-dependent viscous damping coefficient to obtain acceleration was stressed by several authors; see \eg \cite{AC1}. In particular, the case $\gamma(t)=\frac{\alpha}{t}$ was considered by Su, Boyd, and Cand\`es \cite{su}, who were the first to show the rate of convergence $\mathcal{O}(t^{-2})$ of the values in the convex setting for $\alpha \geq 3$, thus making the link with the accelerated gradient method of Nesterov \cite{1983}. For $\alpha > 3$, an even better rate of convergence with little-$o$ instead of big-$\mathcal O$ can be obtained together with global convergence of the trajectory; see \cite{cabot,faster1k2} and \cite{CD} for the discrete algorithmic case.

\smallskip

Another remarkable instance of \eqref{AVD} corresponds to the well-known Heavy Ball with Friction (HBF) method, where $\gamma(t)$ is a constant, first introduced (in its discrete and continuous form) by Polyak in \cite{speedingup}. When $f$ is strongly convex, it was shown that the trajectory converges exponentially with an optimal convergence rate if $\gamma$ is properly chosen as a function of the strong convexity modulus. The convex case was later studied in \cite{Alv} with a convergence rate on the values of only $\mathcal{O}(t^{-1})$. When $f$ is non-convex, HBF was investigated both for the continuous dynamics \cite{27,heavyb,goudou,apido} and discrete algorithms \cite{noncon, ipiano, peter}.  

\bigskip

However, because of the inertial aspects, %and the asymptotic vanishing viscous damping coefficient,
\eqref{AVD} may exhibit many small oscillations which are not desirable from an optimization point of view. To remedy this, a powerful tool consists in introducing into the dynamic a geometric damping driven by the Hessian of $f$. This gives the Inertial System with Explicit Hessian Damping which reads
\begin{equation}
\begin{cases}\label{eq:isehd}\tag{ISEHD}
\begin{aligned}
&\ddot{x}(t)+\gamma(t)\dot{x}(t)+\beta(t)\nabla^2 f(x(t))\dot{x}(t)+\nabla f(x(t))= 0,\quad t>t_0;\\
&x(t_0)=x_0, \quad \dot{x}(t_0)=v_0 ,
\end{aligned}
\end{cases}
\end{equation}
where $\gamma,\beta:[t_0,+\infty[\rightarrow\R_+$. \eqref{eq:isehd} was proposed in \cite{AABR} (see also \cite{APR2,9}).
%and Implicit and \eqref{eq:isihd}, respectively, follow the following differential equations starting at $t_0> 0$ with initial condition $x_0,v_0\in \R^d$:
%where $ \gamma$ and $\beta$ are, the already presented, damping parameters.
%The time discretization of this system has been studied by Attouch, Chbani, Fadili, and Riahi \cite{riahi}. It provides a rich family of first-order methods for minimizing $f$.
The second system we consider, inspired by~\cite{alecsa} (see also \cite{MJ} for a related autonomous system) is
\begin{equation}
\begin{cases}\label{eq:isihd}\tag{ISIHD}
\begin{aligned}
&\ddot{x}(t)+\gamma(t)\dot{x}(t)+\nabla f(x(t)+\beta(t)\dot{x}(t)) = 0,\quad t>t_0;\\
&x(t_0)=x_0, \quad \dot{x}(t_0)=v_0 .
\end{aligned}
\end{cases}
\end{equation}
\eqref{eq:isihd} stands for Inertial System with Implicit Hessian Damping.  The rationale behind the use of the term ``implicit'' comes from a Taylor expansion of the gradient term (as $t \to +\infty$ we expect $\dot{x}(t) \to 0$) around $x(t)$, which makes the Hessian damping appear indirectly in \eqref{eq:isihd}. Following the physical interpretation of these two ODEs, we call the non-negative parameters $\gamma$ and $\beta$ the viscous and geometric damping coefficients, respectively. The two ODEs \eqref{eq:isehd} and \eqref{eq:isihd} were found to have a smoothing effect on the energy error and oscillations \cite{alecsa,Muehlebach19,9}. Moreover, in \cite{hessianpert} they obtain fast convergence rates for the values for the two ODEs when $\gamma(t)=\frac{\alpha}{t}$ ($\alpha>3$) and $\beta(t)=\beta>0$. However, the previous results are exclusive to the convex case. Nevertheless, in \cite{casterainertial}, the authors analyzes \eqref{eq:isehd} with constant viscous and geometric damping in a non-convex setting, concluding that the bounded solution trajectories converges to a critical point of the objective under {\L}ojasiewicz inequality and giving sublinear convergence rates. Moreover, in \cite{casteraavoid} the author shows that the previous dynamic avoids strict saddle points when the objective is Morse. In these two works they propose a discretization called INNA, where the conditions on the stepsize are more stringent than ours but let them consider arbitrary values for the geometric damping. To the best of our knowledge, there is no further analysis of \eqref{eq:isehd} and \eqref{eq:isihd} when the objective is non-convex and definable. In this work, our goal is to fill this gap.

%%%%%%%%%%%%%%%%%%%%%%%%%%%%%%
\subsection{Contributions}
In this paper, we analyze the convergence properties of \eqref{eq:isehd} and \eqref{eq:isihd} when $f$ is non-convex, and when the time-dependent viscous damping coefficient $\gamma$ is non-vanishing and the geometric damping is constant, \ie $\beta(t)\equiv\beta\geq 0$. We will also propose appropriate discretizations of these dynamics and establish the corresponding convergence guarantees for the resulting discrete algorithms. More precisely, our main contributions can be summarized as follows:
\begin{itemize}
\item We provide a Lyapunov analysis of \eqref{eq:isehd} and \eqref{eq:isihd} and show convergence of the gradient to zero and convergence of the values. Moreover, assuming that the objective function is definable, we prove the convergence of the trajectory to a critical point (see Theorem~\ref{convisehd} and \ref{convisihd}). Furthermore, when $f$ is also Morse, using the center stable manifold theorem, we establish a generic convergence of the trajectory to a local minimum of $f$ (see Theorem~\ref{morseisehd} and \ref{morseisihd}).
\item We provide convergence rates of \eqref{eq:isehd} and \eqref{eq:isihd} and show that they depend on the desingularizing function of the Lyapunov energy (see Theorems \ref{cor} and \ref{cori}).
\item By appropriately discretizing \eqref{eq:isehd} and \eqref{eq:isihd}, we propose algorithmic schemes and prove the corresponding convergence properties, which are the counterparts of the ones shown for the continuous-time dynamics. In particular, assuming that the objective function is definable and that its gradient is globally Lipschitz, we show global convergence of the iterates of both schemes to a critical point of $f$. Furthermore, a generic strict saddle points (see Definition \ref{sspoint}) avoidance result will also be proved (see Theorem~\ref{convisehddisc} and \ref{convisihddisc}). 
\item Convergence rates for the discrete schemes will also be established under the {\L}ojasiewicz property where the convergence rate will be shown to depend on the exponent of the desingularizing function (see Theorem~\ref{rates} and \ref{ratesi}).
\item Inspired by the discretizations of \eqref{eq:isehd} and \eqref{eq:isihd}, we propose algorithmic schemes that use a backtracking-based stepsize to ensure convergence to a critical point of the objective, assuming only local Lipschitz continuity of the gradient rather than global (see Theorem ~\ref{convisehddiscgencoeff-back} and \ref{convisihddiscgencoeff-back}). The viscous and geometric damping terms depend on the stepsize as specified in the corresponding theorem statements.
\end{itemize}
We will report a few numerical experiments to support the above findings.

%%%%%%%%%%%%%%%%%%%%%%%%%%%%%%
\subsection{Relation to prior work}
\paragraph{Continuous-time dynamics.}
There is abundant literature regarding the dynamics \eqref{eq:isehd} and \eqref{eq:isihd}, either in the exact case or with deterministic errors, but solely in the convex case; see \cite{alecsa, hessianpert,27,35,8,11,20,34,37,19,9,10}). Nevertheless, to the best of our knowledge, except for two papers that we will discuss below, the non-convex setting is barely considered in the literature..

\paragraph{Heavy ball-like methods.}
As mentioned before, the HBF method, was first introduced by Polyak in \cite{speedingup} where linear convergence was shown for $f$ strongly convex. Convergence rates of HBF taking into account the geometry of the objective can be found in \cite{aujol} for the convex case and \cite{peter} for the non-convex case.

\smallskip

In \cite{haraux}, the authors study the system 
$$\ddot{x}(t)+G(\dot{x}(t))+\nabla f(x(t))=0,$$
where $G:\R^d\rightarrow\R^d$ is such that $\langle G(v),v\rangle\geq c\Vert v\Vert^2$ and $\Vert G(v)\Vert \leq C\Vert v\Vert$, $\forall v\in\R^d$, for some $0<c\leq C$. They show that when $f$ is real analytic (hence verifies the {\L}ojasiewicz inequality), the trajectory converges to a critical point. To put it in our terms, the analysis is equivalent to study \eqref{AVD} in the case where there exists $c,C>0$ such that $0<c\leq \gamma(t)\leq C$ for every $t  \geq 0$ (see assumption \eqref{gamma}), letting us to conclude as in \cite{haraux}. We will extend this result to the systems \eqref{eq:isehd} and \eqref{eq:isihd} which will necessitate new arguments.

\smallskip

HBF was also studied in the non-convex case by \cite{boltequasi} where it was shown that it can be viewed as a quasi-gradient system. They also proved that a desingularizing function of the objective desingularizes the total energy and its deformed versions. They used this to establish the convergence of the trajectory for instance in the definable case. Global convergence of the HBF trajectory was also shown in \cite{heavyb} when $f$ is a Morse function \footnote{It turns out that a Morse function satisfies the {\L}ojasiewicz inequality with exponent $1/2$; see \cite{AttouchAlternating10}.}. 
\paragraph{Inertial algorithms.}
%\todo{Add a literature overview including my work, that of Peter, that of Bot, that of Goudou and Munier, and others.}.

In the literature regarding algorithmic schemes that include inertia in the non-convex setting, we can mention: \cite{jingwei} which proposes a multi-step inertial algorithm using Forward-Backward for non-convex optimization. In \cite{goudou} they study quasiconvex functions and use an implicit discretization of HBF to derive a proximal algorithm. In \cite{ipiano}, they introduce an inertial Forward-Backward splitting method, called iPiano-a generalization of HBF, they show convergence of the values and a convergence rate of the algorithm. Besides, in \cite{peter}, the author presents local convergence results for iPiano. Then, in \cite{blockcoord, anabstract} several abstract convergence theorems are presented to apply them to different versions of iPiano. In \cite{bot1}, the authors propose a Tseng-type inertial algorithm, and in \cite{proximal,bot2,fba,gip}, they introduce Forward-Backward inertial algorithms for nonconvex and nonsmooth optimization problems. In the previous works, they assume K{\L}-like inequality to conclude with some type of convergence (weak or strong) towards a critical point of the objective. 

\smallskip

In this sense, the closest work to ours is \cite{casterainertial}, where they study \eqref{eq:isehd} when the viscous and geometric dampings are positive constants, and a discretization of this dynamic called INNA. Assuming that the objective function is semi-algebraic and the solution trajectories (resp. iterations) are bounded, their main results are that: both the dynamic and the discretization converge to critical points of the objective (using approximation theory), and that the continuous dynamic exhibits a sublinear convergence rate. Our work extends beyond this, it considers not only the explicit version \eqref{eq:isehd} but also the implicit one \eqref{eq:isihd},  and also allows variable (non-vanishing) viscous damping. Even though the choice of the value of the geometric damping is less general than the one presented in \cite{casterainertial}, we propose a less stringent discretization regarding the stepsize, and conclude with the convergence of the dynamic to a critical point under general K{\L} property on the objective, which includes semi-algebraic functions as a special case. Additionally, we show convergence rates under K{\L} inequality. Moreover, we propose new algorithms and an independent analysis for the discrete setting, showing convergence of the proposed algorithms to critical points under K{\L}-inequality, and we also exhibit convergence rates of the proposed algorithms under {\L}ojasiewicz inequality.

\paragraph{Trap avoidance.}
In the non-convex setting, generic strict saddle point avoidance of descent-like algorithms has been studied by several authors building on the (center) stable manifold theorem \cite[Theorem~III.7]{shub_global_1987} which finds its roots in the work of Poincar\'e. Genericity is in general either with respect to initialization or with respect to random perturbation. Note that genericity results for quite general systems, even in infinite dimensional spaces, is an important topic in the dynamical system theory; see \eg \cite{Brunovsky97} and references therein.

\smallskip

First-order descent methods can circumvent strict saddle points provided that they are augmented with unbiased noise whose variance is sufficiently large in each direction. Here, the seminal works of \cite{BrandiereDuflo96} and \cite{Pemantle90} allow to establish that the stochastic gradient descent (and more generally the Robbins-Monro stochastic approximation algorithm) avoids strict saddle points almost surely. Those results were extended to the discrete version of HBF by \cite{agd} who showed that perturbation allows to escape strict saddle points, and it does so faster than gradient descent. In \cite{gadat2} and \cite{hbnoise}, the authors analyze a stochastic version of HBF (in continuous-time), showing convergence towards a local minimum under different conditions on the noise. In this paper, we only study genericity of trap avoidance with respect to initialization.

\smallskip

Recently, there has been active research on how gradient-type descent algorithms escape strict saddle points generically on initialization; see \eg \cite{Lee16, localminima, panageas, cedric} and references therein. In \cite{goudou}, the authors were concerned with HBF and showed that if the objective function is $C^2$, coercive and Morse, then generically on initialization, the solution trajectory converges to a local minimum of $f$. A similar result is also stated in \cite{heavyb}. The algorithmic counterpart of this result was established in \cite{wright} who proved that the discrete version of HBF escapes strict saddles points for almost all initializations. 

\smallskip

We have to mention that the closest work in this regard is \cite{casteraavoid}, where the author studies \eqref{eq:isehd} when the viscous and geometric dampings are positive constants and the discretization INNA proposed in \cite{casterainertial}. Using the continuous and discrete version of the global stable manifold theorem, the author shows that both the continuous-time dynamic (when the objective is Morse) and the INNA algorithm almost always avoids strict saddle points. 

\smallskip

Our goal in this paper is to establish the same kind of results for \eqref{eq:isehd} and \eqref{eq:isihd} with variable, non-vanishing viscous damping and constant geometric damping, as well as for the proposed algorithms \eqref{isehd-disc} and \eqref{isihd-disc}. We would like to point out that the proof for the continuous-time case will necessitate (as in \cite{casteraavoid}) a more stringent assumption on the class of functions, for instance, that $f$ is Morse, while this is not necessary for the discrete algorithms.

\paragraph{Locally Lipschitz gradient objectives}
To the best of our knowledge, our work is the only one that handles the case where the objective has a locally Lipschitz continuous gradient in discrete algorithms with inertia and Hessian-driven damping—a crucial aspect, as many applications in machine learning do not satisfy global Lipschitz continuity of the objective gradient.

\paragraph{Overview.}
Table~\ref{tab:compare_methods} provides a high-level summary of our contributions and comparison to closely related works. Each entry indicates whether the corresponding work addresses convergence rates and/or trap avoidance for different methods. A checkmark (\cmark) indicates the aspect is addressed, while a cross (\xmark) indicates it is not. 

\begin{table}[H]\label{cont}
\begin{tabular}{|c!{\vrule width 1.5pt}cc|cc|cc|}
\hline 
\multicolumn{7}{|c|}{Continuous-time dynamics} \\\hline
 & \multicolumn{2}{c|}{Heavy Ball}    & \multicolumn{2}{c|}{ISEHD}    & \multicolumn{2}{c|}{ISIHD}    \\ 
\hline
  &  \multicolumn{1}{c|}{Conv. \& rates} & Trap avoid. & \multicolumn{1}{c|}{Conv. \& rates} & Trap avoid. & \multicolumn{1}{c|}{Conv. \& rates} &  Trap avoid. \\ \noalign{\hrule height 1.5pt} \hline
\cite{haraux} & \multicolumn{1}{c|}{\cmark} & \xmark & \multicolumn{1}{c|}{\xmark} & \xmark & \multicolumn{1}{c|}{\xmark} & \xmark \\ \hline
\cite{heavyb,goudou} & \multicolumn{1}{c|}{\cmark} & \cmark & \multicolumn{1}{c|}{\xmark} & \xmark & \multicolumn{1}{c|}{\xmark} & \xmark \\ \hline
\cite{casterainertial}& \multicolumn{1}{c|}{\cmark} & \xmark &  \multicolumn{1}{c|}{\cmark} & \xmark & \multicolumn{1}{c|}{\xmark} & \xmark \\ \hline
 \cite{casteraavoid}& \multicolumn{1}{c|}{\xmark} & \cmark &  \multicolumn{1}{c|}{\xmark} & \cmark & \multicolumn{1}{c|}{\xmark} & \xmark \\ \hline
 Ours & \multicolumn{1}{c|}{\cmark} & \cmark & \multicolumn{1}{c|}{\cmark} & \cmark & \multicolumn{1}{c|}{\cmark} & \cmark \\ \hline
\hline 
\multicolumn{7}{|c|}{Algorithms} \\\hline
& \multicolumn{2}{c|}{Heavy Ball}    & \multicolumn{2}{c|}{ISEHD}    & \multicolumn{2}{c|}{ISIHD}    \\ 
\hline
  &  \multicolumn{1}{c|}{Conv. \& rates} & Trap avoid. & \multicolumn{1}{c|}{Conv. \& rates} & Trap avoid. & \multicolumn{1}{c|}{Conv. \& rates} &  Trap avoid. \\ \noalign{\hrule height 1.5pt}
 \hline
 \cite{agd,wright} & \multicolumn{1}{c|}{\cmark} & \cmark & \multicolumn{1}{c|}{\xmark} & \xmark & \multicolumn{1}{c|}{\xmark} & \xmark \\ \hline
\cite{fba,gip} & \multicolumn{1}{c|}{\cmark} & \xmark & \multicolumn{1}{c|}{\xmark} & \xmark & \multicolumn{1}{c|}{\cmark} & \xmark \\ \hline
\cite{casterainertial}& \multicolumn{1}{c|}{\cmark} & \xmark &  \multicolumn{1}{c|}{\cmark} & \xmark & \multicolumn{1}{c|}{\xmark} & \xmark \\ \hline
 \cite{casteraavoid}& \multicolumn{1}{c|}{\xmark} & \cmark &  \multicolumn{1}{c|}{\xmark} & \cmark & \multicolumn{1}{c|}{\xmark} & \xmark \\ \hline
 Ours & \multicolumn{1}{c|}{\cmark} & \cmark & \multicolumn{1}{c|}{\cmark} & \cmark & \multicolumn{1}{c|}{\cmark} & \cmark \\ \hline
\end{tabular}
\caption{Overview of comparison between our contributions and related work.}
    \label{tab:compare_methods}
\end{table}

We note that \cite{fba,gip} and \cite{casterainertial,casteraavoid} analyze different algorithms from ours as they stem from different discretizations of the underlying continuous-time dynamics. %but follow the same underlying approach: applying finite differences to instances of \eqref{eq:isihd} and \eqref{eq:isehd}, respectively.

%Furthermore, if $f\in C^2(\R^d)$ is a Morse function, the global stable manifold theorem ensures us that for almost all initializations $x_0\in\R^d$, $x(t)$ converges (as $t\rightarrow +\infty$) to a local minimum of $f$. This is also true in the discrete setting (gradient descent algorithm with constant stepsize) without assuming the Morse condition but rather that all the saddle points (including local maximums) are strict, $f\in C^2(\R^d)$, $\nabla f$ is $L-$Lipschitz, and the stepsize is less than $\frac{1}{L}$ (see \cite[Corollary 2]{localminima}). Moreover, in \cite{panageas} they show avoidance of strict saddle points, replacing the global Lipschitianity condition by a uniform bound of the Hessian over a convex domain, a forward invariant condition, and a proper tuning of the stepsize. Furthermore, in a recent work (see \cite{cedric}), the author maintains the hypothesis that all saddle points are strict, and gets rid of the global Lipschitianity condition of the gradient, replacing it with the function being definable and with bounded continuous gradient trajectories, a proper choice of the stepsize lets him conclude with the convergence to a local minimum.

%\todo{cite jendoubi, bolte jendoubi, isehd papers, heavy ball like methods...}

%%%%%%%%%%%%%%%%%%%%%%%%%%%%%%
\subsection{Organization of the paper}
Section~\ref{sec:notation} introduces notations and recalls some preliminaries that are essential to our exposition. Sections~\ref{sec:exp} and~\ref{sec:imp} include the main contributions of our paper, establishing convergence of the trajectory and of the iterates under K{\L} inequality, convergence rates and trap avoidance results. Section~\ref{sec:backtracking} presents algorithmic schemes based on backtracking techniques, which allow us to obtain convergence results under the much weaker assumption of local Lipschitz continuity of the gradient of the objective. Section~\ref{sec:num} is devoted to the numerical experiments.

%Additional remarks and properties can be found in Appendix~\ref{aux}.

%%%%%%%%%%%%%%%%%%%%%%%%%%%%%%%%%%%%%%%%%%%%%%%%%%%%%%
\section{Notation and preliminaries}\label{sec:notation}
%%%%%%%%%%%%%%%%%%%%%%%%%%%%%%%%%%%%%%%%%%%%%%%%%%%%%%
We denote by $\R_+$ the set $[0,+\infty[$. Moreover, we denote $\R_+^*$ and $\N^*$ to refer to $\R_+\setminus\{0\}$ and $\N\setminus\{0\}$, respectively. The finite-dimensional space $\R^d$ $(d\in \N^*)$ is endowed with the canonical scalar product $\langle\cdot,\cdot\rangle$ whose norm is denoted by $\Vert\cdot\Vert$. We denote $\R^{n\times d}$ ($n,d\in \N^*$) the space of real matrices of dimension $n\times d$. We denote by $C^{n}(\R^d)$ the class of $n$-times continuously differentiable functions on $\R^d$. For a function $g:\R^d\rightarrow\R$ and $a,b\in\R$ such that $a<b$, we denote $[a<f<b]$ to the sublevel set $\{x\in\R^d:a<f(x)<b\}$.\linebreak
For a differentiable function $g:\R^d\rightarrow\R$, we will denote its gradient as $\nabla g$, the set of its critical points as: $$\mathrm{crit}(g)\eqdef \{u:\nabla g(u)=0\},$$
and when $g$ is twice differentiable, we will denote its Hessian as $\nabla^2 g$. For a differentiable function $G:\R^d\rightarrow \R^n$ we will denote its Jacobian matrix as $J_G\in\R^{n\times d}$. Let $A\in \R^{d\times d}$, then we denote $\lambda_i(A)\in \mathbb{C}$ the $i-$th eigenvalue of $A$, when $A$ is symmetric then the eigenvalues are real and we denote $\lambda_{\min}(A), \lambda_{\max}(A)$ to be the minimum and maximum eigenvalue of $A$, respectively. If every eigenvalue of $A$ is positive (resp. negative), we will say that $A$ is positive (resp. negative) definite. We denote by $I_d$ the identity matrix of dimensions $d\times d$ and $0_{n\times d}$ the null matrix of dimensions $n\times d$, respectively. The following lemma is essential to some arguments presented in this work:

 \begin{lemma}\cite[Theorem 3]{bloc}\label{block}
      Consider $A,B,C,D \in \R^{d\times d}$ and such that $AC=CA$. Then $$\det\left(\begin{pmatrix}
          A & B\\ C & D
      \end{pmatrix}\right)=\det(AD-CB).$$
  \end{lemma} 

And the following definitions will be important throughout the paper:

\begin{definition}[Local extrema and saddle points]
    Consider a function $f\in C^2(\R^d)$. We will say that $\hat{x}$ is a local minimum (resp. maximum) of $f$ if $\hat{x}\in\mathrm{crit}(f)$, $\nabla^2 f(\hat{x})$ is positive (resp. negative) definite. If $\hat{x}$ is a critical point that is neither a local minimum nor a local maximum, we will say that $\hat{x}$ is a saddle point of $f$.
\end{definition}

\begin{definition}[Strict saddle point]\label{sspoint}
    Consider a function $f\in C^2(\R^d)$, we will say that $\hat{x}$ is a strict saddle point of $f$ if $\hat{x}\in\mathrm{crit}(f)$ and $\lambda_{\min}(\nabla^2 f(\hat{x}))<0$.
\end{definition}
\begin{remark}
    According to this definition, local maximum points are strict saddles.
\end{remark}

\begin{definition}[Strict saddle property]\label{ssp}
A function $f\in C^2(\R^d)$ will satisfy the strict saddle property if every critical point is either a local minimum or a strict saddle.   
\end{definition}

This property is a reasonable assumption for smooth minimization. In practice, it holds for specific problems of interest, such as low-rank matrix recovery and phase retrieval. Moreover, as a consequence of Sard's theorem, for a full measure set of linear perturbations of a function $f$, the linearly perturbed function satisfies the strict saddle property. Consequently, in this sense, the strict saddle property holds generically in smooth optimization. We also refer to the discussion in \cite[Conclusion]{localminima}.

\begin{definition}[Morse function]
A function $f\in C^2(\R^d)$ will be Morse if it satisfies the following conditions:
\begin{enumerate}[label=(\roman*)]
    \item For each critical point $\hat{x}$, $\nabla^2 f(\hat{x})$ is nonsingular.
    \item There exists a nonempty set $I\subset\N$  and $(\hat{x}_k)_{k\in I}$ such that $\mathrm{crit}(f)=\bigcup_{k\in I} \{\hat{x}_k\}$. 
\end{enumerate}
\end{definition}

\begin{remark}\label{rem:strictsaddlemorse}
By definition, a Morse function satisfies the strict saddle property.
\end{remark}

\begin{remark}
    Morse functions can be shown to be generic in the Baire sense in the space of $C^2$ functions; see \cite{aubinekeland}.
\end{remark}

\begin{definition}[Desingularizing function]
    For $\eta>0$, we consider 
    $$\kappa(0,\eta)\eqdef \{\psi: C^0([0,\eta[)\cap C^1(]0,\eta[)\rightarrow\R_+, \psi'>0 \text{ on } ]0,\eta[, \psi(0)=0, \text{ and } \psi \text{ concave} \}.$$
\end{definition}
\begin{remark}
   The concavity property of the functions in $\kappa(0,\eta)$ is only required in the discrete setting.
\end{remark} 

\begin{definition}
If $f:\R^d\rightarrow\R$ is differentiable and satisfies the K{\L} inequality at $\bar{x}\in\R^d$, then there exists $r,\eta>0$ and $\psi\in \kappa(0,\eta)$, such that 
\begin{equation}\label{kl0}
\psi'(f(x)-f(\bar{x}))\Vert\nabla f(x)-\nabla f(\bar{x})\Vert\geq 1,\quad\forall x\in B(\bar{x},r)\cap [f(\bar{x})<f<f(\bar{x})+\eta].
\end{equation}
\end{definition}

\begin{definition}
A function $f:\R^d\rightarrow\R$ will satisfy the {\L}ojasiewicz inequality with exponent $q \in ]0,1]$ if $f$ satisfies the K{\L} inequality with $\psi(s)=c_0 s^{1-q}$ for some $c_0>0$.
\end{definition}  

It remains now to identify a broad class of functions $f$ that verifies the K{\L} inequality. A rich family is provided by semi-algebraic functions, \ie, functions whose graph is defined by some Boolean combination of real polynomial equations and inequalities~\cite{coste2002intro}. Such functions satisfy the {\L}ojasiewicz property with $q \in [0, 1[ \cap \mathbb{Q}$; see \cite{loj1,loj2}. An even more general family is that of definable functions on an o-minimal structure over $\R$, which corresponds in some sense to an axiomatization of some of the prominent geometrical and stability properties of semi-algebraic geometry~\cite{vandenDriesMiller96,omin}. An important result by Kurdyka in \cite{kur} showed that definable functions satisfy the K{\L} inequality at every $\bar{x}\in\R^d$.

%The following result is due to {\L}ojasiewicz in its real-analytic version (see \cite{loj1,loj2}), and it was generalized to $o-$minimal structures and simplified by Kurdyka in \cite{kur}.

%\begin{proposition}
%    Let $\mathbb{O}$ be an $o-$minimal structure and let $f\in C^1(\R^d)$ be a  definable function, then $f$ satisfies the K{\L} inequality at every $\bar{x}\in\R^d$.
%\end{proposition}
%\begin{remark}
%    We refer to \cite{omin} for a comprehensive account on $o-$minimal structures.
%\end{remark}

\begin{remark}\label{rem}
Morse functions verify the {\L}ojasiewicz inequality with exponent $q=1/2$; see \cite{AttouchAlternating10}.
\end{remark}
%\todo{Add a remark that a Morse function verifies the {\L}ojasiewicz inequality with exponent $q=1/2$; see \cite{AttouchAlternating10}.}

%%%%%%%%%%%%%%%%%%%%%%%%%%%%%%%%%%%%%%%%%%%%%%%%%%%%%%
\section{Inertial System with Explicit Hessian Damping}\label{sec:exp}
Throughout the paper we will consider \eqref{eq:isehd} and \eqref{eq:isihd} with $t_0=0$. Also we will assume that the viscous damping $\gamma:\R_+\rightarrow\R_+$ is such that $\exists c, C > 0$, $c\leq C$, and  \begin{equation}\label{gamma}\tag{$\mathrm{H_{\gamma}}$}
     c\leq \gamma(t)\leq C , \quad\forall t  \geq 0.
\end{equation}
Moreover, throughout this work, we consider a constant geometric damping, \ie $\beta(t) \equiv \beta > 0$. 
%%%%%%%%%%%%%%%%%%%%%%%%%%%%%%%%%%%%%%%%%%%%%%%%%%%%%%
%We consider $f:\R^d\rightarrow\R$ such that:
%\begin{equation}\label{H0}\tag{$\mathrm{H_0}$}
%    \begin{cases}
%        f\in C^2(\R^d);\\
%        %\mathrm{crit}(f)\neq\emptyset;\\
%        \inf f> -\infty.
%    \end{cases}
%\end{equation}
%\todo{$f$ being $C^2$, and thus $\nabla^2 f$ is continuous, means it is bounded on bounded sets. This implies that $\nabla f$ is Lipschitz continuous on bounded sets. Thus the second assumption is superfluous.}

%%%%%%%%%%%%%%%%%%%%%%%%%%%%%%
\subsection{Continuous-time dynamics}

Let us consider \eqref{eq:isehd}, as in \cite{hessianpert}, we will say that $x:\R_+\rightarrow\R^d$ is a solution trajectory of \eqref{eq:isehd} with initial conditions $x(0)=x_0, \dot{x}(0)=v_0$, if and only if, $x\in C^2(\R_+;\R^d)$ and there exists $y\in C^1(\R_+;\R^d)$ such that $(x,y)$ satisfies: 
    \begin{equation}\label{refor}
        \begin{cases}
            \dot{x}(t)+\beta\nabla f(x(t))-\left(\frac{1}{\beta}-\gamma(t)\right)x(t)+\frac{1}{\beta}y(t)&=0,\\
            \dot{y}(t)-\left(\frac{1}{\beta}-\gamma(t)-\beta\gamma'(t)\right)x(t)+\frac{1}{\beta}y(t)&=0,
        \end{cases}
    \end{equation}
    with initial conditions $x(0)=x_0, y(0)=y_0\eqdef -\beta (v_0+\beta\nabla f(x_0))+(1-\beta\gamma(0))x_0$.  \\
    
%For $\beta>0$ we consider the Inertial System with Explicit Hessian Damping: \begin{equation}\label{ISEHD}\tag{$\mathrm{ISEHD}$}
%    \begin{aligned}
%    \begin{cases}
%    \ddot{x}(t)+\gamma(t)\dot{x}(t)+\beta\frac{d}{dt}\nabla f(x(t))+\nabla f(x(t))=0.\\
%    x(0)=x_0, \dot{x}(0)=v_0.
%    \end{cases}
%    \end{aligned}
%\end{equation}

%%%%%%%%%%%%%%%%%%%%
\subsubsection{Global convergence of the trajectory}

Our first main result is the following theorem.
\begin{theorem}\label{convisehd}
Assume that $0<\beta<\frac{2c}{C^2}$, $f:\R^d\rightarrow\R$ satisfies \eqref{H0}, and $\gamma\in C^1(\R_+;\R_+)$ obeys \eqref{gamma}.
     %\begin{equation}\label{gamma1}\tag{$\mathrm{H}_{\gamma}'$}
    %\gamma \in C^1(\R_+;\R_+) .
    %\end{equation}  
    
    Consider \eqref{eq:isehd} in this setting, then the following holds:    
    \begin{enumerate}[label=(\roman*)]
        \item \label{uno} There exists a global solution trajectory $x:\R_+\rightarrow\R^d$ of \eqref{eq:isehd}.
        \item \label{dos} We have that $\nabla f\circ x\in\Lp^2(\R_+;\R^d)$, and $\dot{x}\in\Lp^2(\R_+;\R^d)$.
        \item \label{tres} If we suppose that the solution trajectory $x$ is bounded over $\R_+$, then $$\lim_{t\rightarrow +\infty}\Vert\nabla f(x(t))\Vert=\lim_{t\rightarrow +\infty}\Vert \dot{x}(t)\Vert=0,$$ and $\lim_{t\rightarrow +\infty} f(x(t))$ exists. 
    \item \label{convx} In addition to \ref{tres}, if we also assume that $f$ is definable, then $\dot{x}\in\Lp^1(\R_+;\R^d)$ and $x(t)$ converges (as $t\rightarrow +\infty$) to a critical point of $f$.
      \end{enumerate}
\end{theorem}

\begin{remark}
The boundedness assumption in assertion \ref{dos} can be dropped if $\nabla f$ is supposed to be globally Lipschitz continuous.
\end{remark}

\begin{proof}
\begin{enumerate}[label=(\roman*)]
        \item 
    We will start by showing the existence of a solution. Setting $Z=(x,y)$, \eqref{refor} can be equivalently written as
    \begin{equation}\label{refor1}
        \dot{Z}(t)+\nabla\mathcal{G}(Z(t))+\mathcal{D}(t,Z(t))=0, \quad Z(0)=(x_0,y_0),
    \end{equation}
    where $\mathcal{G}(Z):\R^d\times \R^d\rightarrow\R$ is the function defined by $\mathcal{G}(Z)=\beta f(x)$ and the time-dependent operator $\mathcal{D}:\R_+\times\R^d\times\R^d\rightarrow\R^d\times\R^d$ is given by:
    \[
    \mathcal{D}(t,Z) \eqdef \left(-\left(\frac{1}{\beta}-\gamma(t)\right)x+\frac{1}{\beta}y,-\left(\frac{1}{\beta}-\gamma(t)-\beta\gamma'(t)\right)x+\frac{1}{\beta}y\right).
    \]
    
Since the map $(t,Z)\mapsto\nabla \mathcal{G}(Z)+\mathcal{D}(t,Z)$ is continuous in the first variable and locally Lipschitz in the second (by hypothesis  \eqref{H0} and the assumptions on $\gamma$), by the classical Cauchy-Lipschitz theorem, we have that there exists $T_{\max} > 0$ and a unique maximal solution of \eqref{refor1} denoted $Z\in C^1([0,T_{\max}[;\R^d\times\R^d)$. Consequently, there exists a unique maximal solution of \eqref{eq:isehd} $x\in C^2([0,T_{\max}[;\R^d)$. \\  
    %\todo{Write the details by lifting. Say the solution classical and $C^2$ which allows to take derivatives. See our paper on \eqref{eq:isehd}.}\\

    Let us consider the \tcb{energy} function $V:[0,T_{\max}[\rightarrow\R$ defined by $$V(t)=f(x(t))+\frac{1}{2}\Vert \dot{x}(t)+\beta\nabla f(x(t))\Vert^2.$$ \tcb{We will prove that it is indeed a Lyapunov function for \eqref{eq:isehd}.} We see that 
    \begin{align*}
        {V}'(t)&=\langle \nabla f(x(t)),\dot{x}(t)\rangle+\langle \ddot{x}(t)+\beta\frac{d}{dt}\nabla f(x(t)),\dot{x}(t)+\beta\nabla f(x(t))\rangle\\
        &=\langle \nabla f(x(t)),\dot{x}(t)\rangle+\langle -\gamma(t)\dot{x}(t)-\nabla f(x(t)),\dot{x}(t)+\beta\nabla f(x(t))\rangle\\
            &=-\langle\gamma(t)\dot{x}(t),\dot{x}(t)\rangle-\beta\langle \gamma(t)\dot{x}(t) ,\nabla f(x(t))\rangle-\beta\Vert\nabla f(x(t))\Vert^2\\
            &\leq -c\Vert \dot{x}(t)\Vert^2+\frac{\beta^2\Vert \gamma(t)\dot{x}(t)\Vert^2}{2\varepsilon}+\frac{\varepsilon\Vert \nabla f(x(t))\Vert^2}{2}\\
            &\leq -c\Vert \dot{x}(t)\Vert^2+\frac{\beta^2C^2\Vert \dot{x}(t)\Vert^2}{2\varepsilon}+\frac{\varepsilon\Vert \nabla f(x(t))\Vert^2}{2} ,
    \end{align*}
    where the last bound is due to Young's inequality with $\varepsilon>0$. Now let $\varepsilon=\frac{\beta^2C^2}{c}$, then 
    \begin{align}\label{eq:lyapdotVineq}
        {V}'(t)\leq -\frac{c}{2}\Vert \dot{x}(t)\Vert^2 - \beta\left(1-\frac{\beta C^2}{2c}\right)\Vert \nabla f(x(t))\Vert^2\leq -\delta_1(\Vert \dot{x}(t)\Vert^2+\Vert \nabla f(x(t))\Vert^2).
    \end{align}
    For $\delta_1\eqdef \min\pa{\frac{c}{2},\beta\left(1-\frac{\beta C^2}{2c}\right)}>0$.\\

\tcb{We will now show that the maximal solution $Z$ of \eqref{refor1} is actually global. For this, we use a standard argument and argue by contradiction assuming that $T_{\max} < +\infty$. It is sufficient to prove that $x$ and $y$ have a limit as $t \to T_{\max}$, and local existence will contradict the maximality of $T_{\max}$. Integrating inequality \eqref{eq:lyapdotVineq}, we obtain $\dot{x}\in\Lp^2([0,T_{\max}[;\R^d)$ and $\nabla f\circ x\in\Lp^2([0,T_{\max}[;\R^d)$, hence implying that $\dot{x}\in\Lp^1([0, \tcb{T_{\max}}[;\R^d)$ and $\nabla f\circ x\in\Lp^1([0,T_{\max}[;\R^d)$, which in turn entails that $(x(t))_{t\in [0,T_{\max}[}$ satisfies the Cauchy property whence we get that $\lim_{t\rightarrow T_{\max}} x(t)$ exists}. Besides, by the first equation of \eqref{refor}, we have that $\lim_{t\rightarrow T_{\max}} y(t)$ exists if $\lim_{t\rightarrow T_{\max}} x(t), \lim_{t\rightarrow T_{\max}} \nabla f(x(t))$ and $\lim_{t\rightarrow T_{\max}} \dot{x}(t)$ exist. We have already that the first two limits exist by continuity of $\nabla f$, and thus we just have to check that $\lim_{t\rightarrow T_{\max}} \dot{x}(t)$ exists. A sufficient condition would be to prove that $\ddot{x}\in\Lp^1([0,T_{\max}[;\R^d)$. By \eqref{eq:isehd} this will hold if $\dot{x},\nabla f\circ x,(\nabla^2 f \circ x)\dot{x}$ are in $\Lp^1([0,T_{\max}[;\R^d)$. We have already checked that the first two terms are in $\Lp^1([0,T_{\max}[;\R^d)$. To conclude, it remains to check that $\nabla^2 f\circ x\in \Lp^{\infty}([0,T_{\max}[;\R^d)$ and this is true since $\nabla^2 f$ is continuous, $x$ is continuous on $[0,T_{\max}]$, and the latter is compact.  Consequently, the solution $Z$ of \eqref{refor1} is global, thus the solution $x$ of \eqref{eq:isehd} is also global.\\ 

\item Integrating \eqref{eq:lyapdotVineq} and using that $V$ is well-defined for every $t > 0$ and is bounded from below, we deduce that \tcb{$\dot{x}\in\Lp^2(\R_+;\R^d)$, and $\nabla f\circ x\in\Lp^2(\R_+;\R^d)$.}

\item 
\tcb{We recall that we are assuming that $\sup_{t  \geq 0}\Vert x(t)\Vert<+\infty$ and $f \in C^2(\R^d)$, hence 
\[
\sup_{t  \geq 0}\Vert\nabla^2 f(x(t))\Vert<+\infty.
\]
In turn, $\nabla f$ is Lipschitz continuous on bounded sets. Moreover, as $\dot{x}\in\Lp^2(\R_+;\R^d)$ and is continuous, then $\dot{x}\in\Lp^{\infty}(\R_+;\R^d)$. The last two facts imply that $t\mapsto \tcb{\nabla f(x(t))}$ is uniformly continuous. In fact, for every $t,s \geq 0$, we have
\[
\norm{\nabla f(x(t)) - \nabla f(x(s))} \leq \sup_{\tau  \geq 0}\norm{\nabla^2 f(x(\tau))} \norm{\dot{x}(\tau)}|t-s|  .
\]
This combined with $\nabla f\circ x\in\Lp^2(\R_+;\R^d)$ yields 
\[
\lim_{t\rightarrow +\infty}\norm{\nabla f(x(t))}=0 .
\]
We also have that $\frac{d}{dt}\nabla f(x(t))=\nabla^2 f(x(t))\dot{x}(t)$, and thus $(\nabla^2 f\circ x)\dot{x}\in \Lp^{\infty}(\R_+;\R^d)$. We also have $\nabla f\circ x \in \Lp^{\infty}(\R_+;\R^d)$ by continuity of $\nabla f$ and boundedness of $x$. It then follows from \eqref{eq:isehd} that $\ddot{x}\in \Lp^{\infty}(\R_+;\R^d)$. This implies that 
\[
\norm{\dot{x}(t) - \dot{x}(s)} \leq \sup_{\tau  \geq 0}\norm{\ddot{x}(\tau)}|t-s| ,
\]
meaning that $t\mapsto \dot{x}(t)$ is uniformly continuous. Recalling that $\dot{x}\in\Lp^2(\R_+;\R^d)$ gives that 
\[
\lim_{t\rightarrow +\infty} \norm{\dot{x}(t)}=0.
\]
}

\tcb{We have from \eqref{eq:lyapdotVineq} that $V$ is non-increasing. Since it is bounded from below, it has a limit, \ie $\lim_{t\rightarrow +\infty}V(t)$ exists and we will denote this limit by $\tilde{L}$. Recall from the definition of $V$ that $$f(x(t))=V(t)-\frac{1}{2}\Vert \dot{x}(t)+\beta\nabla f(x(t))\Vert^2.$$ Using the above three limits we get $$\lim_{t\rightarrow +\infty }f(x(t))=\lim_{t\rightarrow +\infty }V(t) = \tilde{L}.$$
}

\item 
\tcb{From boundedness of \tcb{$(x(t))_{t  \geq 0}$}, by a Lyapunov argument (see \eg \cite[Proposition 4.1]{heavyb}, \cite{haraux0}), the set of its cluster points $\tcb{\mathfrak{C}(x(\cdot))}$ satisfies:
\begin{equation}\label{xkb}
\begin{aligned}
    \begin{cases}
        \mathfrak{C}(x(\cdot)) \subseteq \mathrm{crit}(f);\\
        \mathfrak{C}(x(\cdot)) \text{ is non-empty, compact and connected};\\
        \text{$f$ is constant on } \mathfrak{C}(x(\cdot)).
    \end{cases}     
\end{aligned}
\end{equation}
}
We consider the function
\begin{equation}\label{eq:energE}
E: (x,v,w) \in \R^{3d} \mapsto f(x)+\frac{1}{2}\Vert v+w\Vert^2 .
\end{equation}
Since $f$ is definable, so is $E$ as the sum of a definable function and an algebraic one. \tcb{Therefore, $E$ satisfies the \tcb{K{\L}} inequality \cite{kur}. Let $\mathfrak{C}_1=\mathfrak{C}(x(\cdot)) \times \{0_d\}\times \{0_d\}$. Observe that $E$ takes the constant value $\tilde{L}$ on $\mathfrak{C}_1$ and $\mathfrak{C}_1 \subset \mathrm{crit}(E)$. It then follows from the uniformized K{\L} property \cite[Lemma 6]{proximal} that $\exists r,\eta>0$ and $\exists \psi\in \kappa(0,\eta)$ such that for all $(x,v,w)\in\R^{3d}$ verifying $x \in \mathfrak{C}(x(\cdot))+B_{r},v\in B_r,w\in B_r$ (where $B_r$ is the $\R^d$-ball centered at $0_d$ with radius $r$) and $ 0<E(x,v,w)-\tilde{L} <\eta$, one has 
\begin{equation}\label{kl}
        \psi'(E(x,v,w)-\tilde{L})\Vert \nabla E(x,v,w)\Vert\geq 1 .
\end{equation}
It is clear that $V(t)=E(x(t),\dot{x}(t),\beta\nabla f(x(t)))$, and that $x^{\star}\in \mathrm{crit}(f)$ if and only if $(x^{\star},0,0)\in \mathrm{crit}(E)$.
}\\

%The set of cluster points of $x(t)$ is non-empty as $x(t)$  to a compact set. Let $x^{\star}$ be an arbitrary cluster point of $x(t)$, since $\lim_{t\rightarrow +\infty}\Vert\nabla f(x(t))\Vert=0$, then $x^{\star}\in \mathrm{crit}(f)$, moreover, there exists an increasing sequence $(t_k)_{k\in\N}\subset \R_+$ with $\lim_{k\rightarrow +\infty} t_k=+\infty$ such that $\lim_{k\rightarrow +\infty}x(t_k)=x^{\star}$. Without loss of generality, we can assume that $x^{\star}=0, f(0)=0, \nabla f(0)=0$ by using the change of variable $y(t)\eqdef x(t)-x^{*}$ and $\tilde{f}(y)\eqdef f(y+x^{\star})-f(x^{\star})$ (then $\tilde{f}(0)=0, \nabla \tilde{f}(0)=0$). We see also that $y(t)$ satisfies \eqref{eq:isehd} with objective function $\tilde{f}$ and initial conditions $y(0)=x_0-x^{\star}, y'(0)=v_0$.\\

\tcb{Let us define the translated Lyapunov function $\tilde{V}(t)=V(t)-\tilde{L}$. By the properties of $V$ proved above, we have $\lim_{t\rightarrow +\infty} \tilde{V}(t)=0$ and $\tilde{V}$ is non-increasing, and we can conclude that $\tilde{V}(t)\geq 0$ for every $t > 0$. Without loss of generality, we may assume that $\tilde{V}(t)> 0$ for every $t > 0$ (since otherwise $\tilde{V}(t)$ is eventually zero and thus $\dot{x}(t)$ is eventually zero in view of \eqref{eq:lyapdotVineq}, meaning that $x(\cdot)$ has finite length). This in turn implies that $\lim_{t\rightarrow +\infty} \psi( \tilde{V}(t))=0$. Define the constants $\delta_2=\max\pa{4,1+4\beta^2}$ and $\delta_3=\frac{\delta_1}{\sqrt{\delta_2}}$. \tcb{We have from \eqref{xkb} that $\lim_{t\rightarrow +\infty}\dist(x(t),\mathfrak{C}(x(\cdot)))=0$.} This together with the convergence claims on $\dot{x}$, $\nabla f(x)$ and $\tilde{V}$ imply that there exists $T  > 0$ large enough such that for all $t \geq T$
\begin{equation}
    \begin{cases}
        %\Vert x(t_{k})\Vert< \frac{r}{2\sqrt{2}},\\
        x(t)\in \mathfrak{C}(x(\cdot))+B_r,\\
        \Vert \dot{x}(t)\Vert<r,\\
        \beta\Vert \nabla f(x(t))\Vert<r,\\
        0<\tilde{V}(t)<\eta,\\
         \frac{1}{\delta_3}\psi( \tilde{V}(t))<\frac{r}{2\sqrt{2}}.\\
    \end{cases}
\end{equation}
}

%Now we define $\bar{t}=\sup\{t>t_{k'}:\Vert x(s)\Vert<\frac{r}{\sqrt{2}}, s\in [t_{k'},t]\}$ and we suppose that $\bar{t}<+\infty$.\\

\tcb{
We are now in position to apply \eqref{kl} to obtain 
\begin{equation}\label{klapplied}
 \psi'( \tilde{V}(t))\Vert \nabla E(x(t),\dot{x}(t),\beta\nabla f(x(t)))\Vert\geq 1,\quad \forall t \geq T.   
\end{equation}
On the other hand, for every $t \geq T$:  
\begin{align}
        -\frac{d}{dt}\psi(\tilde{V}(t))&=\psi'( \tilde{V}(t))(-\tilde{V}'(t)) \geq -\frac{\tilde{V}'(t)}{\Vert \nabla E(x(t),\dot{x}(t),\beta\nabla f(x(t)))\Vert}. \label{odeb}
\end{align} 
}
Additionally, for every $t > 0$ we have the bounds 
\begin{equation}\label{eq:VEineq}
\begin{aligned}
	-\tilde{V}'(t)&\geq \delta_1(\Vert \dot{x}(t)\Vert^2+\Vert\nabla f(x(t))\Vert^2),\\
    \Vert \nabla E(x(t),\dot{x}(t),\beta\nabla f(x(t)))\Vert^2&\leq \delta_2(\Vert \dot{x}(t)\Vert^2+\Vert\nabla f(x(t))\Vert^2).
\end{aligned}
\end{equation}
Combining the two previous bounds, then for every $t > 0$: \begin{equation}\label{vee}
    \Vert \nabla E(x(t),\dot{x}(t),\beta\nabla f(x(t)))\Vert\leq \sqrt{\frac{\delta_2}{\delta_1}}\sqrt{-\tilde{V}'(t)}.
    \end{equation}
By \eqref{odeb}, for every \tcb{$t\in [T,+\infty[$} 
\begin{align}\label{lowerode}
    -\frac{d}{dt}\psi(\tilde{V}(t))\geq \sqrt{\frac{\delta_1}{\delta_2}}\frac{-\tilde{V}'(t)}{\sqrt{-\tilde{V}'(t)}}=\sqrt{\frac{\delta_1}{\delta_2}}\sqrt{-\tilde{V}'(t)}\geq \delta_3\sqrt{\Vert \dot{x}(t)\Vert^2+\Vert\nabla f(x(t))\Vert^2}.
\end{align}
Integrating from \tcb{$T$ to $+\infty$}, we obtain 
\begin{equation}\label{ve1}
\int_{\tcb{T}}^{+\infty}\sqrt{\Vert \dot{x}(t)\Vert^2+\Vert\nabla f(x(t))\Vert^2}dt\leq \frac{1}{\delta_3}\psi(V(\tcb{T}))<\frac{r}{2\sqrt{2}}.
\end{equation}
\begin{comment}
And for every $t\in [t_{k'},\bar{t}[$, \begin{align*}
    \Vert x(t)\Vert&\leq \Vert x(t)-x(t_{k'})\Vert+\Vert x(t_{k'})\Vert\\
    &< \Big\Vert\int_{t_{k'}}^t \dot{x}(s)ds\Big\Vert+\frac{r}{2\sqrt{2}}\\
    &<\int_{t_{k'}}^t \Vert\dot{x}(s)\Vert ds+\frac{r}{2\sqrt{2}}\\
    &< \int_{t_{k'}}^{\bar{t}}\sqrt{\Vert \dot{x}(t)\Vert^2+\Vert\nabla f(x(t))\Vert^2}dt+\frac{r}{2\sqrt{2}}\\
    &< \frac{r}{2\sqrt{2}}+\frac{r}{2\sqrt{2}}=\frac{r}{\sqrt{2}}.
\end{align*}
Taking limit when $t\rightarrow \bar{t}$, we have that $\Vert x(\bar{t})\Vert<\frac{r}{\sqrt{2}}$, which contradicts the supposition of $\bar{t}<+\infty$. Then $\bar{t}=+\infty$, and $$\int_{t_{k'}}^{+\infty}\Vert\dot{x}(t)\Vert dt\leq \int_{t_{k'}}^{+\infty}\sqrt{\Vert \dot{x}(t)\Vert^2+\Vert\nabla f(x(t))\Vert^2}dt<\frac{r}{2\sqrt{2}},$$ 
\end{comment}
\tcb{Thus}
$$\int_{\tcb{T}}^{+\infty}\Vert\dot{x}(t)\Vert dt\leq \int_{\tcb{T}}^{+\infty}\sqrt{\Vert \dot{x}(t)\Vert^2+\Vert\nabla f(x(t))\Vert^2}dt<\frac{r}{2\sqrt{2}},$$
this implies that $\dot{x}\in\Lp^1(\R_+;\R^d)$. Therefore $x(t)$ has the Cauchy property and this in turn implies that $\lim_{t\rightarrow +\infty} x(t)$ exists, and is a critical point of $f$ since $\lim_{t\rightarrow +\infty}\Vert\nabla f(x(t))\Vert=0$. 
\end{enumerate}
\end{proof}

%%%%%%%%%%%%%%%%%%%%
\subsubsection{Trap avoidance}
In the previous section, we have seen the convergence of the trajectory to a critical point of the objective, which includes strict saddle points. We will call trap avoidance the effect of avoiding such points at the limit. If the objective function satisfies the strict saddle property (recall Definition \ref{ssp}) as is the case for a Morse function, this would imply convergence to a local minimum of the objective. The following theorem gives conditions to obtain such an effect.
\begin{theorem}\label{morseisehd}
Let $c>0$, assume that $0<\beta<\frac{2}{c}$ and take $\gamma \equiv c$. Suppose that $f:\R^d\rightarrow\R$ satisfies \eqref{H0} and is a Morse function. Consider \eqref{eq:isehd} in this setting. If the solution trajectory $x$ is bounded over $\R_+$, then the conclusions of Theorem~\ref{convisehd} hold. If, moreover, $\beta\neq\frac{1}{c}$, then for Lebesgue almost all initial conditions $x_0,v_0\in\R^d$, $x(t)$ converges (as $t\rightarrow +\infty$) to a local minimum of $f$.
\end{theorem}

\begin{proof}
Since Morse functions are $C^2$ and satisfy the K{\L} inequality (see Remark~\ref{rem}), then all the claims of Theorem~\ref{convisehd}, and in particular \ref{convx}\footnote{In fact, the proof is even more straightforward since the set of cluster points $\mathfrak{C}(x(\cdot))$ satisfies \eqref{xkb} and the critical points are isolated; see the proof of \cite[Theorem 4.1]{heavyb}.}, hold.

\smallskip

As in \cite[Theorem~4]{goudou}, we will use the global stable manifold theorem \cite[page~223]{Perko96} to get the last claim. We recall that \eqref{eq:isehd} is equivalent to \eqref{refor}, and that we are in the case $\gamma(t)=c$ for all $t$, \ie,  
    \begin{equation}\label{reforconst}
        \begin{cases}
            \dot{x}(t)+\beta\nabla f(x(t))-\left(\frac{1}{\beta}-c\right)x(t)+\frac{1}{\beta}y(t)&=0,\\
            \dot{y}(t)-\left(\frac{1}{\beta}-c\right)x(t)+\frac{1}{\beta}y(t)&=0,
        \end{cases}
    \end{equation}
        with initial conditions $x(0)=x_0, y(0)=y_0\eqdef -\beta (v_0+\beta\nabla f(x_0))+(1-\beta c)x_0$. Let us consider $F:\R^d\times\R^d\rightarrow\R^d\times\R^d$ defined by 
\[
        F(x,y)=\left(-\beta\nabla f(x)+\left(\frac{1}{\beta}-c\right)x-\frac{1}{\beta}y, \left(\frac{1}{\beta}-c\right)x-\frac{1}{\beta}y\right).
\]
Defining $z(t)=(x(t),y(t))$ and $z_0=(x_0,v_0)\in\R^{2d}$ , then \eqref{reforconst} is equivalent to the Cauchy problem 
    \begin{equation}\label{reforconst1}
        \begin{cases}
            \dot{z}(t)&=F(z(t)),\\
            z(0)&=z_0 .
        \end{cases}
    \end{equation}
We stated that when $0<\beta<\frac{2}{c}$ and $f$ is definable (see the first claim above), then the solution trajectory $z(t)$ converges (as $t\rightarrow +\infty$) to an equilibrium point of $F$. Let us denote $\Phi(z_0,t)$, the value at $t$ of the solution \eqref{reforconst1} with initial condition $z_0$. Assume that $\hat{z}$ is a hyperbolic equilibrium point of $F$ (to be shown below), meaning that $F(\hat{z})=0$ and that no eigenvalue of $J_F(\hat{z})$ has zero real part. Consider the invariant set
\[
W^s(\hat{z})=\{z_0\in\R^{2d}:\lim_{t\rightarrow +\infty} \Phi(z_0,t)=\hat{z}\}.
\]
The global stable manifold theorem \cite[page~223]{Perko96} asserts that $W^s(\hat{z})$ is an immersed submanifold of $\R^{2d}$, whose dimension equals the number of eigenvalues of $J_F(\hat{z})$ with negative real part.\\

First, we will prove that each equilibrium point of $F$ is hyperbolic. We notice that the set of equilibrium points of $F$ is $\{(\hat{x},(1-\beta c) \hat{x}): \hat{x}\in\mathrm{crit}(f)\}.$ On the other hand, we compute 
\[
J_F(x,y)=\begin{pmatrix}
            -\beta\nabla^2 f(x)+\left(\frac{1}{\beta}-c\right)I_d & -\frac{1}{\beta}I_d\\
            \left(\frac{1}{\beta}-c\right)I_d & -\frac{1}{\beta}I_d
\end{pmatrix}.
\]
Let $\hat{z}=(\hat{x},(1-\beta c) \hat{x})$, where $\hat{x}\in \mathrm{crit}(f)$. Then the eigenvalues of $J_F(\hat{z})$ are characterized by the roots in $\lambda\in\mathbb{C}$ of
    \begin{equation}\label{detisehd}
        \det\left(\begin{pmatrix}
            -\beta\nabla^2 f(\hat{x})+\left(\frac{1}{\beta}-c-\lambda\right)I_d & -\frac{1}{\beta}I_d\\
            \left(\frac{1}{\beta}-c\right)I_d & -\left(\lambda+\frac{1}{\beta}\right)I_d
        \end{pmatrix}\right)=0.
    \end{equation}

By Lemma~\ref{block}, we have that \eqref{detisehd} is equivalent to 
\begin{equation}\label{detisehd1}
    \det((1+\lambda\beta)\nabla^2 f(\hat{x})+(\lambda^2+\lambda c)I_d)=0.
\end{equation}
If $\lambda=-\frac{1}{\beta}$, then by \eqref{detisehd1}, $\beta c=1$, which is excluded by hypothesis. Therefore, $-\frac{1}{\beta}$ cannot be an eigenvalue, \ie $\lambda\neq -\frac{1}{\beta}$. We then obtain that \eqref{detisehd1} is equivalent to 
\begin{equation}\label{detisehd2}
    \det\left(\nabla^2 f(\hat{x})+\frac{\lambda^2+\lambda c}{(1+\lambda\beta)}I_d\right)=0.
\end{equation}
It follows that $\lambda$ satisfies \eqref{detisehd2} if and only if 
\[
\frac{\lambda^2+\lambda c}{(1+\lambda\beta)}=-\eta
\]
where $\eta\in\R$ is an eigenvalue of $\nabla^2 f(\hat{x})$. Equivalently, 
\begin{equation}\label{eqisehd}
        \lambda^2+(c+\eta\beta)\lambda+\eta=0.
\end{equation}

Let $\Delta_{\lambda}\eqdef(c+\eta\beta)^2-4\eta$. We distinguish two cases.
\begin{itemize}
    \item $\Delta_{\lambda}\geq 0$: then the roots of \eqref{eqisehd} are real and we rewrite \eqref{eqisehd} as 
    \[
    \lambda(\lambda+(c+\eta\beta))=-\eta,
    \] 
    since $\eta\neq 0$ (because $f$ is a Morse function), then $\lambda \neq 0$.
    \item $\Delta_{\lambda}< 0$: then \eqref{eqisehd} has a pair of complex conjugate roots whose real part is $-\frac{c+\eta\beta}{2}$. Besides, $\Delta_{\lambda}=c^2+2\beta\eta c+\eta^2\beta^2-4\eta$ can be seen as a quadratic on $c$ whose discriminant is given by $\Delta_c=16\eta$. The fact that $\Delta_{\lambda}< 0$ implies $\Delta_c > 0$, and thus $\eta>0$, therefore $-\frac{c+\eta\beta}{2}<0$. 
\end{itemize}    
Overall, this shows that every equilibrium point of $F$ is hyperbolic. \\

Let us recall that $\mathrm{crit}(f)=\bigcup_{k\in I} \{\hat{x}_{k}\}$. Thus, the set of equilibria of $F$ is also finite and each one takes the form $\hat{z}_k=(\hat{x}_{k},(1-\beta c)\hat{x}_k)$. Since we have already shown that each solution trajectory $x$ of \eqref{eq:isehd} converges towards some $\hat{x}_{k}$, the
following partition then holds 
\[
\R^d\times \R^d=\bigcup_{k\in I} W^s(\hat{z}_k).
\]
Let 
\[
I^-=\{k\in I: \text{ each eigenvalue of $J_F(\hat{z}_k)$ has negative real part.}\},
\]
and $J\eqdef I\setminus I^-$. Now, the global stable manifold theorem \cite[page~223]{Perko96} allows to claim that $W^s(\hat{z}_k)$ is an immersed submanifold of $\R^{2d}$ whose dimension is $2d$ when $k\in I^-$ and at most $2d-1$ when $k \in J$. \\

Let $k\in I^-$, we claim that $\nabla^2 f(\hat{x}_k)$ has only positive eigenvalues. By contradiction, let us assume that $\eta_0<0$ is an eigenvalue of $\nabla^2 f(\hat{x}_k)$ ($\eta_0=0$ is not possible due to the Morse hypothesis). Each solution $\lambda$ of \eqref{eqisehd} is an eigenvalue of $J_F(\hat{z}_k)$ and one of these solutions is 
\[
\frac{-(c+\eta\beta)+\sqrt{(c+\eta\beta)^2-4\eta_0}}{2}
\] 
which is positive since $\eta_0<0$. We then have 
\[
\frac{-(c+\eta_0\beta)+\sqrt{(c+\eta_0\beta)^2-4\eta_0}}{2}>\frac{-(c+\eta_0\beta)+|c+\eta_0\beta|}{2}\geq 0 ,
\]
hence contradicting the assumption that $k\in I^-$. In conclusion, the set of initial conditions $z_0$ such that $\Phi(z_0,t)$ converges to $(x_b,(1-\beta c)x_b)$ (as $t\rightarrow +\infty$), where $x_b$ is not a local minimum of $f$ is $\bigcup_{k\in J} W^s(\hat{z}_k)$ which has Lebesgue measure zero. Therefore, due to the equivalence between \eqref{refor} and \eqref{eq:isehd} and that Morse functions satisfy the strict saddle property (see Remark~\ref{rem:strictsaddlemorse}), we indeed have that for almost all initial conditions $x_0,v_0\in\R^d$, the solution trajectory of \eqref{eq:isehd} will converge to a local minimum of $f$.
\end{proof}

%%%%%%%%%%%%%%%%%%%%
\subsubsection{Convergence rate}
When the objective function is definable, we now provide the convergence rate on the Lyapunov function $E$ in \eqref{eq:energE}, hence $f$, and on the solution trajectory $x$.
\begin{theorem}\label{cor}
Consider the setting of Theorem~\ref{convisehd} with $f$ being also definable and the solution trajectory $x$ is bounded. Recall the function $E$ from \eqref{eq:energE}, which is also definable, and denote $\psi$ its desingularizing function and $\Psi$ any primitive of $-\psi'^2$. Then, $x(t)$ converges (as $t\rightarrow +\infty$) to $x_{\infty} \in \mathrm{crit}(f)$. Denote $\tilde{V}(t)\eqdef E(x(t),\dot{x}(t),\beta\nabla f(x(t)))-f(x_{\infty})$. The following rates of convergence hold:
\begin{itemize}
\item  If $\lim_{t\rightarrow 0}\Psi (t)\in\R$, we have $E(x(t),\dot{x}(t),\beta\nabla f(x(t)))$ converges to $f(x_{\infty})$ in finite time.
\item If $\lim_{t\rightarrow 0}\Psi (t)=+\infty$, there exists some $t_1 \geq 0$ such that 
\begin{equation}\label{eq:VrateKL}
        \tilde{V}(t)=\mathcal{O}(\Psi^{-1}(t-t_1)) .
\end{equation}
    Moreover, 
    \begin{equation}\label{eq:xrateKL}
        \Vert x(t)-x_{\infty}\Vert=\mathcal{O}(\psi\circ\Psi^{-1}(t-t_1)) .
    \end{equation}
    \end{itemize}
\end{theorem}
\begin{proof}
This proof is a generalization of \cite[Theorem~2.7]{chill} to the dynamics \eqref{eq:isehd}. Let $\delta_0\eqdef \frac{\delta_2}{\delta_1}, \delta_3>0$ and $T > 0$ for $\delta_1,\delta_2,\delta_3,T$ defined in the proof of Theorem~\ref{convisehd}. Using \eqref{eq:VEineq} then \eqref{odeb}, we have for $t>T$ 
	\begin{align}
        \frac{d}{dt}\Psi(\tilde{V}(t))&=\Psi'(\tilde{V}(t))\tilde{V}'(t) \nonumber\\
        &=-\psi'^2(\tilde{V}(t))\tilde{V}'(t) \nonumber\\
        &\geq \delta_0 \psi'^2(\tilde{V}(t))\Vert\nabla E(x(t),\dot{x}(t),\beta\nabla f(x(t))\Vert^2 \nonumber\\
        &\geq\delta_0. \label{eq:PsiVdot}
    \end{align}
    Integrating on both sides from $T$ to $t$ we obtain that for every $t>T$
    \[
	\Psi(\tilde{V}(t))\geq \delta_0(t-T)+\Psi(\tilde{V}(T)).
	\]
    Following the arguments shown in \cite[Theorem 3.1.12]{garrigos}, if $\lim_{t\rightarrow 0}\Psi (t)\in\R$, then $\tilde{V}(t)$ converges to $0$ in finite time. Otherwise, we take the inverse of $\Psi$, which is non-increasing, on both sides of \eqref{eq:PsiVdot} to obtain the desired bound. Finally, using \eqref{ve1} we also have for every $t>T$ \begin{equation}
\Vert x(t)-x_{\infty}\Vert\leq \int_{t}^{+\infty}\Vert\dot{x}(s)\Vert ds\leq\frac{1}{\delta_3} \psi(\tilde{V}(t)) \leq \frac{1}{\delta_3}\psi\circ\Psi^{-1}(\delta_0(t-T)+\Psi(\tilde{V}(T))) .
\end{equation}
\end{proof}

\begin{remark}
Observe that the convergence rate \eqref{eq:VrateKL} holds also on $f(x(t))-f(x_{\infty})$ and $\norm{\dot{x}(t)+\nabla f(x(t))}^2$.
\end{remark}

We now specialize this to the {\L}ojasiewicz case.
\begin{corollary}\label{corver}
Consider the setting of Theorem~\ref{cor} where now $f$ satisfies the {\L}ojasiewicz inequality with desingularizing function $\psi_f(s)=c_f s^{1-q}$, $q \in [0,1[$, $c_f > 0$.
%, \ie $x(t)$ be the solution of \eqref{eq:isehd}, $x_{\infty}$ be the limit of $x(t)$ (as $t\rightarrow +\infty$),
%    $$E(x,v,w)=f(x)+\frac{1}{2}\Vert v+w\Vert^2, \text{ and } \tilde{V}(t)=E(x(t),\dot{x}(t),\beta\nabla f(x(t)))-f(x_{\infty}).$$
    Then there exists some $t_1 > 0$ such that the following convergence rates hold:
    \begin{itemize}
        \item If $q \in [0,\frac{1}{2}]$, then there exists $\mu>0$ such that
        \begin{equation}
        \tilde{V}(t)=\mathcal{O}\pa{\exp(-\mu(t-t_1))} \quad \text{and} \quad 
        \Vert x(t)-x_{\infty}\Vert=\mathcal{O}\pa{\exp\left(-\frac{\mu(t-t_1)}{2}\right)} .
    \end{equation}
    \item If $q \in ]\frac{1}{2},1[$, then
    \begin{equation}
        \tilde{V}(t)=\mathcal{O}\pa{(t-t_1)^{\frac{-1}{2q-1}}} \quad \text{and} \quad 
        \Vert x(t)-x_{\infty}\Vert=\mathcal{O}\pa{(t-t_1)^{-\frac{1-q}{2q-1}}} .
    \end{equation}
    \end{itemize}
\end{corollary}

\begin{proof}
$E$ is a separable quadratic perturbation of $f$. But a quadratic function is {\L}ojasiewicz with exponent $1/2$. It then follows from the {\L}ojasiewicz  exponent calculus rule in \cite[Theorem~3.3]{calckl} that the desingularizing function of $E$ is $\psi_E(s)=c_E s^{1-q_E}$ for some $c_E>0$ and $q_E=\max\pa{q,\frac{1}{2}}$. Then,
    \begin{itemize}
        \item If $q \in [0,\frac{1}{2}]$ then $q_E=\frac{1}{2}$ and $\Psi(s)=\frac{c_1^2}{4}\ln\left(\frac{1}{s}\right)$. This implies that $\Psi^{-1}(s)=\exp(-\frac{4}{c_1^2}s)$.
        \item If $q \in ]\frac{1}{2},1[$ then $q_E=q$ and $\Psi(s)=\frac{c_1^2}{4(2q-1)}s^{1-2q}$. This implies that $\Psi^{-1}(s)=\frac{4(2q-1)}{c_1^2}s^{\frac{-1}{2q-1}}$.
    \end{itemize}
    We conclude in both cases by using Theorem~\ref{cor}. 
\end{proof}
%\todo{Instantiate the {\L}ojasiewicz case $\psi(s)=cs^{1-q}$. Add a remark that we have global linear convergence for PL functions and exhibit the rate. Do the same in the discrete case.}

%%%%%%%%%%%%%%%%%%%%%%%%%%%%%%
\subsection{Algorithmic scheme}
Now we will consider the following finite differences explicit discretization of \eqref{eq:isehd} with step-size $h>0$ and for $k\geq 1$:
\begin{equation}%\label{eq:igahd1}
\frac{x_{k+1}-2x_k+x_{k-1}}{h^2}+\gamma(kh)\frac{x_{k+1}-x_k}{h}+\beta\frac{\nabla f(x_k)-\nabla f(x_{k-1})}{h}+\nabla f(x_k)=0.
\end{equation}
Rearranging, this equivalently reads
\begin{equation}\label{isehd-disc}\tag{ISEHD-Disc}
        \begin{aligned}
            \begin{cases}
                y_k&=x_k+\alpha_k(x_k-x_{k-1})-\beta_k(\nabla f(x_k)-\nabla f(x_{k-1})),\\
                x_{k+1}&=y_k-s_k\nabla f(x_k),
            \end{cases}
        \end{aligned}
\end{equation}
with initial conditions $x_0, x_1 \in \R^d$, where $\gamma_k\eqdef \gamma(kh), \alpha_k\eqdef \frac{1}{1+\gamma_k h}, \beta_k\eqdef \beta h\alpha_k, s_k\eqdef h^2\alpha_k$. 

%%%%%%%%%%%%%%%%%%%%
\subsubsection{Global convergence and trap avoidance}
The following theorem summarizes our main results on the behavior of \eqref{isehd-disc}. Observe that as the discretization is explicit, we will need $\nabla f$ to be globally Lipschitz continuous. We refer to Section \ref{sec:backtracking} for results that require only local Lipschitz continuity of $\nabla f$, where the stepsize is determined via backtracking.

\begin{theorem}\label{convisehddisc}
Let $f:\R^d\rightarrow\R$ be satisfying \eqref{H0} with $\nabla f$ being globally $L$-Lipschitz-continuous. Consider the scheme \eqref{isehd-disc} with $h>0$, $\beta\geq 0$ and $c\leq \gamma_k\leq C$ for some $c,C > 0$ and all $k \in \N$. Then the following holds:
    \begin{enumerate}[label=(\roman*)]
        \item 
    If $\beta+\frac{h}{2}<\frac{c}{L}$, then $(\Vert\nabla f(x_k)\Vert)_{k\in\N}\in\ell^2(\N)$, and $(\Vert x_{k+1}-x_k\Vert)_{k\in\N}\in \ell^2(\N)$, in particular $$\lim_{k\rightarrow +\infty} \Vert\nabla f(x_k)\Vert=0.$$ 
    \item Moreover, if $\seq{x_k}$ is bounded and $f$ is definable, then $(\Vert x_{k+1}-x_{k}\Vert)_{k\in\N}\in\ell^1(\N)$ and $x_k$ converges (as $k\rightarrow +\infty$) to a critical point of $f$.
  
    \item Furthermore, if $\gamma_k\equiv c>0$, $0<\beta<\frac{c}{L},  \beta\neq \frac{1}{c}$, and $h<\min(2\left(\frac{c}{L}-\beta\right),\frac{1}{L\beta})$, then for almost all $x_0, x_1\in\R^d$, $x_k$ converges (as $k\rightarrow +\infty$) to a critical point of $f$ that is not a strict saddle. Consequently, if $f$ satisfies the strict saddle property then for almost all $x_0, x_1\in\R^d$, $x_k$ converges (as $k\rightarrow +\infty$) to a local minimum of $f$.
        \end{enumerate}
\end{theorem}

\begin{remark}
When $\beta=0$, we recover the HBF method and the condition $h<\min(2\left(\frac{c}{L}-\beta\right),\frac{1}{L\beta})$ becomes $h<\frac{2c}{L}$.
\end{remark}

\begin{proof}
 \begin{enumerate}[label=(\roman*)]
 \item
    By definition of $x_{k+1}$ in \eqref{isehd-disc}, for $k\in\N^*$ 
    \begin{equation}
    x_{k+1}=\argmin_{x\in\R^d}\frac{1}{2}\Vert x-(y_k-s_k\nabla f(x_k))\Vert^2 .
    \end{equation}
    $1-$strong convexity of $x\mapsto\frac{1}{2}\Vert x-(y_k-s_k\nabla f(x_k))\Vert^2$ then yields
    \begin{equation}\label{eq1}
        \frac{1}{2}\Vert x_{k+1}-(y_k-s_k\nabla f(x_k))\Vert^2\leq \frac{1}{2}\Vert x_k-(y_k-s_k\nabla f(x_k))\Vert^2-\frac{1}{2}\Vert x_{k+1}-x_k\Vert^2.
    \end{equation}
    Let $\ubar{\alpha}=\frac{1}{1+Ch},\bar{\alpha}=\frac{1}{1+ch}, \ubar{s}=h^2\ubar{\alpha},\bar{s}=h^2\bar{\alpha}$, and thus for every $k\in\N$, $\ubar{\alpha}\leq \alpha_k\leq \bar{\alpha}$  and $\ubar{s}\leq s_k\leq \bar{s}$. Let also $v_k\eqdef x_k-x_{k-1}$, $z_k\eqdef \alpha_k v_k-\beta_k(\nabla f(x_k)-\nabla f(x_{k-1}))$, then $y_k=x_k+z_k$. After expanding the terms of \eqref{eq1} we have that \begin{equation}\label{eq2}
    \begin{aligned}
        \langle \nabla f(x_k),v_{k+1}\rangle&\leq -\frac{\Vert v_{k+1}\Vert^2}{s_k}+\frac{1}{s_k}\langle v_{k+1},z_k\rangle\\
        &\leq -\frac{\Vert v_{k+1}\Vert^2}{\bar{s}}+\frac{1}{h^2}\langle v_{k+1},v_k\rangle-\frac{\beta}{h}\langle v_{k+1},\nabla f(x_k)-\nabla f(x_{k-1})\rangle.
    \end{aligned}
    \end{equation}
    By the descent lemma for $L$-smooth functions, we obtain 
    \begin{equation}\label{eq3}
        f(x_{k+1})\leq f(x_k)+\langle \nabla f(x_k), v_{k+1}\rangle+\frac{L}{2}\Vert v_{k+1}\Vert^2.
    \end{equation}
    Using the bound in \eqref{eq2}, we get
    \begin{equation}
        f(x_{k+1})\leq f(x_k)+\frac{1}{h^2}\langle v_{k+1}, v_k\rangle-\frac{\beta}{h}\langle v_{k+1},\nabla f(x_k)-\nabla f(x_{k-1})\rangle-\left(\frac{1}{\bar{s}}-\frac{L}{2}\right)\Vert v_{k+1}\Vert^2.
    \end{equation}
    According to our hypothesis $h<\frac{2c}{L}$, so $h<\frac{c+\sqrt{c^2+2L}}{L}$ and this implies that $\bar{s}<\frac{2}{L}$. Using Young's inequality twice, for $\varepsilon,\varepsilon'>0$, the fact that $\nabla f$ is $L-$Lipschitz, and adding $\frac{\varepsilon+\varepsilon'}{2}\Vert v_{k+1}\Vert^2$ at both sides, then
    \begin{equation}
    \begin{aligned}
        f(x_{k+1})+\frac{\varepsilon+\varepsilon'}{2}\Vert v_{k+1}\Vert^2
        &\leq  f(x_k)+\frac{\varepsilon+\varepsilon'}{2}\Vert v_k\Vert^2\\
        &+\left(\frac{1}{2}\left[\frac{1}{h^4\varepsilon}+\varepsilon+\frac{\beta^2L^2}{h^2\varepsilon'}+\varepsilon'\right]-\left(\frac{1}{\bar{s}}-\frac{L}{2}\right)\right)\Vert v_{k+1}\Vert^2.
        \end{aligned}
    \end{equation}
    In order to make the last term negative, we want to impose 
    \begin{equation}\label{condhh}
    \frac{1}{2}\left[\frac{1}{h^4\varepsilon}+\varepsilon+\frac{\beta^2L^2}{h^2\varepsilon'}+\varepsilon'\right]<\left(\frac{1}{\bar{s}}-\frac{L}{2}\right).
    \end{equation}
    Minimizing the left-hand side with respect to $\varepsilon,\varepsilon'>0$ we get $\varepsilon=\frac{1}{h^2}, \varepsilon'=\frac{\beta L}{h}$, and one can check that in this case, the condition \eqref{condhh} becomes equivalent to $\beta+\frac{h}{2}<\frac{c}{L}$ which is assumed in the hypothesis.\\
    
    Setting $C_1\eqdef \frac{1}{h^2}+\frac{\beta L}{h}, \delta\eqdef \frac{1}{\bar{s}}-\frac{L}{2}-\frac{1}{h^2}-\frac{\beta L}{h}>0$, and defining  $V_k\eqdef f(x_k)+\frac{C_1}{2}\Vert v_k\Vert^2$, we have for any $k\in\N^*$
    \begin{equation}\label{eq:Vklyap}
        V_{k+1}\leq V_k -\delta\Vert v_{k+1}\Vert^2 .
    \end{equation}
    Clearly, $V_k$ is non-increasing and bounded from below, hence $\lim_{k\rightarrow +\infty} V_k$ exists (say $\tilde{L}$). Summing this inequality over $k$, we have that $(\Vert v_{k+1}\Vert)_{k\in\N}\in \ell^2(\N)$ entailing that $\lim_{k\rightarrow +\infty} \Vert v_k\Vert=0$. In turn,we have that $\lim_{k\rightarrow +\infty} f(x_k)=\tilde{L}$. Embarking again from the update in \eqref{isehd-disc}, we have
    \begin{align*}
        s_k\Vert \nabla f(x_k)\Vert = \Vert x_{k+1}-y_k\Vert&\leq \Vert x_{k+1}-x_k\Vert+\Vert x_{k}-y_k\Vert\\
        &\leq \Vert v_{k+1}\Vert+(\alpha_k+\beta_k L)\Vert v_k\Vert\\
        &\leq \Vert v_{k+1}\Vert+\Vert v_{k}\Vert,
    \end{align*}
    since $\bar{\alpha}(1+\beta h L)<1$ by hypothesis. Therefore 
    \begin{align*}
        \Vert \nabla f(x_k)\Vert^2\leq \delta_2(\Vert v_{k+1}\Vert^2+\Vert v_k\Vert^2),
    \end{align*}
    where $\delta_2=\frac{2}{\ubar{s}^2}$. Consequently $(\Vert \nabla f(x_k)\Vert)_{k\in\N}\in\ell^2(\N)$, which implies that $\lim_{k\rightarrow +\infty} \Vert\nabla f(x_k)\Vert=0$.

    \item If, moreover, $(x_k)_{k\in \N}$ is bounded, then the set of its cluster points $\mathfrak{C}(\seq{x_k})$ satisfies (see \eg \cite[Lemma~5]{proximal}):  
    \begin{equation}\label{xkbdisc}
\begin{aligned}
    \begin{cases}
        \mathfrak{C}(\seq{x_k}) \subseteq \mathrm{crit}(f);\\
        \mathfrak{C}(\seq{x_k}) \text{ is non-empty, compact and connected};\\
        \text{$f$ is constant on } \mathfrak{C}(\seq{x_k}).
    \end{cases}     
\end{aligned}
\end{equation}
Define
\begin{equation}\label{eq:energEdisc}
E: (x,v) \in \R^{2d} \mapsto f(x)+\frac{C_1}{2}\Vert v\Vert^2 .
\end{equation}
Since $f$ is definable, so is $E$ as the sum of a definable function and an algebraic one, whence $E$ satisfies the K{\L} inequality. Let $\mathfrak{C}_1=\mathfrak{C}((x_k)_{k\in \N})\times \{0_d\}$. Since $E\big|_{\mathfrak{C}_1}=\tilde{L}, \nabla E\big|_{\mathfrak{C}_1}=0$, $\exists r,\eta>0, \exists \psi\in\kappa(0,\eta)$ such that for every $(x,v)$ such that $ x\in \mathfrak{C}((x_k)_{k\in \N})+B_{r},v\in B_r,$ (where $B_r$ is the $\R^d$-ball centred at $0_d$ with radius $r$) and $0<E(x,v)-\tilde{L}<\eta$, one has 
    \begin{equation}\label{kldisc}
        \psi'(E(x,v)-\tilde{L})\Vert \nabla E(x,v)\Vert\geq 1 .
    \end{equation}
    Let us define $\tilde{V}_k=V_k-\tilde{L}$, or equivalently $\tilde{V}_k=E\left(x_k,v_k\right)-\tilde{L}$. From \eqref{eq:Vklyap}, $\seq{\tilde{V}_k}$ is a non-increasing sequence and its limit is $0$ by definition of $\tilde{L}$. This implies that that $\tilde{V}_k\geq 0$ for all $k \in \N^*$. We may assume without loss of generality that $\tilde{V}_k > 0$. Indeed, suppose there exists $K \in \N$ such that $\tilde{V}_K = 0$, then the decreasing property \eqref{eq:Vklyap} implies that $\tilde{V}_k = 0$ holds for all $k \geq K$. Thus $v_{k+1}=0$, or equivalently $x_k=x_K$, for all $k \geq K$, hence $\seq{x_k}$ has finite length.
    
    Since $\lim_{k\rightarrow +\infty}\dist(x_k,\mathfrak{C}((x_k)_{k\in \N}))=0, \lim_{k\rightarrow +\infty}\Vert v_k\Vert=0,$ and $\lim_{k\rightarrow +\infty}\tilde{V}_k=0$, there exists $\tilde{K}\in \N$ such that for all $k \geq \tilde{K}$, 
    \begin{align*}
    \begin{cases}
         x_k\in \mathfrak{C}((x_k)_{k\in \N})+B_r ;\\
        \Vert v_k\Vert<r ;\\
        0<\tilde{V}_k<\eta.
    \end{cases}     
\end{align*}
    Then, by \eqref{kldisc}, we have 
    \begin{equation}\label{kldiscapp}
\psi'(\tilde{V}_k)\Vert \nabla E\left(x_k,v_k\right)\Vert\geq 1,\quad \forall k\geq \tilde{K}.
    \end{equation}
    
    By concavity of $\psi$ and \eqref{kldiscapp}, we have
    \begin{align*}
        \psi(\tilde{V}_k)-\psi(\tilde{V}_{k+1})&\geq -\psi'(\tilde{V}_k)(\tilde{V}_{k+1}-\tilde{V}_{k})\\
        &\geq \delta \psi'(\tilde{V}_k)\Vert v_{k+1}\Vert^2\\
        &\geq \delta \frac{\Vert v_{k+1}\Vert^2}{\Vert \nabla E\left(x_k,v_k\right)\Vert}.
    \end{align*} 

    On the other hand, 
    \begin{equation}\label{eq:nablaEbnddisc}
    \Vert\nabla E\left(x_k,v_k\right)\Vert\leq \delta_3(\Vert v_{k+1}\Vert+\Vert v_k\Vert)
    \end{equation}
    where $\delta_3=\sqrt{C_1^2+\delta_2}$. Let us define for $k\in\N^*$, $(\Delta \psi)_k\eqdef \psi(\tilde{V}_k)-\psi(\tilde{V}_{k+1})$ and $\delta_4=\frac{\delta_3}{\delta}$. We then have for all $k \geq \tilde{K}$, 
    \begin{align*}
        \Vert v_{k+1}\Vert^2\leq \delta_4(\Delta\psi)_k(\Vert v_{k}\Vert+\Vert v_{k+1}\Vert).
    \end{align*}
    Using Young's inequality and concavity of $\sqrt{\cdot}$, this implies that for every $\varepsilon>0$
    \begin{align*}
        \Vert v_{k+1}\Vert\leq \delta_4\frac{(\Delta\psi)_k}{\sqrt{2}\varepsilon}+\varepsilon\frac{\Vert v_{k+1}\Vert+\Vert v_k\Vert}{\sqrt{2}}.
    \end{align*}
    Rearranging the terms and imposing $0<\varepsilon<\sqrt{2}$ gives
    \begin{align*}
       \left(1-\frac{\varepsilon}{\sqrt{2}}\right) \Vert v_{k+1}\Vert\leq \delta_4\frac{(\Delta\psi)_k}{\sqrt{2}\varepsilon}+\varepsilon\frac{\Vert v_k\Vert}{\sqrt{2}}.
    \end{align*}
    Dividing by $\left(1-\frac{\varepsilon}{\sqrt{2}}\right)$ on both sides, we get 
    \begin{equation}\label{estvk}
       \Vert v_{k+1}\Vert\leq \delta_4\frac{(\Delta\psi)_k}{\varepsilon(\sqrt{2}-\varepsilon)}+\varepsilon\frac{\Vert v_k\Vert}{\sqrt{2}-\varepsilon}.
    \end{equation}
    Choosing now $\varepsilon$ such that $0<\varepsilon<\frac{\sqrt{2}}{2},$ we get that $0<\frac{\varepsilon}{\sqrt{2}-\varepsilon}<1$. Since $(\Delta\psi)_k\in \ell^1(\N^*)$ as a telescopic sum, we conclude that $(\Vert v_k\Vert)_{k\in\N}\in \ell^1(\N^*)$. This means that $\seq{x_k}$ has finite length, hence is a Cauchy sequence, entailing that $x_k$ has a limit (as $k\rightarrow +\infty$) denoted $x_{\infty}$ which is a critical point of $f$ since $\lim_{k\rightarrow +\infty}\Vert\nabla f(x_k)\Vert=0$.

    \item If $\gamma_k\equiv c$, we denote $\alpha_k\equiv \alpha=\frac{1}{1+ch}, \beta_k\equiv\tilde{\beta}=\beta h\alpha, s_k\equiv s=h^2\alpha$. Let $z_{k}=(x_k,x_{k-1})$ for $k\geq 1$, and $g:\R^d\times \R^d\rightarrow \R^d\times\R^d$ defined by 
    \[
    g:(x_+,x_-)\mapsto [(1+\alpha)x_+-\alpha x_- -(\tilde{\beta}+s)\nabla f(x_+)+\tilde{\beta}\nabla f(x_-),x_+] .
    \]
    \eqref{isehd-disc} is then equivalent to 
    \begin{equation}\label{zkp1}
        z_{k+1}=g(z_k).
    \end{equation}
    To complete the proof, we will capitalize on \cite[Corollary~1]{localminima} which builds on the center stable manifold theorem \cite[Theorem~III.7]{shub_global_1987}. For this, one needs to check two conditions: 
    \begin{enumerate}[label=(\alph*)]
        \item $\det (J_g(x_+,x_-))\neq 0 $ for every $x_+,x_-\in\R^d$. \label{cond:a}
        \item Let $\mathcal{A}_g^{\star} \eqdef \{(x,x)\in\R^{2d}: x\in\mathrm{crit}(f), \max_i |\lambda_i(J_g(x,x))|>1 \}$, $\mathcal{X}^{\star}$ be the set of strict saddle points of $f$, and $\hat{\mathcal{X}}\eqdef\{(x,x)\in\R^{2d}:x\in \mathcal{X}^{\star}\}$. One needs to check that  $\hat{\mathcal{X}}\subset \mathcal{A}_g^{\star}$. \label{cond:b}
    \end{enumerate}

$\mathcal{A}_g^{\star}$ is the set of unstable fixed points. Indeed, the fixed points of $g$ are of the form $(x^{\star},x^{\star})$ where $x^{\star}\in\mathrm{crit}(f)$. 

\medskip

    We first compute $J_g(x_+,x_-)$, given by
    \begin{equation}
          \begin{pmatrix}
        (1+\alpha)I_d-(\tilde{\beta}+s)\nabla^2 f(x_+) & -\alpha I_d+\tilde{\beta}\nabla^2 f(x_-)\\
        I_d & 0_{d\times d}
    \end{pmatrix}
    \end{equation}
 This is a block matrix that comes in a form amenable to applying Lemma~\ref{block}. We then have 
 \[
\det(J_g(x_+,x_-))=\det (\alpha I_d-\tilde{\beta}\nabla^2 f(x_-)). 
\]
Since the eigenvalues of $\nabla^2 f(x_-)$ are contained in $[-L,L]$, if $\alpha>L\tilde{\beta}$, then $\alpha-\tilde{\beta}\eta\neq 0$ for every eigenvalue $\eta\in\R$ of $\nabla^2 f(x_-)$. This implies that the first condition is satisfied, \ie $\det(J_g(x_+,x_-))\neq 0$ for every $x_+,x_-\in\R^d$. The condition $\alpha>L\tilde{\beta}$ in terms of $h$ reads $h<\frac{1}{L\beta}$, since we already needed $h<2\left(\frac{c}{L}-\beta\right)$, we just ask $h$ to be less than the minimum of the two quantities.

\smallskip

To check the second condition, let us take $x$ a strict saddle point of $f$, \ie $x\in \mathrm{crit}(f)$ and \linebreak
$\lambda_{\min}(\nabla^2 f(x))=-\eta < 0$. To compute the eigenvalues of $J_g(x,x)$ we consider $$\det\left(\begin{pmatrix}
     (1+\alpha-\lambda)I_d-(\tilde{\beta}+s)\nabla^2 f(x) & -\alpha I_d+\tilde{\beta}\nabla^2 f(x)\\
     I_d & -\lambda I_d
 \end{pmatrix}\right)=0.$$
 Again by Lemma~\ref{block}, we get that \begin{align*}
     &\det\left(\begin{pmatrix}
     (1+\alpha-\lambda)I_d-(\tilde{\beta}+s)\nabla^2 f(x) & -\alpha I_d+\tilde{\beta}\nabla^2 f(x)\\
     I_d & -\lambda I_d
 \end{pmatrix}\right)=\\
 &\det[(-\lambda(1+\alpha)+\lambda^2)I_d+\lambda(\tilde{\beta}+s)\nabla^2 f(x)+\alpha I_d-\tilde{\beta}\nabla^2 f(x)]=\\
 &\det[(\lambda(\tilde{\beta}+s)-\tilde{\beta})\nabla^2 f(x)+(\lambda^2-\lambda(1+\alpha)+\alpha)I_d] .
 \end{align*}
 We then need to solve for $\lambda$
 \begin{equation}\label{quadd}
     \det[(\lambda(\tilde{\beta}+s)-\tilde{\beta})\nabla^2 f(x)+(\lambda^2-\lambda(1+\alpha)+\alpha)I_d]=0.
 \end{equation}
\begin{comment}

If $0<\tilde{\beta}=s$, then \eqref{quadd} is equivalent to $$\det\left(\nabla^2 f(x)+\frac{\lambda^2-\lambda(1+\alpha)+\alpha}{\tilde{\beta}}I_d\right)=0.$$
Therefore, for every eigenvalue $\eta'$ of $\nabla^2 f(x)$, we have that $$\frac{\lambda^2-\lambda(1+\alpha)+\alpha}{\tilde{\beta}}=-\eta'.$$

In particular, we can take $\eta'=-\eta$ (the minimum eigenvalue, which is negative), solving the quadratic equation we get the biggest solution (the one with the plus sign) is bigger than one, in fact,
$$\frac{(1+\alpha)+\sqrt{(1-\alpha)^2+4\eta \tilde{\beta}}}{2}>\frac{1+\alpha+|1-\alpha|}{2}\geq 1,$$
since $4\eta\tilde{\beta}>0$, then we have ensured the second condition in this case.\\

For the rest of the proof, we can consider $\beta\neq s$.
\end{comment}
If $\lambda=\frac{\tilde{\beta}}{\tilde{\beta}+s}$, then \eqref{quadd} becomes $$\left(\frac{\tilde{\beta}}{\tilde{\beta}+s}\right)^2-\left(\frac{\tilde{\beta}}{\tilde{\beta}+s}\right)(1+\alpha)+\alpha=0.$$
This implies that $\alpha=\frac{\tilde{\beta}}{\tilde{\beta}+s}$, which in terms of $\beta$ and $c$ is equivalent to $\beta c=1$. But this case is excluded by hypothesis. 
Now we can focus on the case where $\lambda\neq\frac{\tilde{\beta}}{\tilde{\beta}+s}$ and we can rewrite \eqref{quadd} as 
\begin{equation}\label{detisehddisc}
\det\left(\nabla^2 f(x)-\frac{\lambda^2-\lambda(1+\alpha)+\alpha}{\tilde{\beta}-\lambda(\tilde{\beta}+s)}I_d\right)=0.
\end{equation}
Therefore, as argued for the time-continuous dynamic, for every eigenvalue $\eta'\in\R$ of $\nabla^2 f(x)$, $\lambda$ satisfies \eqref{detisehddisc} if and only if 
\[
\frac{\lambda^2-\lambda(1+\alpha)+\alpha}{\tilde{\beta}-\lambda(\tilde{\beta}+s)}=\eta'.
\]
where $\eta'\in\R$ is an eigenvalue of $\nabla^2 f(\hat{x})$.
Thus if $\eta'=-\eta$ is negative, we have 
\begin{equation}\label{quadd1}
     \lambda^2-\lambda((1+\alpha)+\eta(\tilde{\beta}+s))+\alpha+\eta\tilde{\beta}=0.
\end{equation}
 We analyze its discriminant $\Delta_{\lambda}=((1+\alpha)+\eta(\tilde{\beta}+s))^2-4(\alpha+\eta\tilde{\beta})$. After developing the terms we get that $$\Delta_{\lambda}=\alpha^2+2\alpha(\eta(\tilde{\beta}+s)-1)+(\eta(\tilde{\beta}+s)+1)^2-4\eta\tilde{\beta},$$
 which can be seen as a quadratic equation on $\alpha$. We get that its discriminant $\Delta_{\alpha}$ is $-16\eta s$, which is negative (since $\eta, s$ are positive), thus the quadratic equation on $\alpha$ does not have real roots, implying that $\Delta_{\lambda}>0$. We can write the solutions of \eqref{quadd1}, 
 \[
 \lambda=\frac{((1+\alpha)+\eta(\tilde{\beta}+s))\pm \sqrt{\Delta_{\lambda}}}{2}.
 \]
Let us consider the biggest solution (the one with the plus sign) and let us see that $\lambda>1$, this is equivalent to 
\[
\frac{((1+\alpha)+\eta(\tilde{\beta}+s)) + \sqrt{\Delta_{\lambda}}}{2}>1,
\]
which in turn is equivalent to $$\sqrt{\Delta_{\lambda}}>2-(1+\alpha)-\eta(\tilde{\beta}+s).$$
Squaring both sides of this inequality, we have 
\begin{align*}
    [(1-\alpha)-\eta(\tilde{\beta}+s)]^2&<\Delta_{\lambda}\\
    &=[(1+\alpha)+\eta(\tilde{\beta}+s)]^2-4(\alpha+\eta\tilde{\beta}).
\end{align*}
After expanding the terms, we see that the inequality is equivalent to $0<4\eta s$, which is always true as $\eta > 0$. Consequently, $\lambda>1$ and in turn $\hat{\mathcal{X}}\subset \mathcal{A}_g^{\star}$. 

We have then checked the two conditions \ref{cond:a}-\ref{cond:b} above. This entails that the invariant set $\{z_1\in\R^{2d}: \lim_{k\rightarrow +\infty} g^k(z_1)\in\hat{\mathcal{X}}\}$ has Lebesgue measure zero. Equivalently, the set of initializations $x_0,x_1\in\R^d$ for which $x_k$ converges to a strict saddle point of $f$ has Lebesgue measure zero. 
%that the set of initializations $x_0,v_0\in\R^d$ for which $x_k$ converges (as $k\rightarrow +\infty$) to a strict saddle point of $f$ has Lebesgue measure zero. 
\end{enumerate}
\end{proof}

%%%%%%%%%%%%%%%%%%%%
\subsubsection{Convergence rate}
The following result provides the convergence rates for algorithm \eqref{isehd-disc} in the case where $f$ has the {\L}ojasiewicz property. The original idea of proof for descent-like algorithms can be found in \cite[Theorem~5]{attouchbolte}. 
\begin{theorem}\label{rates}
Consider the setting of Theorem~\ref{convisehddisc}, where $f$ also satisfies the {\L}ojasiewicz property with exponent $q\in [0,1[$. Then $x_k \to x_{\infty} \in \mathrm{crit}(f)$ as $k\rightarrow +\infty$ at the rates:
\begin{itemize}
\item If $q \in [0,\frac{1}{2}]$ then there exists $\rho\in ]0,1[$ such that
            \begin{equation}
                \Vert x_k-x_{\infty}\Vert=\mathcal{O}(\rho^k).
            \end{equation}
            
\item If $q \in ]\frac{1}{2},1[$ then
			\begin{equation}
            	\Vert x_k-x_{\infty}\Vert=\mathcal{O}\pa{k^{-\frac{1-q}{2q-1}}}.
            \end{equation}
\end{itemize}
\end{theorem}

\begin{proof}
Recall the function $E$ from \eqref{eq:energEdisc}. Since $f$ satisfies the {\L}ojasiewicz property with exponent $q\in [0,1[$, and $E$ is a separable quadratic perturbation of $f$, it follows from \cite[Theorem~3.3]{calckl} that $E$ has the {\L}ojasiewicz property with exponent $q_E = \max\pa{q,1/2} \in [1/2,1[$, \ie there exists $c_E>0$ such that the desingularizing function of $E$ is $\psi_E(s)=c_Es^{1-q_E}$. 

Let $v_k=x_k-x_{k-1}$ and $\Delta_k=\sum_{p=k}^{+\infty}\Vert v_{p+1}\Vert$. The triangle inequality yields $\Delta_k\geq \Vert x_k-x_{\infty}\Vert$ so it suffices to analyze the behavior of $\Delta_k$ to obtain convergence rates for the trajectory. Recall the constants $\delta,\delta_3,\delta_4>0$ and the sequences $\seq{(\Delta\psi)_k}$ and $\seq{\tilde{V}_k}$ defined in the proof of Theorem~\ref{convisehddisc}. Denote  $\lambda=\frac{\varepsilon}{\sqrt{2}-\varepsilon}\in (0,1)$ and $M=\frac{\delta_4}{\varepsilon(\sqrt{2}-\varepsilon)}$  for $0<\varepsilon<\frac{\sqrt{2}}{2}$. Using \eqref{estvk}, we have that there exists $\tilde{K}\in\N$ large enough such that for all $k\geq \tilde{K}$ 
\[
\Vert v_{k+1}\Vert\leq \lambda \Vert v_k\Vert+M(\Delta\psi)_k.
\]
Recall that $q_E \in [\frac{1}{2},1[$ (so $\frac{1-q_E}{q_E}\leq 1$) and that $\lim_{k\rightarrow +\infty} \tilde{V}_k=0$. We obtain by induction that for all $k\geq \tilde{K}$
\begin{equation}
      \sum_{p=k}^{+\infty} \Vert v_{p+1}\Vert\leq \frac{\lambda}{1-\lambda}\Vert v_{k}\Vert+\frac{Mc_E}{1-\lambda}\tilde{V}_k^{1-q_E}.
    \end{equation}
    Or equivalently, \begin{equation}\label{impp}
        \Delta_k\leq \frac{\lambda}{1-\lambda}(\Delta_{k-1}-\Delta_k)+\frac{Mc_E}{1-\lambda}\tilde{V}_k^{1-q_E} .
    \end{equation}
    
   % Recalling that $\psi(s)=c_1s^{1-q_E}$, $q_E\in ]\frac{1}{2},1)$ (so $\frac{1-q_E}{q_E}\leq 1$) and that $\tilde{V}_k\geq 0$, we have that in this case:\begin{align*}    
   % (\Delta \psi)_{k}&=c_1((\tilde{V}_k)^{1-q_E}-(\tilde{V}_{k+1})^{1-q_E})\\
    %&\leq c_1 (\tilde{V}_k)^{1-q_E}\\
    %\end{align*}
    
Denoting $c_2=(c_E(1-q_E))^{\frac{1-q_E}{q_E}}$, then by \eqref{kldiscapp} and \eqref{eq:nablaEbnddisc}
\begin{align*}
   \tilde{V}_k^{1-q_E} &\leq c_2\norm{\nabla E\left(x_k,v_k\right)}^{\frac{1-q_E}{q_E}}\\
    &\leq c_2\delta_3^{\frac{1-q_E}{q_E}}(\Vert v_k\Vert+\Vert v_{k+1}\Vert)^{\frac{1-q_E}{q_E}}\\
    &\leq  c_2\delta_3^{\frac{1-q_E}{q_E}}(\Delta_{k-1}-\Delta_{k}+\Delta_{k}-\Delta_{k+1})^{\frac{1-q_E}{q_E}} \\
    &=  c_2\delta_3^{\frac{1-q_E}{q_E}}(\Delta_{k-1}-\Delta_{k+1})^{\frac{1-q_E}{q_E}}.
    \end{align*}
Plugging this into \eqref{impp}, and using that $\Delta_k \rightarrow 0$ and $\frac{1-q_E}{q_E}\leq 1$, then there exists and integer $\tilde{K}_1 \geq \tilde{K}$ such that for all $k \geq \tilde{K}_1$ 
    \begin{equation*}
        \Delta_k\leq \frac{\lambda}{1-\lambda}(\Delta_{k-1}-\Delta_k)^{\frac{1-q_E}{q_E}}+\frac{M_1}{1-\lambda}(\Delta_{k-1}-\Delta_{k+1})^{\frac{1-q_E}{q_E}} ,
    \end{equation*}
    where $M_1=c_Ec_2\delta_3^{\frac{1-q_E}{q_E}} M$.
    Taking the power $\frac{q_E}{1-q_E}\geq 1$ on both sides and using the fact that $\Delta_{k+1}\leq\Delta_k$, we have for all $k\geq \tilde{K}_1$
     \begin{equation}\label{estimate}
        \Delta_k^{\frac{q_E}{1-q_E}}\leq M_2(\Delta_{k-1}-\Delta_{k+1}) ,
    \end{equation}
    where we set $M_2=(1-\lambda)^{-\frac{q_E}{1-q_E}}\max(\lambda,M_1)^{\frac{q_E}{1-q_E}}$, 
    We now distinguish two cases: 
        \begin{itemize}
            \item $q \in [0,1/2]$, hence $q_E=\frac{1}{2}$: \eqref{estimate} then becomes \begin{equation*}
        \Delta_k\leq M_2(\Delta_{k-1}-\Delta_{k+1}),
    \end{equation*}
    with $M_2=(1-\lambda)^{-1}\max(\lambda,M_1)$ and $M_1=\frac{c_E^2}{2}\frac{\delta_3^2}{\delta}$. Using again that $\Delta_{k+1}\leq\Delta_{k}$, we obtain that for $k\geq \tilde{K}_1$  
    \begin{equation*}
        \Delta_{k}\leq \frac{M_2}{1+M_2}\Delta_{k-2},
    \end{equation*}
    which implies 
    \begin{equation*}
        \Delta_k\leq \left(\frac{M_2}{1+M_2}\right)^{\frac{k-\tilde{K}_1}{2}}\Delta_{\tilde{K}_1}=\mathcal{O}(\rho^{k}),
    \end{equation*}
    for $\rho\eqdef \left(\frac{M_2}{1+M_2}\right)^{\frac{1}{2}}\in ]0,1[$.
      
    %\todo{Another idea is to use Taylor around $h_0=\frac{c}{L}-\beta$ to say that $$h^3\approx h_0^3-3h_0^2h+3h_0h^2.$$
    %And solve $$(a_1+3a_0h_0)h^2+(a_2-3a_0h_0^2)h+(a_3+a_0h_0^3)=0.$$
    %}
    \item $q \in ]\frac{1}{2},1[$, hence $q_E=q$: we define the function $h:\R_+^*\rightarrow\R$ by $h(s)=s^{-\frac{q}{1-q}}$. Let $R>1$. Assume first that $h(\Delta_{k})\leq Rh(\Delta_{k-1})$. Then from \eqref{estimate}, we get
    \begin{align*}
        1&\leq M_2 (\Delta_{k-1}-\Delta_{k+1})h(\Delta_k)\\
        &\leq RM_2 (\Delta_{k-1}-\Delta_{k+1})h(\Delta_{k-1})\\
            &\leq RM_2 \int_{\Delta_{k+1}}^{\Delta_{k-1}} h(s)ds\\
            &\leq RM_2\frac{1-q}{1-2q}\pa{\Delta_{k-1}^{\frac{1-2q}{1-q}}-\Delta_{k+1}^{\frac{1-2q}{1-q}}} .
        \end{align*}
    Setting $\nu=\frac{2q-1}{1-q}>0$ and $M_3=\frac{\nu}{RM_2}>0,$ one obtains 
    \begin{equation}\label{eq:Deltaknucase1}
    0<M_3\leq \Delta_{k+1}^{-\nu}-\Delta_{k-1}^{-\nu}.
    \end{equation}
    Now assume that $h(\Delta_k)>Rh(\Delta_{k-1})$. Since $h$ is decreasing and $\Delta_{k+1}\leq \Delta_k$, then $h(\Delta_{k+1}) > Rh(\Delta_{k-1})$. Set $q=R^{\frac{2q-1}{q}}> 1$, we directly have that 
    \[
    \Delta_{k+1}^{-\nu}>q\Delta_{k-1}^{-\nu}.
    \]
    Since $q-1>0$ and $\Delta_k^{-\nu}\rightarrow+\infty$ as $k\rightarrow +\infty$, there exists $M_4>0$ and a large enough integer $\tilde{K}_2 \geq \tilde{K}_1$ such that for every $k\geq \tilde{K}_2$ that satisfies our assumption ($h(\Delta_k)>Rh(\Delta_{k-1})$), we have
    \begin{equation}\label{eq:Deltaknucase2}
    0<M_4\leq \Delta_{k+1}^{-\nu}-\Delta_{k-1}^{-\nu}.
    \end{equation}
    Taking $M_5=\min(M_3,M_4)$, \eqref{eq:Deltaknucase1} and \eqref{eq:Deltaknucase2} show that for all $k\geq \tilde{K}_2$
    \[
    0<M_5\leq  \Delta_{k+1}^{-\nu}-\Delta_{k-1}^{-\nu} .
    \]
     Summing both sides from $\tilde{K}_2$ up to $K-1 \geq \tilde{K}_2$, we obtain 
     \[
     M_5(K-\tilde{K}_2)\leq \Delta_{K}^{-\nu}-\Delta_{\tilde{K}_2}^{-\nu}+\Delta_{K-1}^{-\nu}-\Delta_{\tilde{K}_2-1}^{-\nu}\leq 2(\Delta_{K}^{-\nu}-\Delta_{\tilde{K}_2-1}^{-\nu}).
     \]
    Therefore  \begin{equation}
        \Delta_{K}^{-\nu}\geq \Delta_{\tilde{K}_2-1}^{-\nu}+\frac{M_5}{2}(K-\tilde{K}_2).
    \end{equation}
    Inverting, we get 
    \begin{equation}
        \Delta_{K}\leq \left[\Delta_{\tilde{K}_2-1}^{-\nu}+\frac{M_5}{2}(K-(\tilde{K}_2+1))\right]^{-\frac{1}{\nu}}=\mathcal{O}(K^{-\frac{1}{\nu}}).
    \end{equation}
     \end{itemize}
\end{proof}

%%%%%%%%%%%%%%%%%%%%
\subsubsection{General coefficients}
The discrete scheme \eqref{isehd-disc} opens the question of whether we can consider $\alpha_k, \beta_k, s_k$ to be independent. Though this would omit the fact they arise from a discretization of the continuous-time dynamic \eqref{eq:isehd}, hence ignoring its physical interpretation, it will gives us a more flexible choice of these parameters while preserving the desired convergence behavior.

\begin{theorem}\label{convisehddiscgencoeff}
    Let $f:\R^d\rightarrow\R$ be satisfying \eqref{H0} with $\nabla f$ being globally $L$-Lipschitz-continuous. Consider $(\alpha_k)_{k\in\N},(\beta_k)_{k\in\N},(s_k)_{k\in\N}$ to be three positive sequences, and the following algorithm with $x_0,x_1\in\R^d$:
        \begin{equation}\label{isehd-disc1}
        \begin{aligned}
            \begin{cases}
                y_k&=x_k+\alpha_k(x_k-x_{k-1})-\beta_k(\nabla f(x_k)-\nabla f(x_{k-1})),\\
                x_{k+1}&=y_k-s_k\nabla f(x_k).
            \end{cases}
        \end{aligned}
    \end{equation}
    If there exists $\bar{s}$ such that: 
    \begin{itemize}
        \item $0 < \inf_{k\in\N} s_k \leq \sup_{k\in\N} s_k \leq \bar{s} < \frac{2}{L}$;
%       \item $\bar{s}<\frac{2}{L}$;
        \item $0<\sup_{k\in\N}\left(\frac{\alpha_k+\beta_kL}{s_k}\right) < \frac{1}{\bar{s}}-\frac{L}{2}$.
        %\item $\Lambda=\frac{\alpha_k+\beta_k L}{s_k}$, and $\Lambda<\frac{1}{\bar{s}}-\frac{L}{2}$.
    \end{itemize}
       Then the following holds:
       \begin{enumerate}[label=(\roman*)]
           \item $(\Vert\nabla f(x_k)\Vert)_{k\in\N}\in\ell^2(\N)$, and $(\Vert x_{k+1}-x_k\Vert)_{k\in\N}\in \ell^2(\N)$, and thus 
           \[
           \lim_{k\rightarrow +\infty} \Vert\nabla f(x_k)\Vert=0.
           \]
           \item Moreover, if $\seq{x_k}$ is bounded and $f$ is definable, then $\seq{\norm{x_{k+1}-x_{k}}}\in\ell^1(\N)$ and $x_k$ converges (as $k\rightarrow +\infty$) to a critical point of $f$.
           \item Furthermore, if $\alpha_k\equiv \alpha,\beta_k\equiv \beta,s_k\equiv s$, then the previous conditions reduce to 
           \[
           \alpha+\beta L+\frac{sL}{2}<1.
           \] 
           If, in addition, $\alpha\neq \frac{\beta}{\beta+s}$, and $\alpha> \beta L$, then for almost all $x_0, x_1\in\R^d$, $x_k$ converges (as $k\rightarrow +\infty$) to a critical point of $f$ that is not a strict saddle. Consequently, if $f$ satisfies the strict saddle property, for almost all $x_0, x_1\in\R^d$, $x_k$ converges (as $k\rightarrow +\infty$) to a local minimum of $f$.
       \end{enumerate}
\end{theorem}
%\begin{remark}
%    If $\alpha_k\equiv \alpha,\beta_k\equiv \beta,s_k\equiv s$, then the conditions to verify in the previous Theorem reduce to $$\alpha+\beta L+\frac{sL}{2}<1.$$
%\end{remark}
\begin{remark}
If $\alpha_k,\beta_k,s_k$ are given as in \eqref{isehd-disc}, \ie $\alpha_k=\frac{1}{1+\gamma_k h},\beta_k=\beta h\alpha_k, s_k=h^2 \alpha_k$, then the requirements of Theorem~\ref{convisehddiscgencoeff} reduce to $\beta+\frac{h}{2}<\frac{c}{L}$ (recall that $c$ is such that $c\leq\gamma_k$).
\end{remark}

\begin{proof}
Adjusting equation \eqref{eq2} to this setting, \ie not using the dependent explicit forms of $\alpha_k,\beta_k, s_k$, we get an analogous proof to the one of Theorem~\ref{convisehddisc}. We omit the details for the sake of brevity.
\end{proof}
%\todo{Linear Convergence rates under writing it in form $z_{k}=(x_{k}-x^{\star},x_{k-1}-x^{\star})$ and show that $z_{k+1}=Az_k+o(|z_k|)$}

%%%%%%%%%%%%%%%%%%%%%%%%%%%%%%%%%%%%%%%%%%%%%%%%%%%%%%
\section{Inertial System with Implicit Hessian Damping}\label{sec:imp}
%%%%%%%%%%%%%%%%%%%%%%%%%%%%%%%%%%%%%%%%%%%%%%%%%%%%%%

%%%%%%%%%%%%%%%%%%%%%%%%%%%%%%
\subsection{Continuous-time dynamics}
We now turn to the second-order system with implicit Hessian damping as stated in \eqref{eq:isihd}, where we consider a constant geometric damping, \ie $\beta(t) \equiv \beta>0$.
%\begin{equation}\label{H1}\tag{$\mathrm{H}_1$}
%    \begin{cases}
%        f\in C^1(\R^d);\\
%        \nabla f \text{ is locally Lipschitz};\\
%        %\mathrm{crit}(f)\neq\emptyset;\\
%        \inf f> -\infty.
%    \end{cases}
%\end{equation}
%And $\gamma:\R_+\rightarrow\R_+$ satisfying \eqref{gamma}.
% We consider the ODE \eqref{eq:isihd} with $\beta(t) \equiv \beta>0$.
%we consider the Inertial System with Implicit Hessian Damping:
%\begin{equation}\label{ISIHD}\tag{$\mathrm{ISIHD}$}
%    \begin{aligned}
%    \begin{cases}
%    \ddot{x}(t)+\gamma(t)\dot{x}(t)+\nabla f(x(t)+\beta\dot{x}(t))=0.\\
%    x(0)=x_0, \dot{x}(0)=v_0.
%    \end{cases}
%    \end{aligned}
%\end{equation}
We will use the following equivalent reformulation of \eqref{eq:isihd} proposed in \cite{hessianpert}. We will say that $x$ is a solution trajectory of \eqref{eq:isihd} with initial conditions $x(0)=x_0, \dot{x}(0)=v_0$, if and only if, $x\in C^2(\R_+;\R^d)$ and there exists $y\in C^1(\R_+;\R^d)$ such that $(x,y)$ satisfies: 
    \begin{equation}\label{refor2}
        \begin{cases}
            \dot{x}(t)+\frac{x(t)-y(t)}{\beta}&=0,\\
            \dot{y}(t)+\beta\nabla f(y(t))+\left(\frac{1}{\beta}-\gamma(t)\right)(x(t)-y(t))&=0,
        \end{cases}
    \end{equation}
    with initial conditions $x(0)=x_0, y(0)=y_0\eqdef x_0+\beta v_0.$\\

%%%%%%%%%%%%%%%%%%%%    
\subsubsection{Global convergence of the trajectory}
%In the following main theorem, we will show conditions to ensure existence of a global solution, the convergence to zero of the gradient evaluated at the solution, and then the convergence of the solution to a critical point of the objective function. 
Our next main result is the following theorem, which is the implicit counterpart of Theorem~\ref{convisehd}.
\begin{theorem}\label{convisihd}
    Let $0<\beta<\frac{2c}{C^2}$, $f:\R^d\rightarrow\R$ satisfying \eqref{H0}, $\gamma$ is continuous and satisfies \eqref{gamma}.

    Consider \eqref{eq:isihd} in this setting, then the following holds:
    
    \begin{enumerate}[label=(\roman*)]
        \item There exists a global solution trajectory $x:\R_+\rightarrow\R^d$ of \eqref{eq:isihd}. 
        \item We have that $\dot{x}\in\Lp^2(\R_+;\R^d)$, and $\nabla f\circ (x+\beta\dot{x})\in\Lp^2(\R_+;\R^d)$.
        \item \label{tresi} If we suppose that the solution trajectory $x$ is bounded over $\R_+$, then $\nabla f\circ x\in\Lp^2(\R_+;\R^d)$, $$\lim_{t\rightarrow +\infty}\Vert\nabla f(x(t))\Vert=\lim_{t\rightarrow +\infty}\Vert \dot{x}(t)\Vert=0,$$ and $\lim_{t\rightarrow +\infty} f(x(t))$ exists.
        \item In addition to \ref{tresi}, if we also assume that $f$ is definable, then $\dot{x}\in\Lp^1(\R_+;\R^d)$ and $x(t)$ converges (as $t\rightarrow +\infty$) to a critical point of $f$.
   \end{enumerate}
\end{theorem}
\begin{proof}
\begin{enumerate}[label=(\roman*)]
    \item 
We will start by showing the existence of a solution. Setting $Z=(x,y)$, \eqref{refor2} can be equivalently written as: \begin{equation}\label{refor3}
        \dot{Z}(t)+\nabla\mathcal{G}(Z(t))+\mathcal{D}(t,Z(t))=0, \quad Z(0)=(x_0,y_0),
    \end{equation}
    where $\mathcal{G}(Z):\R^d\times \R^d\rightarrow\R$ is the function defined by $\mathcal{G}(Z)=\beta f(y)$ and the time-dependent operator $\mathcal{D}:\R_+\times\R^d\times\R^d\rightarrow\R^d\times\R^d$ is given by:
    $$\mathcal{D}(t,Z)=\left(\frac{x-y}{\beta},\left(\frac{1}{\beta}-\gamma(t)\right)(x-y)\right).$$
    
    Since the map $(t,Z)\mapsto\nabla \mathcal{G}(Z)+\mathcal{D}(t,Z)$ is continuous in the first variable and locally Lipschitz in the second (by \eqref{H0} and the assumptions on $\gamma$),we get from Cauchy-Lipschitz theorem that there exists $T_{\max} > 0$ and a unique maximal solution of \eqref{refor1} denoted $Z\in C^1([0,T_{\max}[;\R^d\times\R^d)$. Consequently, there exists a unique maximal solution of \eqref{eq:isihd} $x\in C^2([0,T_{\max}[;\R^d)$. \\

Let us consider the \tcb{energy} function $V:[0,T_{\max}[\rightarrow\R$ defined by 
\[
V(t)=f(x(t)+\beta\dot{x}(t))+\frac{1}{2}\Vert \dot{x}(t)\Vert^2. 
\]
Proceeding as in the proof of Theorem~\ref{convisehd}, we prove it is indeed a Lyapunov function for \eqref{eq:isihd}. Denoting $\delta_1\eqdef \min\pa{\frac{c}{2},\beta\left(1-\frac{\beta C^2}{2c}\right)}>0$, we have
\begin{equation}\label{eqq2}
    {V}'(t)\leq -\delta_1(\Vert \dot{x}(t)\Vert^2+\Vert \nabla f(x(t)+\beta\dot{x}(t))\Vert^2).
\end{equation}
We will now show that the maximal solution $Z$ of \eqref{refor3} is actually global. For this, we argue by contradiction and assume that $T_{\max} < +\infty$. It is sufficient to prove that $x$ and $y$ have a limit as $t \to T_{\max}$, and local existence will contradict the maximality of $T_{\max}$. Integrating  \eqref{eqq2}, we obtain $\dot{x}\in\Lp^2([0,T_{\max}[;\R^d)$ and $\nabla f\circ (x+\beta\dot{x})\in\Lp^2([0,T_{\max}[;\R^d)$, which entails that $\dot{x}\in\Lp^1([0, \tcb{T_{\max}}[;\R^d)$ and $\nabla f\circ (x+\beta\dot{x})\in\Lp^1([0,T_{\max}[;\R^d)$, and in turn $\pa{x(t)}_{t\in [0,T_{\max}[}$ satisfies the Cauchy property and $\lim_{t\rightarrow T_{\max}} x(t)$ exists. Besides, by the first equation of \eqref{refor2}, we will have that $\lim_{t\rightarrow T_{\max}} y(t)$ will exist if both $\lim_{t\rightarrow T_{\max}} x(t)$ and $\lim_{t\rightarrow T_{\max}} \dot{x}(t)$ exist. So we just have to check the existence of the second limit. A sufficient condition would be to prove that $\ddot{x}\in\Lp^1([0,T_{\max}[;\R^d)$. By \eqref{eq:isihd} this will hold if $\dot{x},\nabla f\circ (x+\beta\dot{x})$ are in $\Lp^1([0,T_{\max}[;\R^d)$. But we have already shown these claims. Consequently, the solution $Z$ of \eqref{refor3} is global, and thus the solution $x$ of \eqref{eq:isihd} is also global.\\

\item Integrating \eqref{eqq2}, using that $V$ is well-defined and bounded from below, we get that $\dot{x}\in\Lp^2(\R_+;\R^d)$, and $\nabla f(x(t)+\beta \dot{x}(t))\in\Lp^2(\R_+;\R^d)$. 

\item By assumption, $\sup_{t > 0}\Vert x(t)\Vert<+\infty$. Moreover, since $\dot{x} \in \Lp^2(\R_+;\R^d)$ and continuous, $\dot{x}\in\Lp^{\infty}(\R_+;\R^d)$ and then using that $\nabla f$ is locally Lipschitz, we have
\begin{align*}
\int_{0}^{+\infty}\Vert\nabla f(x(t))\Vert^2 dt
&\leq 2 \int_{0}^{+\infty}\Vert\nabla f(x(t)+\beta \dot{x}(t))-\nabla f(x(t))\Vert^2 dt\\
&+2\int_{0}^{+\infty}\Vert\nabla f(x(t)+\beta\dot{x}(t))\Vert^2 dt\\
&\leq 2\beta^2L_0^2\int_{0}^{+\infty}\Vert\dot{x}(t)\Vert^2 dt+2\int_{0}^{+\infty}\Vert\nabla f(x(t)+\beta\dot{x}(t))\Vert^2 dt<+\infty ,
\end{align*}
where $L_0$ is the Lipschitz constant of $\nabla f$ on the centered ball of radius $$\sup_{t > 0}\norm{x(t)} + \beta\sup_{t > 0}\norm{\dot{x}(t)} < +\infty.$$ Moreover, for every $t,s \geq 0$,
\[
\norm{\nabla f(x(t)) - \nabla f(x(s))} \leq L_0 \sup_{\tau  \geq 0} \norm{\dot{x}(\tau)}|t-s| .
\]
This combined with $\nabla f\circ x\in\Lp^2(\R_+;\R^d)$ yields 
\[
\lim_{t\rightarrow +\infty}\norm{\nabla f(x(t))}=0 .
\]
We also have that 
\begin{align*}
    \sup_{t > 0}\Vert \nabla f(x(t)+\beta\dot{x}(t))\Vert
    &\leq \sup_{t > 0}(\Vert\nabla f(x(t)+\beta\dot{x}(t))-\nabla f(0) \Vert)+\Vert\nabla f(0)\Vert\\
    &\leq L_0\sup_{t > 0}\Vert x(t)\Vert + L_0 \beta\sup_{t > 0}\Vert \dot{x}(t)\Vert+\Vert\nabla f(0)\Vert<+\infty.
\end{align*}
Therefore, in view of \eqref{eq:isihd}, we get that $\ddot{x}\in \Lp^{\infty}(\R_+;\R^d)$. This implies that 
\[
\norm{\dot{x}(t) - \dot{x}(s)} \leq \sup_{\tau  \geq 0}\norm{\ddot{x}(\tau)}|t-s| .
\]
Combining this with $\dot{x}\in\Lp^2(\R_+;\R^d)$ gives that $\lim_{t\rightarrow +\infty} \norm{\dot{x}(t)}=0$.

From \eqref{eqq2}, $V$ is non-increasing, and since it is bounded from below, $V(t)$ has a limit, say $ \tilde{L}$. Passing to the limit in the definition of $V(t)$, using that the velocity vanishes, gives $\lim_{t\rightarrow +\infty} f(x(t)+\beta\dot{x}(t))=\tilde{L}$. On the other hand, we have
\begin{align*}
|f(x(t)+\beta\dot{x}(t))-f(x(t))|
&= \beta \abs{\int_0^1 \dotp{\nabla f(x(t)+s\beta\dot{x}(t)}{\dot{x}(t)} ds} \\
&\leq \beta \pa{\int_0^1 \norm{\nabla f(x(t)+s\beta\dot{x}(t)} ds} \norm{\dot{x}(t)} .
%&\leq \beta\max\{\Vert\nabla f(x(t))\Vert,\Vert \nabla f(x(t)+\beta\dot{x}(t))\Vert\}\Vert\Vert\dot{x}(t)\Vert+\frac{L\beta^2}{2}\Vert\dot{x}(t)\Vert^2.
\end{align*}
Passing to the limit as $t \to +\infty$, the right hand side goes to $0$ from the above limits on $\nabla f(x(t))$ and $\dot{x}(t)$. We deduce that $\lim_{t\rightarrow +\infty} f(x(t))=\tilde{L}$.

\item As in the proof for \eqref{eq:isehd}, since $(x(t))_{t  \geq 0}$ is bounded, then \eqref{xkb} holds. Besides, consider the function 
\begin{equation}\label{eq:energEisihd}
E: (x,v,w) \in \R^{3d} \mapsto f(x+v)+\frac{1}{2}\norm{w}^2 .
\end{equation}
Since $f$ is definable, so is $E$. In turn, $E$ satisfies has the K{\L} property. 
Let $\mathfrak{C}_1=\mathfrak{C}(x(\cdot))\times \{0_d\}\times \{0_d\}$. Since $E\big|_{\mathfrak{C}_1}=\tilde{L}, \nabla E\big|_{\mathfrak{C}_1}=0$, $\exists r,\eta>0, \exists \psi\in \kappa(0,\eta)$ such that for every $(x,v,w)\in\R^{3d}$ such that $ x\in \mathfrak{C}(x(\cdot))+B_{r},v\in B_r,w\in B_r$ and $0<E(x,v,w)-\tilde{L}<\eta$, we have 
	\begin{equation}\label{kl2}
        \psi'(E(x,v,w)-\tilde{L})\Vert \nabla E(x,v,w)\Vert\geq 1
    \end{equation}
By definition, we have $V(t)=E(x(t),\beta\dot{x}(t),\dot{x}(t)$. We also define $\tilde{V}(t)=V(t)-\tilde{L}$. By the properties of $V$ above, we have $\lim_{t\rightarrow +\infty} \tilde{V}(t)=0$ and $\tilde{V}$ is a non-increasing function. Thus $\tilde{V}(t)\geq 0$ for every $t > 0$. Without loss of generality, we may assume that $\tilde{V}(t)> 0$ for every $t > 0$ (since otherwise $\tilde{V}(t)$ is eventually zero entailing that $\dot{x}(t)$ is eventually zero in view of \eqref{eqq2}, meaning that $x(\cdot)$ has finite length).

 Define the constants $\delta_2=2, \delta_3=\frac{\delta_1}{\sqrt{2}}$. In view of the convergence claims on $\dot{x}$ and $\tilde{V}$ above, there exists $T > 0$, such that for any $t>T$
\begin{equation}
    \begin{cases}
         x(t)\in \mathfrak{C}(x(\cdot))+B_r,\\
         0<\tilde{V}(t)<\eta,\\
         \max\pa{\beta,1}\Vert \dot{x}(t)\Vert<r,\\
         \frac{1}{\delta_3}\psi( \tilde{V}(t))<\frac{r}{2\sqrt{2}}.
    \end{cases}  
\end{equation}
The rest of the proof is analogous to the one of Theorem~\ref{convisehd}. Since 
\begin{equation}
    \Vert \nabla E(x(t),\beta\dot{x}(t),\dot{x})\Vert^2\leq \delta_2(\Vert \dot{x}(t)\Vert^2+\Vert\nabla f(x(t)+\beta\dot{x}(t))\Vert^2),\end{equation}
and \begin{equation}\psi'(\tilde{V}(t))\Vert\nabla E(x(t),\beta\dot{x}(t),\dot{x})\Vert\geq 1,\quad\forall t\geq T.\end{equation}
We can lower bound the term $-\frac{d}{dt}\psi(\tilde{V}(t))$ for $t\geq T$ (as in \eqref{lowerode}) and conclude that $\dot{x}\in\Lp^1(\R_+;\R^d)$, and that this implies that $x(t)$ has finite length and thus has a limit as $t\rightarrow +\infty$. This limit is necessarily a critical point of $f$ since $\lim_{t\rightarrow +\infty}\Vert\nabla f(x(t))\Vert=0$.
\end{enumerate}
\end{proof}

%%%%%%%%%%%%%%%%%%%%
\subsubsection{Trap avoidance}
We now show that \eqref{eq:isihd} provably avoids strict saddle points, hence implying convergence to a local minimum if the objective function is Morse.

\begin{theorem}\label{morseisihd}
Let $c>0$, $0<\beta<\frac{2}{c}$ and $\gamma\equiv c$. Assume that $f:\R^d\rightarrow\R$ satisfies \eqref{H0} and is a Morse function. Consider \eqref{eq:isihd} in this setting. If the solution trajectory $x$ is bounded over $\R_+$, then the conclusions of Theorem~\ref{convisihd} hold. If, moreover, $\beta\neq\frac{1}{c}$, then for almost all $x_0,v_0\in\R^d$ initial conditions, $x(t)$ converges (as $t\rightarrow +\infty$) to a local minimum of $f$.
\end{theorem}
\begin{proof}
Since Morse functions are $C^2$ and satisfy the K{\L} inequality, and $x$ is assumed bounded, then all the claims of Theorem~\ref{convisihd} hold.

\smallskip

%When $(x(t))_{t  \geq 0}$ is bounded, the set of its cluster points $\mathfrak{C}(x(\cdot))$ satisfies \eqref{xkb}. By assumption, $f$ is a Morse function, and following the arguments in \cite[Theorem 4.1]{heavyb} we conclude with the convergence to a critical point of $f$.
%\smallskip

As in the proof of Theorem~\ref{morseisehd}, we will use again the global stable manifold theorem to prove the last point. Since, $\gamma(t)=c$ for all $t$, introducing the velocity variable $v=\dot{x}$, we have the equivalent phase-space formulation of \eqref{eq:isihd} 
    \begin{equation}\label{reforconsti}
        \begin{cases}
            \dot{x}(t)&=v(t),\\
            \dot{v}(t)&=-cv(t)-\nabla f(x(t)+\beta v(t)),
        \end{cases}
    \end{equation}
        with initial conditions $x(0)=x_0, v(0)=v_0$. Let us consider $F:\R^d\times\R^d\rightarrow\R^d\times\R^d$ defined by $$F(x,y)=\left(v, -cv -\nabla f(x+\beta v)\right).$$ 
        Defining $z(t)=(x(t),v(t))$ and $z_0=(x_0,v_0)\in\R^{2d}$ , then \eqref{reforconsti} is equivalent to  \begin{equation}\label{reforconst1i}
        \begin{cases}
            \dot{z}(t)&=F(z(t)),\\
            z(0)&=z_0 .
        \end{cases}
    \end{equation}
We know from above that under our conditions, the solution trajectory $z(t)$ converges (as $t\rightarrow +\infty$) to an equilibrium point of $F$, and the set of equilibria is $\{(\hat{x},0): \hat{x}\in\mathrm{crit}(f)\}$. Following the same ideas as in the proof of Theorem~\ref{morseisehd}, first, we will prove that each equilibrium point of $F$ is hyperbolic. We first compute the Jacobian
	\[
	J_F(x,y)=\begin{pmatrix}
            0_{d \times d} & I_d\\
            -\nabla^2 f(x+\beta v) & -c I_d-\beta\nabla^2 f(x+\beta v)
    \end{pmatrix}.
    \]
    Let $\hat{z}=(\hat{x},0)$, where $\hat{x}\in \mathrm{crit}(f)$. Then the eigenvalues of $J_F(\hat{z})$ are characterized by the solutions on $\lambda\in\mathbb{C}$ of
    \begin{equation}\label{detisehdi}
        \det\left(\begin{pmatrix}
            -\lambda I_d & I_d \\ -\nabla^2 f(\hat{x}) & -(\lambda+c)I_d-\beta\nabla^2 f(\hat{x})
        \end{pmatrix}\right)=0.
    \end{equation}
By Lemma~\ref{block}, \eqref{detisehdi} is equivalent to 
\begin{equation}\label{detisehd1i}
    \det((1+\lambda\beta)\nabla^2 f(\hat{x})+(\lambda^2+\lambda c)I_d)=0.
\end{equation}
This is the exact same equation as \eqref{detisehd1}. Thus the rest of the analysis goes as in the proof of Theorem~\ref{morseisehd}.
\end{proof}

%%%%%%%%%%%%%%%%%%%%
\subsubsection{Convergence rate}
We now give asymptotic convergence rates on the objective and trajectory.
\begin{theorem}\label{cori}
Consider the setting of Theorem~\ref{convisihd} with $f$ being also definable. Recall the function $E$ from \eqref{eq:energEisihd}, which is also definable, and denote $\psi$ its desingularizing function and $\Psi$ any primitive of $-\psi'^2$. Then, $x(t)$ converges (as $t\rightarrow +\infty$) to $x_{\infty} \in \mathrm{crit}(f)$. Denote $\tilde{V}(t)\eqdef E(x(t),\beta\dot{x}(t),\dot{x}(t))-f(x_{\infty})$. Then, the following rates of convergence hold:
\begin{itemize}
\item  If $\lim_{t\rightarrow 0}\Psi (t)\in\R$, we have $E(x(t),\dot{x}(t),\beta\nabla f(x(t)))$ converges to $f(x_{\infty})$ in finite time.

\item If $\lim_{t\rightarrow 0}\Psi (t)=+\infty$, there exists some $t_1 \geq 0$ such that \begin{equation}
        \tilde{V}(t)=\mathcal{O}(\Psi^{-1}(t-t_1))
    \end{equation}
    Moreover, 
    \begin{equation}
        \Vert x(t)-x_{\infty}\Vert=\mathcal{O}(\psi\circ\Psi^{-1}(t-t_1))
    \end{equation}
    \end{itemize}
\end{theorem}
\begin{proof}
    Analogous to Theorem~\ref{cor}.
\end{proof}

When $f$ has the {\L}ojasiewicz property, we get the following corollary of Theorem~\ref{cori}.
\begin{corollary}
Consider the setting of Theorem~\ref{cori} where now $f$ satisfies the {\L}ojasiewicz inequality with desingularizing function $\psi_f(s)=c_f s^{1-q}$, $q \in [0,1[$, $c_f > 0$. Then there exists some $t_1 > 0$ such that:
    \begin{itemize}
    \item If $q \in [0,\frac{1}{2}]$, then there exists $\mu>0$ such that:
        \begin{equation}
        \tilde{V}(t)=\mathcal{O}(\exp(-\mu(t-t_1))) \quad \text{and} \quad
    	\Vert x(t)-x_{\infty}\Vert=\mathcal{O}\left(\exp\left(-\frac{\mu(t-t_1)}{2}\right)\right)
    	\end{equation}
    \item If $q \in ]\frac{1}{2},1[$, then
    	\begin{equation}
        \tilde{V}(t)=\mathcal{O}((t-t_1)^{\frac{-1}{2q-1}}) \quad \text{and} \quad
        \Vert x(t)-x_{\infty}\Vert=\mathcal{O}((t-t_1)^{-\frac{1-q}{2q-1}})
    	\end{equation}
    \end{itemize}
    
\end{corollary}
\begin{proof}
    Analogous to Corollary \ref{corver}.
\end{proof}

%%%%%%%%%%%%%%%%%%%%%%%%%%%%%%
\subsection{Algorithmic scheme}
In this section, we will study the properties of an algorithmic scheme derived from the following explicit discretization discretization of \eqref{eq:isihd} with step-size $h>0$ and for $k\geq 1$:
\begin{equation}
    \frac{x_{k+1}-2x_k+x_{k-1}}{h^2}+\gamma(kh)\frac{x_{k+1}-x_k}{h}+\nabla f\left(x_k+\beta\frac{x_k-x_{k-1}}{h}\right)=0.
\end{equation}
This is equivalently written as 
\begin{equation}\label{isihd-disc}\tag{ISIHD-Disc}
        \begin{aligned}
            \begin{cases}
                y_k&=x_k+\alpha_k(x_k-x_{k-1}),\\
                x_{k+1}&=y_k-s_k\nabla f(x_k+\beta'(x_k-x_{k-1})),
            \end{cases}
        \end{aligned}
\end{equation}
with initial conditions $x_0,x_1 \in \R^d$, where $\gamma_k\eqdef \gamma(kh), \alpha_k\eqdef \frac{1}{1+\gamma_k h}, s_k\eqdef h^2\alpha_k$ and $\beta'\eqdef\frac{\beta}{h}$.

%%%%%%%%%%%%%%%%%%%%
\subsubsection{Global convergence and trap avoidance}
We have the following result which characterizes the asymptotic behavior of algorithm \eqref{isihd-disc}, which shows that the latter enjoys the same guarantees as \eqref{isehd-disc} given in Theorem~\ref{convisehddisc}. We will again require that $\nabla f$ is globally Lipschitz-continuous. We refer to Section \ref{sec:backtracking} for results that require only local Lipschitz continuity of $\nabla f$, where the stepsize is determined via backtracking.

\begin{theorem}\label{convisihddisc}
Let $f:\R^d\rightarrow\R$ satisfying \eqref{H0} with $\nabla f$ being globally $L$-Lipschitz-continuous. Consider algorithm \eqref{isihd-disc} with $h>0$, $\beta\geq 0$ and $c\leq\gamma_k\leq C$ for some $c,C > 0$ and for every $k\in\N$. Then the following holds:
\begin{enumerate}[label=(\roman*)]
        \item If $\beta+\frac{h}{2}<\frac{c}{L}$, then $(\Vert\nabla f(x_k)\Vert)_{k\in\N}\in \ell^2(\N)$, and $(\Vert x_{k+1}-x_k\Vert)_{k\in\N}\in \ell^2(\N)$, in particular $$\lim_{k\rightarrow +\infty} \Vert\nabla f(x_k)\Vert=0.$$
        
    \item Moreover, if $\seq{x_k}$ is bounded and $f$ is definable, then $(\Vert x_{k+1}-x_{k}\Vert)_{k\in\N}\in\ell^1(\N)$ and $x_k$ converges (as $k\rightarrow +\infty$) to a critical point of $f$.

    \item Furthermore, if $\gamma_k\equiv c>0$, $0<\beta<\frac{c}{L}, \beta\neq \frac{1}{c}$, and $h<\min\pa{2\left(\frac{c}{L}-\beta\right),\frac{1}{L\beta}}$, then for almost all $x_0, x_1 \in\R^d$, $x_k$ converges (as $k\rightarrow +\infty$) to a critical point of $f$ that is not a strict saddle. Consequently, if $f$ satisfies the strict saddle property, for almost all $x_0, x_1\in\R^d$, $x_k$ converges (as $k\rightarrow +\infty$) to a local minimum of $f$.
\end{enumerate}
\end{theorem}

%\begin{remark}
%    When $\beta=0$, we recover the Heavy Ball with Friction method and the condition $h<\min\{2\left(\frac{c}{L}-\beta\right),\frac{1}{L\beta}\}$ becomes $h<\frac{2c}{L}$.
%\end{remark}

\begin{proof}
\begin{enumerate}[label=(\roman*)]
    \item 
   Let $v_k\eqdef x_k-x_{k-1}$, $\bar{\alpha}\eqdef\frac{1}{1+ch}$, $\ubar{\alpha}\eqdef \frac{1}{1+Ch}$, $\bar{s}=h^2\bar{\alpha}, \ubar{s}=h^2\ubar{\alpha}$, so $\ubar{\alpha}\leq \alpha_k\leq\bar{\alpha}$ and $\ubar{s}\leq s_k\leq \bar{s}$ for every $k\in\N$. Proceeding as in the proof of Theorem~\ref{convisehddisc}, we have by definition that for $k\in\N^*$
   \begin{equation}
   x_{k+1}=\argmin_{x\in\R^d}\frac{1}{2}\Vert x-(y_k-s_k\nabla f(x_k+\beta'v_k))\Vert^2,
   \end{equation}
   and $1-$strong convexity of $x\mapsto \frac{1}{2}\Vert x-(y_k-s_k\nabla f(x_k+\beta'v_k))\Vert^2$ then gives
   \begin{equation}\label{eqqq1}
        \frac{1}{2}\Vert x_{k+1}-(y_k-s_k\nabla f(x_k+\beta'v_k))\Vert^2\leq \frac{1}{2}\Vert x_k-(y_k-s_k\nabla f(x_k+\beta'v_k))\Vert^2-\frac{1}{2}\Vert x_{k+1}-x_k\Vert^2.
    \end{equation}
   Expanding and rearranging, we obtain 
   \begin{equation}\label{eqqq2}
       \langle \nabla f(x_k+\beta'v_k),v_{k+1}\rangle\leq -\frac{\Vert v_{k+1}\Vert^2}{s_k}+\frac{1}{h^2}\langle v_k,v_{k+1}\rangle.
   \end{equation}
   Combining this with the descent lemma of $L$-smooth functions applied to $f$, we arrive at 
   %\begin{equation}
   %    f(x_{k+1})\leq f(x_k)+\langle\nabla f(x_k),v_{k+1}\rangle+\frac{L}{2}\Vert v_{k+1}\Vert^2.
   %\end{equation}
  % And we bound as follows 
  \begin{equation}\label{imp:isihd}
  \begin{split}
      f(x_{k+1})&\leq f(x_k)+\langle\nabla f(x_k),v_{k+1}\rangle+\frac{L}{2}\Vert v_{k+1}\Vert^2\\
       &=f(x_k)+\langle\nabla f(x_k)-\nabla f(x_k+\beta'v_k),v_{k+1}\rangle+\langle \nabla f(x_k+\beta'v_k),v_{k+1}\rangle+\frac{L}{2}\Vert v_{k+1}\Vert^2\\
       &\leq f(x_k)+\left(\beta'L+\frac{1}{h^2}\right)\Vert v_k\Vert \Vert v_{k+1}\Vert-\left(\frac{1}{\bar{s}}-\frac{L}{2}\right)\Vert v_{k+1}\Vert^2.
  \end{split}
   \end{equation}
   Where we have used that the gradient of $f$ is $L-$Lipschitz and Cauchy-Schwarz inequality in the last bound. Denote $\tilde{\alpha}=\beta'L+\frac{1}{h^2}$. We can check that since $h<\frac{2c}{L}$, then $0<\bar{s}<\frac{2}{L}$ and the last term of the inequality is negative. Using Young's inequality we have that for $\varepsilon>0:$
   \begin{align*}
       f(x_{k+1})&\leq f(x_k)+\frac{\tilde{\alpha}^2}{2\varepsilon}\Vert v_k\Vert^2+\varepsilon\frac{\Vert v_{k+1}\Vert^2}{2}-\left(\frac{1}{\bar{s}}-\frac{L}{2}\right)\Vert v_{k+1}\Vert^2.
   \end{align*}
   Or equivalently,\begin{align*}
       f(x_{k+1})+\frac{\tilde{\alpha}^2}{2\varepsilon}\Vert v_{k+1}\Vert^2\leq f(x_k)+\frac{\tilde{\alpha}^2}{2\varepsilon}\Vert v_k\Vert^2+\left[\frac{\varepsilon}{2}+\frac{\tilde{\alpha}^2}{2\varepsilon}-\left(\frac{1}{\bar{s}}-\frac{L}{2}\right)\right]\Vert v_{k+1}\Vert^2.
   \end{align*}
   In order to make the last term negative, we impose 
   \[
   \frac{\varepsilon}{2}+\frac{\tilde{\alpha}^2}{2\varepsilon}<\frac{1}{\bar{s}}-\frac{L}{2} .
   \]
   Minimizing for $\varepsilon$ at the left-hand side we obtain $\varepsilon=\tilde{\alpha}$ and the condition to satisfy is 
   \begin{equation}\label{conds}
       \bar{s}<\frac{2}{2\tilde{\alpha}+L}.
   \end{equation}
   Recalling the definitions of $\bar{s},\tilde{\alpha},\beta'$, this is equivalent to
   %Since $\frac{2}{2\tilde{\alpha}+L}<\frac{2}{L}$,  
   \[
   \frac{h^2}{1+ch}<\frac{2}{2\left(L\frac{\beta}{h}+\frac{1}{h^2}\right)+L} \iff 2L\beta h+2+Lh^2<2+2ch .
   \]
   Simplifying, this reads 
   \[
   \beta+\frac{h}{2}<\frac{c}{L},
   \]
   which is precisely what we have assumed. Let $\delta=\left(\frac{1}{\bar{s}}-\frac{L}{2}\right)-\tilde{\alpha}>0$, then  
   \begin{equation}\label{lyapisihd}
       f(x_{k+1})+\frac{\tilde{\alpha}}{2}\Vert v_{k+1}\Vert^2\leq f(x_k)+\frac{\tilde{\alpha}}{2}\Vert v_k\Vert^2-\delta\Vert v_{k+1}\Vert^2.
   \end{equation}
       
   Toward our Lyapunov analysis, define now $V_k=f(x_k)+\frac{\tilde{\alpha}}{2}\Vert v_{k}\Vert^2$ for $k\in\N^*$. In view of \eqref{lyapisihd}, $V_k$ obeys
   \begin{equation}\label{Vklyapisihd}
        V_{k+1} \leq V_k - \delta\Vert v_{k+1}\Vert^2.
   \end{equation}
   and thus $V_k$ is non-increasing. Since it is also bounded from below, $V_k$ converges to a limit, say $\tilde{L}$. Summing \eqref{lyapisihd} over $k\in\N^*$, we get that $(\Vert v_{k+1}\Vert)_{k\in\N}\in\ell^2(\N)$, hence $\lim_{k\rightarrow +\infty} \Vert v_k\Vert=0$.
   Besides, since $\bar{\alpha} < 1$ 
   \begin{align*}
       \Vert\nabla f(x_k+\beta'v_k)\Vert=\frac{1}{s_k}\Vert x_{k+1}-y_k\Vert&\leq\frac{1}{\ubar{s}}(\Vert x_{k+1}-x_k\Vert+\Vert x_k-y_k\Vert)\\
       &\leq \frac{1}{\ubar{s}}(\Vert v_{k+1}\Vert+\bar{\alpha}\Vert v_k\Vert)\\
       &\leq \frac{1}{\ubar{s}}(\Vert v_{k+1}\Vert+\Vert v_k\Vert) ,
   \end{align*}
   which implies 
   \[
   \Vert\nabla f(x_k+\beta'v_k)\Vert^2\leq \delta_2(\Vert v_{k+1}\Vert^2+\Vert v_k\Vert^2),
   \]
   where $\delta_2=\frac{2}{\ubar{s}^2}$. Consequently $(\Vert\nabla f(x_k+\beta'v_k)\Vert)_{k\in\N}\in\ell^2(\N),$ and \begin{align*}
       \Vert \nabla f(x_k)\Vert^2 &= 2(\Vert \nabla f(x_k)-\nabla f(x_k+\beta'v_k)\Vert^2+\Vert \nabla f(x_k+\beta'v_k)\Vert^2)\\
       &\leq 2(L^2\beta'^2\Vert v_k\Vert^2+\Vert \nabla f(x_k+\beta'v_k)\Vert^2).
   \end{align*}
   Thus, $(\Vert\nabla f(x_k)\Vert)_{k\in\N}\in\ell^2(\N)$, hence $\lim_{k\rightarrow +\infty} \Vert \nabla f(x_k)\Vert=0$.

   \item When $\seq{x_k}$ is bounded and $f$ is definable, we proceed analogously as in the proof of Theorem~\ref{convisehddisc} to conclude that $(\Vert v_k\Vert)_{k\in\N}\in\ell^1(\N^*)$, so $\seq{x_k}$ is a Cauchy sequence which implies that it has a limit (as $k\rightarrow +\infty$) denoted $x_{\infty}$, which is a critical point of $f$ since $\lim_{k\rightarrow +\infty} \Vert \nabla f(x_k)\Vert=0$.

\item When $\gamma_k\equiv c$, we let $\alpha_k\equiv \alpha=\frac{1}{1+ch},s_k\equiv s=h^2\alpha$. Let $z_{k}=(x_k,x_{k-1})$, and $g:\R^d\times \R^d\rightarrow \R^d\times\R^d$ defined by 
\[
g:(x_+,x_-)\mapsto [(1+\alpha)x_+-\alpha x_- -s\nabla f(x_+ +\beta'(x_+-x_-)),x_+] .
\]
\eqref{isihd-disc} is then equivalent to 
\begin{equation}\label{zkp2}
    z_{k+1}=g(z_k).
\end{equation}

To conclude, we will again use \cite[Corollary~1]{localminima}, similarly to what we did in the proof of  Theorem~\ref{convisehddisc}, by checking that: 
    \begin{enumerate}[label=(\alph*)]
        \item \label{condisihda} $\det (J_g(x_+,x_-))\neq 0 $ for every $x_+,x_-\in\R^d$.
        \item \label{condisihdb} $\hat{\mathcal{X}}\subset \mathcal{A}_g^{\star}$, where $\mathcal{A}_g^{\star}\eqdef \{(x,x)\in\R^{2d}: x\in\mathrm{crit}(f), \max_i |\lambda_i(J_g(x,x))|>1 \}$ and $\hat{\mathcal{X}}\eqdef\{(x,x)\in\R^{2d}:x\in \mathcal{X}^{\star}\}$, with $\mathcal{X}^{\star}$ the set of strict saddle points of $f$.
    \end{enumerate}

    The Jacobian $J_g(x_+,x_-)$ reads    
    \begin{equation}
        \begin{pmatrix}
        (1+\alpha)I_d-s(1+\beta')\nabla^2 f(x_+ +\beta'(x_+-x_-)) & -\alpha I_d+s\beta'\nabla^2 f(x_+ +\beta'(x_+-x_-))\\
        I_d & 0_{d\times d}
    \end{pmatrix} .
    \end{equation}
 This is a block matrix, where the bottom-left matrix commutes with the upper-left matrix (since is the identity matrix), then by Lemma~\ref{block} $\det(J_g(x_+,x_-))=\det (\alpha I_d-\beta's\nabla^2 f(x_+ +\beta'(x_+-x_-)))$. Since the eigenvalues of $\nabla^2 f(x_+ +\beta'(x_+-x_-))$ are contained in $[-L,L]$. It is then sufficient that $\alpha>\beta'Ls$ to have that $\eta-\frac{\alpha}{\beta's}\neq 0$ for every eigenvalue $\eta \neq 0$ of $\nabla^2 f(x_+ +\beta'(x_+-x_-))$. This means that under $\alpha>\beta'Ls$,  condition \ref{condisihda} is in force. Requiring $\alpha>\beta'Ls$ is equivalent to $h<\frac{1}{L\beta}$, and since we already need $h<2\left(\frac{c}{L}-\beta\right)$, we just ask $h$ to be less than the minimum of the two quantities.\\

Let us check condition \ref{condisihdb}. Let $x$ be a strict saddle point of $f$, \ie $x \in \mathrm{crit}(f)$ and $\lambda_{\min}(\nabla^2 f(x))=-\eta < 0$. To characterize the eigenvalues of $J_g(x,x)$ we could use Lemma \ref{block} as before, however, we will present an equivalent argument. Let $\eta_i \in \R$, $i=1,\ldots,d$, be the eigenvalues of $\nabla^2 f(x)$. By symmetry of the Hessian, it is easy to see that the $2d$ eigenvalues of $J_g(x,x)$ coincide with the eigenvalues of the $2 \times 2$ matrices
\[
\begin{pmatrix}
(1+\alpha)-s(1+\beta')\eta_i & -\alpha+s\beta'\eta_i \\
1 & 0
\end{pmatrix} .
\] 
These eigenvalues are therefore the (complex) roots of
\begin{equation}\label{quaddi}
\lambda^2 - \lambda\pa{(1+\alpha)-s(1+\beta')\eta_i} + \alpha - s\beta'\eta_i = 0 .
\end{equation}
%To compute the eigenvalues of $J_g(x,x)$ we consider $$\det\left(\begin{pmatrix}
%        (1+\alpha-\lambda)I_d-s(1+\beta')\nabla^2 f(x) & -\alpha I_d+s\beta'\nabla^2 f(x)\\
%        I_d & -\lambda I_d
%    \end{pmatrix}\right)=0.$$
% Again by Lemma~\ref{block}, we get that \begin{align*}
%     &\det\left(\begin{pmatrix}
%     (1+\alpha-\lambda)I_d-s(1+\beta')\nabla^2 f(x) & -\alpha I_d+s\beta'\nabla^2 f(x)\\
%     I_d & -\lambda I_d
% \end{pmatrix}\right)=\\
% &\det[(-\lambda(1+\alpha)+\lambda^2)I_d+\lambda s(1+\beta')\nabla^2 f(x)+\alpha I_d-s\beta'\nabla^2 f(x)]=\\
% &\det[s(\lambda (1+\beta')-\beta')\nabla^2 f(x)+(\lambda^2-\lambda(1+\alpha)+\alpha)I_d]
% \end{align*}
% So we want to solve for $\lambda$, \begin{equation}\label{quaddi}
%     \det[s(\lambda (1+\beta')-\beta')\nabla^2 f(x)+(\lambda^2-\lambda(1+\alpha)+\alpha)I_d]=0.
% \end{equation}
If $\lambda=\frac{\beta'}{\beta'+1}$, then \eqref{quaddi} becomes 
\[
\left(\frac{\beta'}{\beta'+1}\right)^2-\left(\frac{\beta'}{\beta'+1}\right)(1+\alpha)+\alpha=0.
\]
This implies that $\alpha=\frac{\beta'}{\beta'+1}$, or equivalently $\frac{1}{1+ch}=\frac{\beta}{\beta+h}$. But this contradicts our assumption that $\beta c\neq 1$, and thus this case cannot occur. 
%From now on, we handle the case $\lambda\neq\frac{\beta'}{\beta'+1}$ and we can rewrite \eqref{quaddi} as $$\det\left(\nabla^2 f(x)-\frac{\lambda^2-\lambda(1+\alpha)+\alpha}{s(\beta'-\lambda(1+\beta'))}I_d\right)=0.$$
% Therefore, as argued before, for every eigenvalue $\eta'\in\R$ of $\nabla^2 f(x)$, a $\lambda$ satisfying the previous equation also satisfies: $$\frac{\lambda^2-\lambda(1+\alpha)+\alpha}{s(\beta'-\lambda(1+\beta'))}=\eta'.$$
Let us now solve \eqref{quaddi} for $\eta_i=-\eta$.
%for this choice we have the equation \begin{equation}\label{quadd1i}
%     \lambda^2-\lambda[(1+\alpha)+\eta s(1+\beta')]+\alpha+\eta s\beta'=0.
% \end{equation}
%$\Delta_{\lambda}=((1+\alpha)+\eta s(1+\beta'))^2-4(\alpha+\eta\beta')$. After developing the terms we get that
Its discriminant is
\[
\Delta_{\lambda}=\alpha^2+2\alpha(\eta s(1+\beta')-1)+(\eta s(1+\beta')+1)^2-4\eta s\beta',
\]
which can be seen as a quadratic equation in $\alpha$ whose discriminant $\Delta_{\alpha}=-16\eta s$. Since $\Delta_{\alpha} < 0$ (recall that $\eta, s > 0$). Therefore the quadratic equation on $\alpha$ does not have real roots implying that $\Delta_{\lambda}>0$. We can then write the solutions of \eqref{quaddi}, \[
\lambda=\frac{\pa{(1+\alpha)+\eta s(1+\beta')} \pm \sqrt{\Delta_{\lambda}}}{2}.
\]
Let us examine the largest solution (the one with the plus sign) and show that actually $\lambda>1$. Simple algebra shows that this is equivalent to verifying that
%\[
%\frac{[(1+\alpha)+\eta s(1+\beta')]+ \sqrt{\Delta_{\lambda}}}{2}>1,
%\] 
\[
\Delta_{\lambda}>(2-(1+\alpha)-\eta s(1+\beta'))^2.
\]
%which in turn is equivalent to Taking the square at both sides of this inequality, the resulting will give us a sufficient condition to ensure $\lambda>1$. 
or, equivalently, 
\begin{align*}
    (1-\alpha)^2+\eta^2s^2(1+\beta')^2-2(1-\alpha)\eta s(1+\beta')&<\Delta_{\lambda}\\
    &=\pa{(1+\alpha)^2+\eta(s-\beta')}^2-4(\alpha-\eta\beta').
\end{align*}
Simple algebra again shows that this inequality is equivalent to  $0<4\eta s$, which is always true as $\eta > 0$. We have thus shown that $\hat{\mathcal{X}}\subset \mathcal{A}_g^{\star}$.

Overall, we have checked the two conditions \ref{condisihda}-\ref{condisihdb} above. Therefore the invariant set $\{z_1\in\R^{2d}: \lim_{k\rightarrow +\infty} g^k(z_1)\in\hat{\mathcal{X}}\}$ has Lebesgue measure zero. This means that the set of initializations $x_0,x_1\in\R^d$ for which $x_k$ converges to a strict saddle point of $f$ has Lebesgue measure zero.
\end{enumerate}
\end{proof}

%%%%%%%%%%%%%%%%%%%%
\subsubsection{Convergence rate}
The asymptotic convergence rate of algorithm \eqref{isihd-disc} for {\L}ojasiewicz functions is given in the following theorem. This shows that \eqref{isihd-disc} enjoys the same asymptotic convergence rates as \eqref{isehd-disc}.
\begin{theorem}\label{ratesi}
Consider the setting of Theorem~\ref{convisihddisc}, where $f$ also satisfies the {\L}ojasiewicz property with exponent $q\in [0,1[$.  Then $x_k \to x_{\infty} \in \mathrm{crit}(f)$ as $k\rightarrow +\infty$ at the rates:
\begin{itemize}
\item If $q \in [0,\frac{1}{2}]$ then there exists $\rho\in ]0,1[$ such that
            \begin{equation}
                \Vert x_k-x_{\infty}\Vert=\mathcal{O}(\rho^k).
            \end{equation}
            
\item If $q \in ]\frac{1}{2},1[$ then
			\begin{equation}
            	\Vert x_k-x_{\infty}\Vert=\mathcal{O}\pa{k^{-\frac{1-q}{2q-1}}}.
            \end{equation}
\end{itemize}
\end{theorem}
\begin{proof}
Since the Lyapunov analysis of \eqref{isihd-disc} is analogous to that of \eqref{isehd-disc} (though the Lyapunov functions are different), the proof of this theorem is similar to the one of Theorem~\ref{rates}.
\end{proof}

%%%%%%%%%%%%%%%%%%%%
\subsubsection{General coefficients}
As discussed for the explicit case, the discrete scheme \eqref{isihd-disc} rises from a discretization of the ODE \eqref{eq:isihd}. However, the parameters $\alpha_k, s_k$ are linked to each other. We now consider \eqref{isihd-disc} where $\alpha_k, s_k$ are independent. Though this would hide somehow the physical interpretation of these parameters, it allows for some flexibility in their choice while preserving the  convergence behavior.

\begin{theorem}
    Let $f:\R^d\rightarrow\R$ be satisfying \eqref{H0} with $\nabla f$ being globally $L$-Lipschitz-continuous. Consider $(\alpha_k)_{k\in\N},(\beta_k)_{k\in\N},(s_k)_{k\in\N}$ to be three positive sequences, and the following algorithm with $x_0,x_1\in\R^d$:
        \begin{equation}\label{isihd-disc1}
        \begin{aligned}
            \begin{cases}
                y_k&=x_k+\alpha_k(x_k-x_{k-1}),\\
                x_{k+1}&=y_k-s_k\nabla f(x_k+\beta_k(x_k-x_{k-1})).
            \end{cases}
        \end{aligned}
    \end{equation}
    If there exists $\bar{s}>0$ such that: 
    \begin{itemize}
        \item $0 < \inf_{k\in\N} s_k \leq \sup_{k\in\N} s_k\leq \bar{s}<\frac{2}{L}$;
        \item $0<\sup_{k\in\N}\left(\beta_kL+\frac{\alpha_k}{s_k}\right) < \frac{1}{\bar{s}}-\frac{L}{2}$.
    \end{itemize}
       Then the following holds:
       \begin{enumerate}[label=(\roman*)]
           \item  $(\Vert\nabla f(x_k)\Vert)_{k\in\N}\in\ell^2(\N)$, and $(\Vert x_{k+1}-x_k\Vert)_{k\in\N}\in \ell^2(\N)$, hence 
           \[
           \lim_{k\rightarrow +\infty} \Vert\nabla f(x_k)\Vert=0.
           \]
           \item Moreover, if $\seq{x_k}$ is bounded and $f$ is definable, then $(\Vert x_{k+1}-x_{k}\Vert)_{k\in\N}\in\ell^1(\N)$ and $x_k$ converges (as $k\rightarrow +\infty$) to a critical point of $f$. 
       \item Furthermore, if $\alpha_k\equiv \alpha,\beta_k\equiv \beta,s_k\equiv s$, then the previous conditions reduce to 
    \[
	\alpha+sL\left(\beta+\frac{1}{2}\right)<1.
	\] 
	If, in addition, $\alpha\neq\frac{\beta}{\beta+1}$, $\alpha> \beta L s$, then for almost all $x_0,x_1\in\R^d$, $x_k$ converges (as $k\rightarrow +\infty$) to a critical point of $f$ that is not a strict saddle. Consequently, if $f$ satisfies the strict saddle property, for almost all $x_0, x_1\in\R^d$, $x_k$ converges (as $k\rightarrow +\infty$) to a local minimum of $f$.
       \end{enumerate}
\end{theorem}
%\begin{remark}
%    If $\alpha_k\equiv \alpha,\beta_k\equiv \beta,s_k\equiv s$, then the conditions to verify in the previous Theorem reduce to $$\alpha+sL\left(\beta+\frac{1}{2}\right)<1.$$
%\end{remark}
%\begin{remark}
%    If $\alpha_k,\beta_k,s_k$ are given as in \eqref{isihd-disc}, \ie $\alpha_k=\frac{1}{1+\gamma_k h},\beta_k\equiv\frac{\beta}{h}, s_k=h^2 \alpha_k$, then the first four requirements of the previous Theorem reduce to $\beta+\frac{h}{2}<\frac{c}{L}$ (recall that $c$ is such that $c\leq\gamma_k$).
%\end{remark}
\begin{proof}
Adjusting equation \eqref{eqqq2} to our setting (\ie not using the dependency of $\alpha_k, s_k$) we get an analogous proof to the one of Theorem~\ref{convisihddisc}. We omit the details.
\end{proof}

\section{Adaptive Stepsize via Backtracking}\label{sec:backtracking}
To remove the global Lipschitz assumption on the gradient of $f$—which is often unrealistic in practical settings such as machine learning—or even improve the estimate of the Lipschitz constant, we propose the following discrete schemes, inspired by \eqref{isehd-disc1}, \eqref{isihd-disc1}, but now incorporating an adaptive stepsize via a backtracking procedure that allows us to leverage techniques relying only on the local Lipschitz continuity of $\nabla f$.\\

\subsection{Explicit Hessian Damping}\label{sec:explicitbacktracking}
Let $f:\R^d\rightarrow\R$ be lower bounded, differentiable, and $\nabla f$ locally Lipschitz continuous. Consider $\alpha,\beta:\R_+\rightarrow\R_+$, such that $\lim_{x\rightarrow 0^+} \alpha(x)+\beta(x)=0$, $s_0>0, \delta\in ]0,2[, \rho\in ]0,1[$, and the following algorithm with initial data $x_0,x_1\in\R^d$:
        \begin{equation}\label{isehd-disc-back}
        \begin{aligned}
            \begin{cases}
                y_k&=x_k+\alpha(s_k)(x_k-x_{k-1})-\beta(s_k)(\nabla f(x_k)-\nabla f(x_{k-1})),\\
                x_{k+1}&=y_k-s_k\nabla f(x_k),
            \end{cases}
        \end{aligned}
    \end{equation}
where $s_k=\rho^{i_k}s_0$, and $i_k$ is the smallest nonnegative integer such that
\begin{align}\label{cond1back}
    f(x_{k+1})-f(x_k)-\langle \nabla f(x_k),x_{k+1}-x_k\rangle\leq \frac{\delta}{2s_k}\Vert x_{k+1}-x_k\Vert^2,
\end{align}
and 
\begin{align}\label{cond2back}
    \Vert \nabla f(x_{k+1})-\nabla f(x_k)\Vert\leq \frac{\delta}{s_k}\Vert x_{k+1}-x_k\Vert.
\end{align}
The following lemma is an adaptation of \cite[Lemma 1.4.3]{implicitreg}, where the authors study the finite termination of \eqref{isehd-disc-back} in the case $\alpha(x) = \alpha_0 x$ and $\beta(x) = \beta_0 x^2$. For completeness, we include the proof below.

\begin{lemma}\label{finiteter}
    The backtracking procedure in \eqref{isehd-disc-back} terminates in a finite number of iterations and $\bar{s}=\sup_{k\in\N} s_k\leq s_0$. Moreover, if $(x_k)_{k\in\N}$ is bounded, then $\ubar{s}=\inf_{k\in\N}s_k >0$.
\end{lemma}
\begin{proof}
Denote the Bregman divergence of $f$ as $$D_f(\tilde{x},x)\eqdef f(\tilde{x})-f(x)-\langle\nabla f(x),\tilde{x}-x\rangle.$$
    We write generically each iteration of Algorithm \eqref{isehd-disc-back} as $$x^+(s_i)\eqdef x+\alpha(s_i)(x-x_{-})-\beta(s_i)(\nabla f(x)-\nabla f(x_{-}))-s_i\nabla f(x), \quad \forall i\in\N,$$
    where $s_i=\rho^i s_0$. Since $\lim_{x\rightarrow 0^+} \alpha(x)+\beta(x)=0$ and $\rho<1$, we have that $\lim_{i\rightarrow\infty}x^+(s_i)=x$. Thus $\forall \varepsilon>0, \exists i_{\varepsilon}>0$ such that $x^+(s_i)\subset B(x,\varepsilon), \forall i\geq i_{\varepsilon}$. It follows from the local Lipschitz continuity of $\nabla f$ and the descent lemma that $\exists L_{\varepsilon}>0$ such that $\forall i\geq i_{\varepsilon}$,
    \begin{align}\label{locallip}
        \Vert \nabla f(x^+(s_i))-\nabla f(x)\Vert\leq L_{\varepsilon}\Vert x^+(s_i)-x\Vert, \quad \text{and} \quad D_f(x^+(s_i),x)\leq \frac{L_{\varepsilon}}{2}\Vert x^+(s_i)-x\Vert^2.
    \end{align}
    If we assume by contradiction that the backtracking procedure does not terminate, this would imply that for all $i\geq 0$,
    \begin{align*}
        s_i\Vert \nabla f(x^+(s_i))-\nabla f(x)\Vert>\delta\Vert x^+(s_i)-x\Vert, \quad \text{or} \quad s_i D_f(x^+(s_i),x)>\frac{\delta}{2}\Vert x^+(s_i)-x\Vert^2.
    \end{align*}
    This together with \eqref{locallip} gives us that for all $i\geq i_{\varepsilon}$,
    \begin{align*}
        \delta \Vert x^+(s_i)-x\Vert<s_i\Vert \nabla f(x^+(s_i))&-\nabla f(x)\Vert\leq s_iL_{\varepsilon}\Vert  x^+(s_i)-x\Vert, \\
        &\text{or}\\
        \frac{\delta}{2} \Vert x^+(s_i)-x\Vert^2<s_i D_f(x^+&(s_i),x)\leq \frac{s_iL_{\varepsilon}}{2}\Vert  x^+(s_i)-x\Vert^2.
    \end{align*}
    In both cases, we obtain that $\delta< s_iL_{\varepsilon}$. Passing to the limit as $i\rightarrow\infty$ yields $\delta=0$, a contradiction.\\
    The fact that $\bar{s}\leq s_0$ is direct. We now show that $\ubar{s}>0$ when the iterates are bounded. There exists a convex bounded set $\Omega$, such that $(x_k)_{k\in\N}\subset \Omega$. The descent lemma implies that there exists $L_{\Omega}>0$ such that for all $k\in\N$,
    \begin{equation}\label{backlocal}
        D_f(x_{k+1},x_k)\leq \frac{L_{\Omega}}{2}\Vert x_{k+1}-x_k\Vert^2.
    \end{equation}
    We can show by induction that for all $k\in\N$,
    \begin{equation}\label{indback}
        s_k\geq \min\left(s_0,\frac{\rho\delta}{L_{\Omega}} \right)>0.
    \end{equation}
    It is obviously true for $k=0$. Assume that \eqref{indback} holds for some $k\geq 1$, if $i_{k+1}\leq i_k$, then $s_{k+1}\geq s_k$ and we could conclude. If $i_{k+1}\geq i_k+1$, we suppose by contradiction that $s_{k+1}<\min\left(s_0,\frac{\rho\delta}{L_{\Omega}}\right)$, this implies that $L_{\Omega}<\frac{\delta}{\rho^{i_{k+1}-1}s_0}$, by \eqref{backlocal} and the previous observation, we get that \begin{align*}
        D_f(x_{k+2},x_{k+1})\leq \frac{L_{\Omega}}{2}\Vert x_{k+2}-x_{k+1}\Vert^2< \frac{\delta}{2\rho^{i_{k+1}-1}s_0}\Vert x_{k+2}-x_{k+1}\Vert^2.%\leq \frac{\delta}{2\rho^{i_{k+1}}s_0}\Vert x_{k+2}-x_{k+1}\Vert^2.
    \end{align*}
    Which implies that the backtracking procedure ends at or before iteration $i_{k+1}-1$, which contradicts the fact that the procedure was supposed to terminate at $i_{k+1}$. This concludes the proof.
\end{proof}
\begin{theorem}\label{convisehddiscgencoeff-back}
     Consider the algorithm \eqref{isehd-disc-back} under the previously described  setting. Let $\bar{s}=\sup_{k\in\N} s_k\leq s_0$. If 
\begin{align}\label{conditionback}
0<\sup_{k\in\N}\left(\frac{\alpha(s_k)}{s_k}+\frac{\beta(s_k)\delta}{s_k^2}\right) < \frac{1}{\bar{s}}\left(1-\frac{\delta}{2}\right),
\end{align}
       then the following holds:
       \begin{enumerate}[label=(\roman*)]
           \item $(\Vert\nabla f(x_k)\Vert)_{k\in\N}\in\ell^2(\N)$, and $(\Vert x_{k+1}-x_k\Vert)_{k\in\N}\in \ell^2(\N)$, and thus 
           \[
           \lim_{k\rightarrow +\infty} \Vert\nabla f(x_k)\Vert=0.
           \]
           \item Moreover, if the sequence $ \seq{x_k} $ is bounded, then $ 0 < \inf_{k \in \N} s_k $; and if, in addition, $ f $ is definable, then $ \seq{\|x_{k+1} - x_k\|} \in \ell^1(\N) $, and $ x_k $ converges (as $ k \to +\infty $) to a critical point of $ f $, denoted by $ x_\infty $.

           \item Furthermore, if $f$ satisfies the {\L}ojasiewicz property with exponent $q\in [0,1[$, then the convergence of $x_k$ converges to $x_\infty$ at the rates:
           \begin{itemize}
            \item If $q \in [0,\frac{1}{2}]$ then there exists $\rho_0\in ]0,1[$ such that
            \begin{equation*}
                \Vert x_k-x_{\infty}\Vert=\mathcal{O}(\rho_0^k).
            \end{equation*}
            
            \item If $q \in ]\frac{1}{2},1[$ then
			\begin{equation*}
            	\Vert x_k-x_{\infty}\Vert=\mathcal{O}\pa{k^{-\frac{1-q}{2q-1}}}.
            \end{equation*}
\end{itemize}
           %\item Furthermore, if $\alpha_k\equiv \alpha,\beta_k\equiv \beta,s_k\equiv s$, then the previous conditions reduce to 
           %\[
           %\alpha+\beta L+\frac{sL}{2}<1.
           %\] 
           %If, in addition, $\alpha\neq \frac{\beta}{\beta+s}$, and $\alpha> \beta L$, then for almost all $x_0, x_1\in\R^d$, $x_k$ converges (as $k\rightarrow +\infty$) to a critical point of $f$ that is not a strict saddle. Consequently, if $f$ satisfies the strict saddle property, for almost all $x_0, x_1\in\R^d$, $x_k$ converges (as $k\rightarrow +\infty$) to a local minimum of $f$.
       \end{enumerate}
\end{theorem}

\begin{remark}\label{notcircular}
If $\alpha(x)=\alpha_0 x,\beta(x)=\beta_0 x^2$  then the condition \eqref{conditionback} reduce to $0<\alpha_0+\beta_0\delta<\frac{1}{\bar{s}}\left(1-\frac{\delta}{2}\right)$. In general, checking the condition \eqref{conditionback} directly can lead to a circular reasoning issue, since the sequence $s_k$ depends on the iterates $x_k$, and vice versa, while conclusions are drawn about the iterates themselves. However, we observe that it is sufficient to verify the condition
$$
0 < \sup_{x \in [0, s_0]} \frac{\alpha(x)}{x} + \frac{\beta(x) \delta}{x^2} < \frac{1}{\bar{s}} \left( 1 - \frac{\delta}{2} \right),
$$
where $\lim_{x\rightarrow 0^+}\alpha(x)+\beta(x)=0$, which is considerably simpler to verify in practice.

\end{remark}

\begin{proof}
Let $\Lambda_k=\frac{\alpha(s_k)}{s_k}+\frac{\beta(s_k)\delta}{s_k^2}$ and $\bar{\Lambda}=\sup_{k\in\N}\Lambda_k$. Adjusting equation \eqref{eq2} to this setting and using \eqref{cond1back}, we obtain \begin{align*}
    f(x_k)\leq f(x_{k+1})+\frac{\alpha(s_k)}{s_k}\langle v_k,v_{k+1}\rangle-\frac{\beta(s_k)}{s_k}\langle v_{k+1},\nabla f(x_k)-\nabla f(x_{k-1})\rangle-\frac{1}{\bar{s}}\left(1-\frac{\delta}{2}\right)\Vert v_{k+1}\Vert^2.
\end{align*}
Using Cauchy-Schwarz inequality and \eqref{cond2back} we get
\begin{align*}
    f(x_{k+1})\leq f(x_{k})+\Lambda_k\Vert v_k\Vert\Vert v_{k+1}\Vert-\frac{1}{\bar{s}}\left(1-\frac{\delta}{2}\right)\Vert v_{k+1}\Vert^2.
\end{align*}
Applying Young's inequality, we end up with
\begin{align*}
    f(x_{k+1})\leq f(x_{k})+\frac{\Lambda_k^2\Vert v_k\Vert^2}{2\varepsilon_k}+\varepsilon_k\frac{\Vert v_{k+1}\Vert^2}{2}-\frac{1}{\bar{s}}\left(1-\frac{\delta}{2}\right)\Vert v_{k+1}\Vert^2.
\end{align*}
Choosing $\varepsilon_k=\Lambda_k$ and adding the term $\frac{\bar{\Lambda}}{2}\Vert v_{k+1}\Vert^2$ to both sides of the inequality, 
\begin{align*}
    f(x_{k+1})+\frac{\bar{\Lambda}}{2}\Vert v_{k+1}\Vert^2&\leq f(x_{k})+\frac{\Lambda_k}{2}\Vert v_k\Vert^2+\left[\frac{\bar{\Lambda}}{2}+\frac{\Lambda_k}{2}-\frac{1}{\bar{s}}\left(1-\frac{\delta}{2}\right)\right]\Vert v_{k+1}\Vert^2\\
    &\leq f(x_{k})+\frac{\bar{\Lambda}}{2}\Vert v_k\Vert^2+\left[\bar{\Lambda}-\frac{1}{\bar{s}}\left(1-\frac{\delta}{2}\right)\right]\Vert v_{k+1}\Vert^2.
\end{align*}
And we can conclude the first two points as in the proof of Theorem \ref{convisehddisc}. For the convergence rates, we follow the exact same strategy as in the proof of Theorem \ref{rates}, therefore, we omit the proof.

\end{proof}

\subsection{Implicit Hessian Damping}
Let $f:\R^d\rightarrow\R$ be lower bounded, differentiable, and $\nabla f$ locally Lipschitz continuous. Consider $\alpha,\beta:\R_+\rightarrow\R_+$, such that $\lim_{x\rightarrow 0^+} \alpha(x)=0$, $s_0>0, \delta\in ]0,2[, \rho\in ]0,1[$, and the following algorithm with initial data $x_0,x_1\in\R^d$:
        \begin{equation}\label{isihd-disc-back}
        \begin{aligned}
            \begin{cases}
                y_k&=x_k+\alpha(s_k)(x_k-x_{k-1}),\\
                x_{k+1}&=y_k-s_k\nabla f(x_k+\beta(s_k)(x_k-x_{k-1})),
            \end{cases}
        \end{aligned}
    \end{equation}
where $s_k=\rho^{i_k}s_0$, and $i_k$ is the smallest nonnegative integer such that
\begin{align}\label{cond1backimp}
    f(x_{k+1})-f(x_k)-\langle \nabla f(x_k),x_{k+1}-x_k\rangle\leq \frac{\delta}{2s_k}\Vert x_{k+1}-x_k\Vert^2,
\end{align}
and 
\begin{align}\label{cond2backimp}
    \Vert \nabla f(x_{k+1})-\nabla f(x_k)\Vert\leq \frac{\delta}{s_k}\Vert x_{k+1}-x_k\Vert.
\end{align}

\begin{lemma}
     The backtracking procedure in \eqref{isihd-disc-back} terminates in a finite number of iterations and $\bar{s}=\sup_{k\in\N} s_k\leq s_0$. Moreover, if $(x_k)_{k\in\N}$ is bounded, then $\ubar{s}=\inf_{k\in\N}s_k >0$.
\end{lemma}
\begin{proof}
   The proof follows the same spirit as that of Lemma~\ref{finiteter}; therefore, we omit it.
\end{proof}

\begin{theorem}\label{convisihddiscgencoeff-back}
     Consider the algorithm \eqref{isihd-disc-back} under the previously described setting. Let $\bar{s}=\sup_{k\in\N} s_k\leq s_0$. If 
\begin{align}\label{conditionback1}
0<\sup_{k\in\N}\frac{\alpha(s_k)+\beta(s_k)\delta}{s_k}< \frac{1}{\bar{s}}\left(1-\frac{\delta}{2}\right),
\end{align}
       then the following holds:
       \begin{enumerate}[label=(\roman*)]
           \item $(\Vert\nabla f(x_k)\Vert)_{k\in\N}\in\ell^2(\N)$, and $(\Vert x_{k+1}-x_k\Vert)_{k\in\N}\in \ell^2(\N)$, and thus 
           \[
           \lim_{k\rightarrow +\infty} \Vert\nabla f(x_k)\Vert=0.
           \]
           \item Moreover, if the sequence $ \seq{x_k} $ is bounded, then $ 0 < \inf_{k \in \N} s_k $; and if, in addition, $ f $ is definable, then $ \seq{\|x_{k+1} - x_k\|} \in \ell^1(\N) $, and $ x_k $ converges (as $ k \to +\infty $) to a critical point of $ f $, denoted by $ x_\infty $.

           \item Furthermore, if $f$ satisfies the {\L}ojasiewicz property with exponent $q\in [0,1[$, then the convergence of $x_k$ converges to $x_\infty$ at the rates:
           \begin{itemize}
            \item If $q \in [0,\frac{1}{2}]$ then there exists $\rho_0\in ]0,1[$ such that
            \begin{equation*}
                \Vert x_k-x_{\infty}\Vert=\mathcal{O}(\rho_0^k).
            \end{equation*}
            
            \item If $q \in ]\frac{1}{2},1[$ then
			\begin{equation*}
            	\Vert x_k-x_{\infty}\Vert=\mathcal{O}\pa{k^{-\frac{1-q}{2q-1}}}.
            \end{equation*}
\end{itemize}
       \end{enumerate}
\end{theorem}
\begin{remark}
If $\alpha(x)=\alpha_0 x,\beta(x)=\beta_0 x$  then the condition \eqref{conditionback1} reduce to $0<\alpha_0+\beta_0\delta<\frac{1}{\bar{s}}\left(1-\frac{\delta}{2}\right)$. As in Remark \ref{notcircular}, we observe that to verify the condition \eqref{conditionback1}, it suffices to check
$$
0 < \sup_{x \in [0, s_0]} \frac{\alpha(x)+\beta(x)\delta}{x} < \frac{1}{\bar{s}} \left( 1 - \frac{\delta}{2} \right),
$$
where $\lim_{x\rightarrow 0^+}\alpha(x)=0$, which is considerably simpler to verify in practice.
\end{remark}
\begin{proof}
    Let $\Lambda_k=\frac{\alpha(s_k)+\beta(s_k)\delta}{s_k}$ and $\bar{\Lambda}=\sup_{k\in\N} \Lambda_k$. Adjusting the equation \eqref{imp:isihd} to our setting we get \begin{align*}
        f(x_{k+1})\leq f(x_k)+\Lambda_k\Vert v_k\Vert\Vert v_{k+1}\Vert-\frac{1}{\bar{s}}\left(1-\frac{\delta}{2}\right)\Vert v_{k+1}\Vert^2.
    \end{align*}
    Using Young's inequality we get \begin{align*}
        f(x_{k+1})\leq f(x_k)+\frac{\Lambda_k^2}{2\varepsilon_k}\Vert v_k\Vert^2+\frac{\varepsilon_k}{2}\Vert v_{k+1}
        \Vert^2-\frac{1}{\bar{s}}\left(1-\frac{\delta}{2}\right)\Vert v_{k+1}\Vert^2.
    \end{align*}
    Choosing $\varepsilon_k=\Lambda_k$ and adding the term $\frac{\bar{\Lambda}}{2}\Vert v_{k+1}\Vert^2$ at both sides of the inequality we obtain

\begin{align*}
        f(x_{k+1})+\frac{\bar{\Lambda}}{2}\Vert v_{k+1}\Vert^2&\leq f(x_k)+\frac{\Lambda_k}{2}\Vert v_k\Vert^2+\left(\frac{\Lambda_k}{2}+\frac{\bar{\Lambda}}{2}-\frac{1}{\bar{s}}\left(1-\frac{\delta}{2}\right)\right)\Vert v_{k+1}\Vert^2\\
        &\leq f(x_k)+\frac{\bar{\Lambda}}{2}\Vert v_k\Vert^2+\left(\bar{\Lambda}-\frac{1}{\bar{s}}\left(1-\frac{\delta}{2}\right)\right)\Vert v_{k+1}\Vert^2.
    \end{align*}
And we conclude as in the proof of Theorem \ref{convisehddisc}. The convergence rates follow the strategy presented in the proof of Theorem \ref{rates} and we omit the proof.
\end{proof}

%%%%%%%%%%%%%%%%%%%%%%%%%%%%%%%%%%%%%%%%%%%%%%%%%%%%%%
\section{Numerical experiments}\label{sec:num}
%%%%%%%%%%%%%%%%%%%%%%%%%%%%%%%%%%%%%%%%%%%%%%%%%%%%%%

Before describing the numerical experiments, let us start with a few observations on the computational complexity and memory storage requirement of \eqref{isehd-disc} and \eqref{isihd-disc}. 
The number of gradient access per iteration is the same for Gradient Descent (GD), discrete HBF, \eqref{isehd-disc} and \eqref{isihd-disc} is the same (one per iteration). However, the faster convergence (in practice) of inertial methods comes at the cost of storing previous information. For the memory storage requirement per iteration, GD stores only the previous iterate, the discrete HBF and \eqref{isihd-disc} store the two previous iterates, while \eqref{isehd-disc} additionally stores the previous gradient iterate as well. This has to be kept in mind when comparing these algorithms especially for in very high dimensional settings.\\

We will illustrate our findings with two numerical experiments. The first one is the optimization of the Rosenbrock function in $\R^2$, while the second one is on image deblurring. We will apply the proposed discrete schemes \eqref{isehd-disc} and \eqref{isihd-disc} and compare them with gradient descent and the (discrete) HBF. We will call $\Vert\nabla f(x_k)\Vert$ the residual.

%%%%%%%%%%%%%%%%%%%%%%%%%%%%%%
\subsection{Rosenbrock function}
%\todo{The gradient of the Rosenbrock function is not globally Lipschitz continuous ($f$ is quartic in $x$). Say that it is algebraic (hence K{\L}). Can we show it has trap avoidance ?}
We will minimize the classical Rosenbrock function, \ie,
\[
f: (x,y) \in \R^2 \mapsto (1-x)^2+100(y-x^2)^2,
\]
with global minimum at $(x^{\star},y^{\star})=(1,1)$.

\smallskip

We notice that its global minimum is the only critical point. Therefore this function is Morse and thus satisfies the {\L}ojasiewicz inequality with exponent $\frac{1}{2}$ (see Remark~\ref{rem}). Consider \eqref{eq:isehd} and \eqref{eq:isihd} with the Rosenbrock function as the objective, $\gamma:\R_+\rightarrow\R_+$ satisfying \eqref{gamma} (\ie $0<c\leq \gamma(t)\leq C<+\infty$), and $0<\beta<\frac{2c}{C^2}$. By Theorems~\ref{convisehd}-\ref{rates} for \eqref{eq:isehd}, and Theorems~\ref{convisihd}-\ref{ratesi} for \eqref{eq:isihd}, we get that the solution trajectories of these dynamics will converge to the global minimum eventually at a linear rate. Due to the low dimensionality of this problem, we could use an ODE solver to show numerically these results. However, we will just the iterates generated by our proposed algorithmic schemes \eqref{isehd-disc} and \eqref{isihd-disc}. Although the gradient of the objective is not globally Lipschitz continuous, our proposed algorithmic schemes worked very well for $h$ small enough. This suggests that we may relax this hypothesis in future work, as proposed in \cite{panageas, cedric} for GD. 

\smallskip

We applied \eqref{isehd-disc} and \eqref{isihd-disc} with $\beta \in \{0.02, 0.04\}$, $\gamma(t)\equiv \gamma_0=3$, $h=10^{-3}$ and initial conditions $x_0=(-1.5,0), x_1=x_0$. We compared our algorithms with GD and HBF (with the same initial conditions) after $2*10^4$ iterations:
\begin{equation}\label{ngd}\tag{GD}
    x_{k+1}=x_k-\frac{h^2}{1+\gamma_0 h}\nabla f(x_k), 
\end{equation}
and
\begin{equation}\tag{HBF}\label{nhbf}
        \begin{aligned}
            \begin{cases}
                y_k&=x_k+\frac{1}{1+\gamma_0 h}(x_k-x_{k-1}),\\
                x_{k+1}&=y_k-\frac{h^2}{1+\gamma_0 h}\nabla f(x_k).
            \end{cases}
        \end{aligned}
    \end{equation}
The behavior of all algorithms is depicted in Figures~\ref{rose0} and \ref{rose}.
\begin{figure}[H]
    \centering
    \begin{subfigure}[b]{0.32\textwidth}
         \centering
         \includegraphics[width=\textwidth]{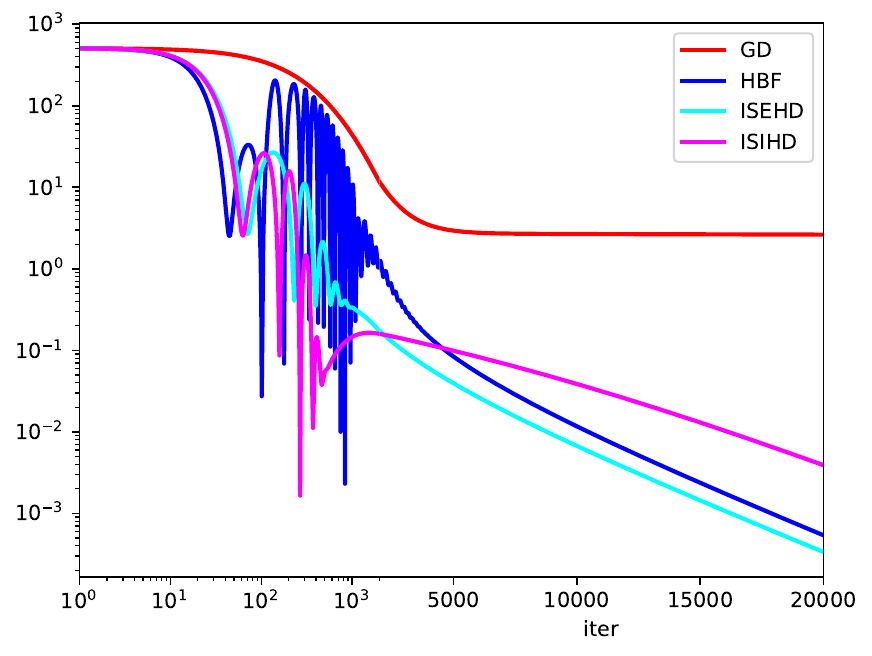}
         \caption{Residual vs Iteration.}
     \end{subfigure}
    \begin{subfigure}[b]{0.32\textwidth}
         \centering
         \includegraphics[width=\textwidth]{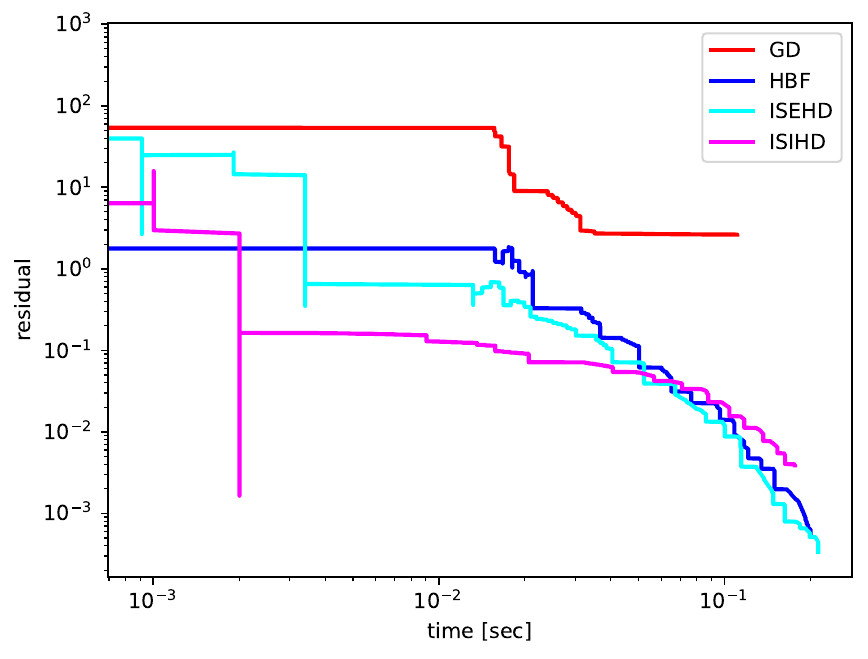}
         \caption{Residual vs Time.}
     \end{subfigure}
     \begin{subfigure}[b]{0.32\textwidth}
         \centering
         \includegraphics[width=\textwidth]{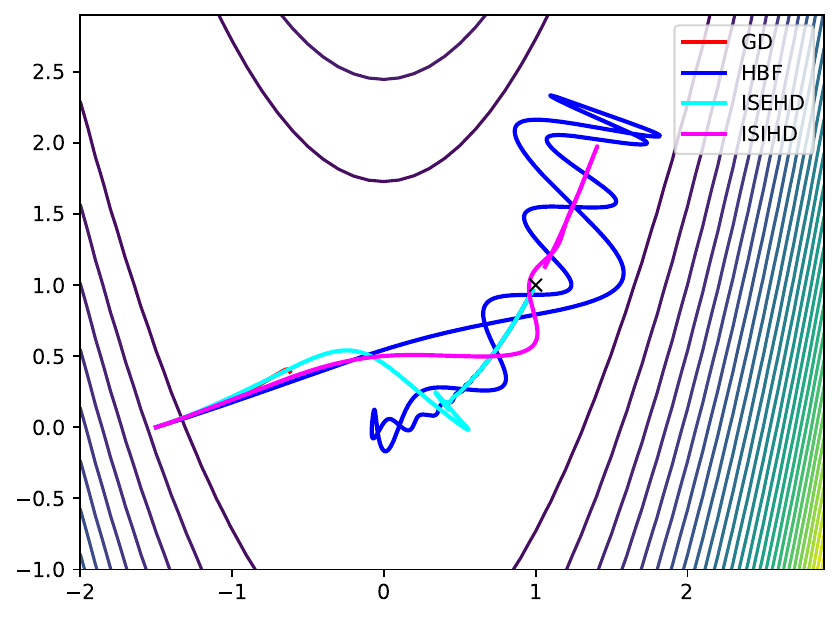}
         \caption{Trajectory.}
     \end{subfigure}
    \caption{Results on the Rosenbrock function with $\beta=0.02$.}
    \label{rose0}
\end{figure}
\begin{figure}[H]
    \centering
    \begin{subfigure}[b]{0.32\textwidth}
         \centering
         \includegraphics[width=\textwidth]{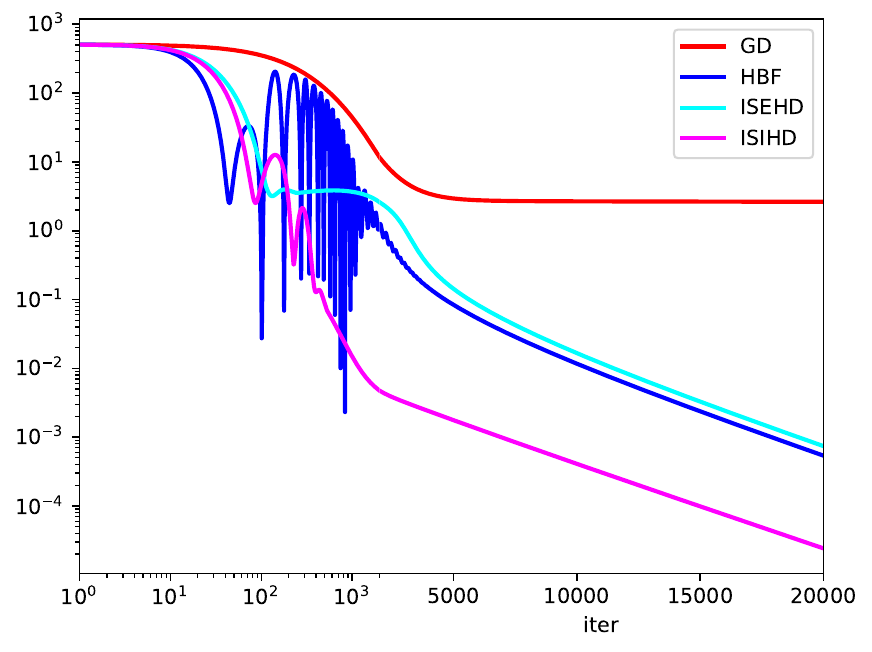}
         \caption{Residual vs Iteration.}
     \end{subfigure}
    \begin{subfigure}[b]{0.32\textwidth}
         \centering
         \includegraphics[width=\textwidth]{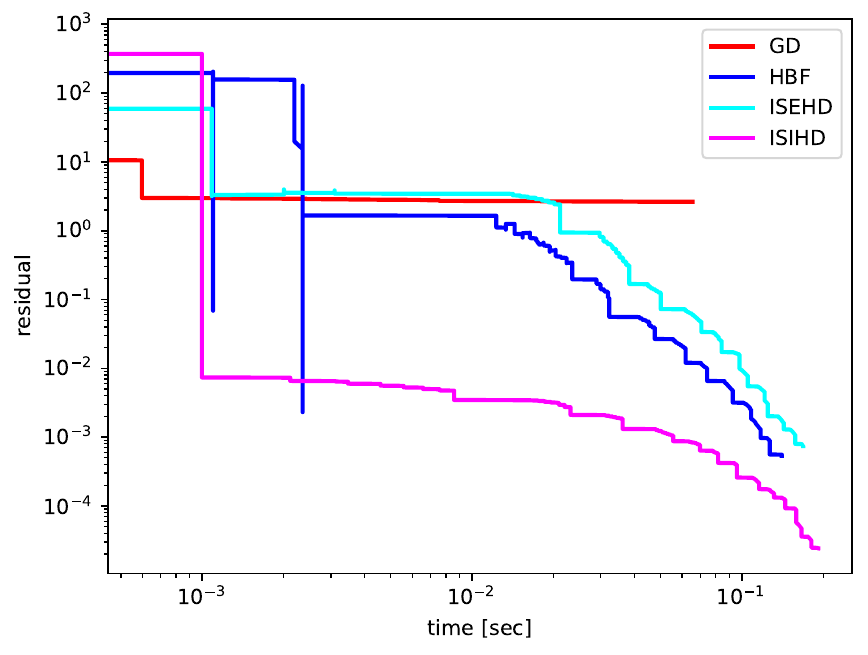}
         \caption{Residual vs Time.}
     \end{subfigure}
     \begin{subfigure}[b]{0.32\textwidth}
         \centering
         \includegraphics[width=\textwidth]{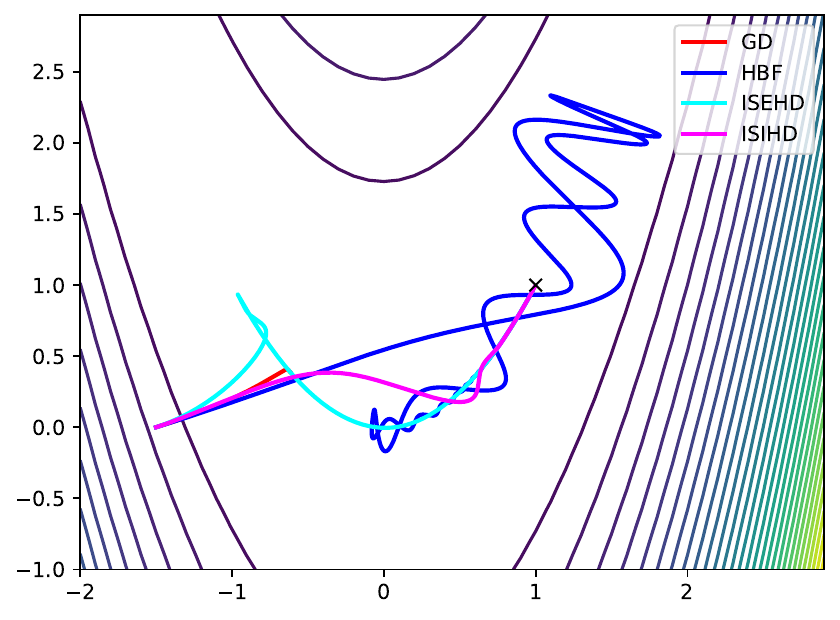}
         \caption{Trajectory.}
     \end{subfigure}
    \caption{Results on the Rosenbrock function with $\beta=0.04$.}
    \label{rose}
\end{figure}

We can notice that the iterates generated by \eqref{isehd-disc} and \eqref{isihd-disc} oscillate much less towards the minimum than \eqref{nhbf}, and this damping effect is more notorious as $\beta$ gets larger. In the case $\beta=0.02$, we observe there are still some oscillations, which benefit the dynamic generated by \eqref{isehd-disc} more than the one generated by \eqref{isihd-disc}. However, we have the opposite effect in the case $\beta=0.04$, where the oscillations are more damped. These three methods (\eqref{isehd-disc}, \eqref{isihd-disc}, \eqref{nhbf}) share a similar asymptotic convergence rate, which is linear as predicted (recall $f$ is {\L}ojasiewicz with exponent $1/2$), and they are significantly faster than \eqref{ngd}.
%\todo{The results need some comments and discussion.}
    
%%%%%%%%%%%%%%%%%%%%%%%%%%%%%%
\subsection{Image Deblurring}
%\begin{comment}
%In the context of greyscale images with length and height (in pixels) of $n_x\in\N$ and $n_y\in\N$, respectively, and  $N=n_x n_y$ the total number of pixels. We consider $A\in\R^{N\times N}$ to be a blur operator (with reflecting boundary conditions) with rows $(A_i)_{i=1}^N$, $b\in\R^N$ a noisy observation of an objective flattened image ($u^{\star}\in\R^N$) through the blur operator, \ie $$b_i=\langle A_i,u^{\star}\rangle+\mathcal{N}(0,5*10^{-3}),$$ where $\mathcal{N}(0,\sigma^2)$ is the normal distribution with $0$ mean and $\sigma^2$ variance), and $K_x, K_y\in \R^{N\times N}$ the discrete derivative operator on $x$ and $y$, respectively, with rows $(K_x)_{i=1}^N$ and $(K_y)_{i=1}^N$, respectively.
%
%\smallskip
%
%We model the image deblurring problem with logarithmic regularizer as follows:
%$$\min_{u\in\R^N} f(u)\eqdef \frac{1}{2}\sum_{i=1}^N (\langle A_i,u\rangle-b_i)^2+\frac{\mu}{2}\sum_{i=1}^N \log(\rho+(\langle (K_x)_i,u\rangle)^2+(\langle (K_y)_i,u\rangle)^2),$$
%where $\mu,\rho$ are positive constants for numerical stability set to $5*10^{-5}$ and $10^{-3}$, respectively.\\
%\end{comment}

In the task of image deblurring, we are given a blurry and noisy (gray-scale) image $b\in \R^{n_x\times n_y}$ of size $n_x\times n_y$. The blur corresponds to a convolution with a known low-pass kernel. Let $A:\R^{n_x\times n_y}\rightarrow\R^{n_x\times n_y}$ be the blur linear operator. We aim to solve the (linear) inverse problem of reconstructing $u^{\star}\in\R^{n_x\times n_y}$ from the relation $b=A\bar{u}+\xi$, where $\xi$ is the noise, that is additive pixel-wise, has $0-$mean and is Gaussian. Through this experiment, we used $n_x=n_y=256$.

\smallskip

In order to reduce noise amplification when inverting the operator $A$, we solve a regularized optimization problem to recover $u^{\star}$ as accurately as possible. As natural images can be assumed to be smooth except for a (small) edge-set between objects in the image, we use a non-convex logarithmic regularization term that penalizes finite forward differences in horizontal and vertical directions of the image, implemented as linear operators $K_x,K_y:\R^{n_x\times n_y}\rightarrow\R^{n_x\times n_y}$ with Neumann boundary conditions. In summary, we aim to solve the following:
\[
\min_{u\in \R^{n_x\times n_y}} f(u), \quad f(u)\eqdef \frac{1}{2}\Vert Au-b\Vert^2+\frac{\mu}{2}\sum_{i=1}^{n_x}\sum_{j=1}^{n_y}\log (\rho+(K_x u)_{i,j}^2+(K_y u)_{i,j}^2),
\]
where $\mu,\rho$ are positive constants for regularization and numerical stability set to $5 \cdot 10^{-5}$ and $10^{-3}$, respectively. $f$ definable as the sum of compositions of definable mappings, and $\nabla f$ is Lipschitz continuous.

To solve the above optimization problem, we have used \eqref{isehd-disc} and \eqref{isihd-disc} with parameters $\beta=1.3,$ $\gamma_k \equiv 0.25, h=0.5$, and initial conditions $x_0=x_1=0_{n_x\times n_y}$. We compared both algorithms with the baseline algorithms \eqref{ngd}, \eqref{nhbf} (with the same initial condition). All algorithms were run for $250$ iterations. The results are shown in Figure~\ref{posterisehd} and Figure~\ref{fig2}.

In Figure~\ref{posterisehd}, the original image $\bar{u}$ is shown on the left. In the middle, we display the blurry and noise image $b$. Finally, the image recovered by \eqref{isehd-disc} is shown on the right.

\begin{figure}[H]
    \centering
    \includegraphics[scale=0.9]{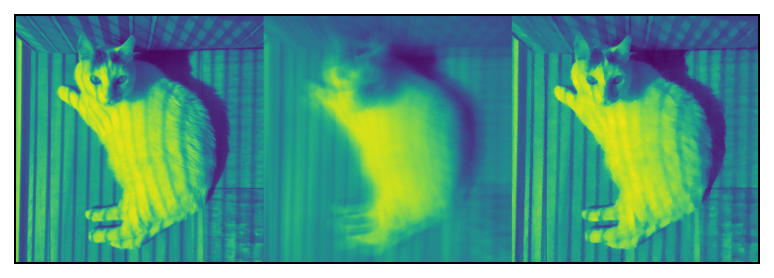}
    \caption{Results of the discretization of ISEHD.}
    \label{posterisehd}
\end{figure}

In Figure~\ref{fig2}, see that the residual plots of \eqref{isehd-disc} and \eqref{isihd-disc} overlap. Again, as expected, the trajectory of \eqref{isehd-disc} and \eqref{isihd-disc} has much less oscillation than \eqref{nhbf} which is a very desirable feature in practice. At the same time, \eqref{isehd-disc} and \eqref{isihd-disc} seem to converge faster, though \eqref{nhbf} eventually shows a similar convergence rate. Again, \eqref{ngd} is the slowest. Overall, \eqref{isehd-disc} and \eqref{isihd-disc} seem to take the best of both worlds: small oscillations and a faster asymptotic convergence rate.

\begin{figure}[H]
    \centering
    \begin{subfigure}[b]{0.49\textwidth}
         \centering
         \includegraphics[width=\textwidth]{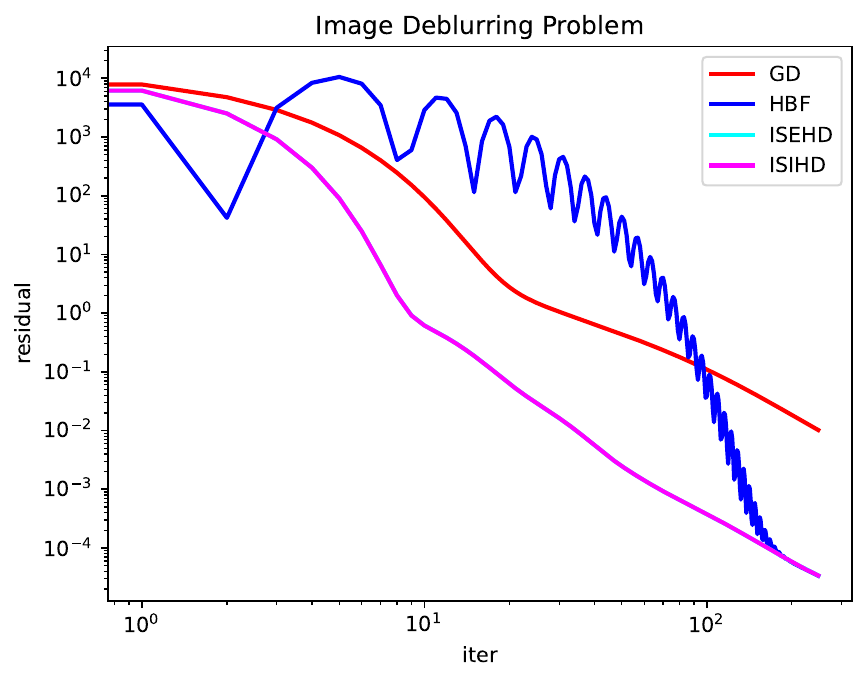}
         \caption{Residual vs Iteration.}
         \label{itres}
    \end{subfigure}
    \begin{subfigure}[b]{0.49\textwidth}
         \centering
         \includegraphics[width=\textwidth]{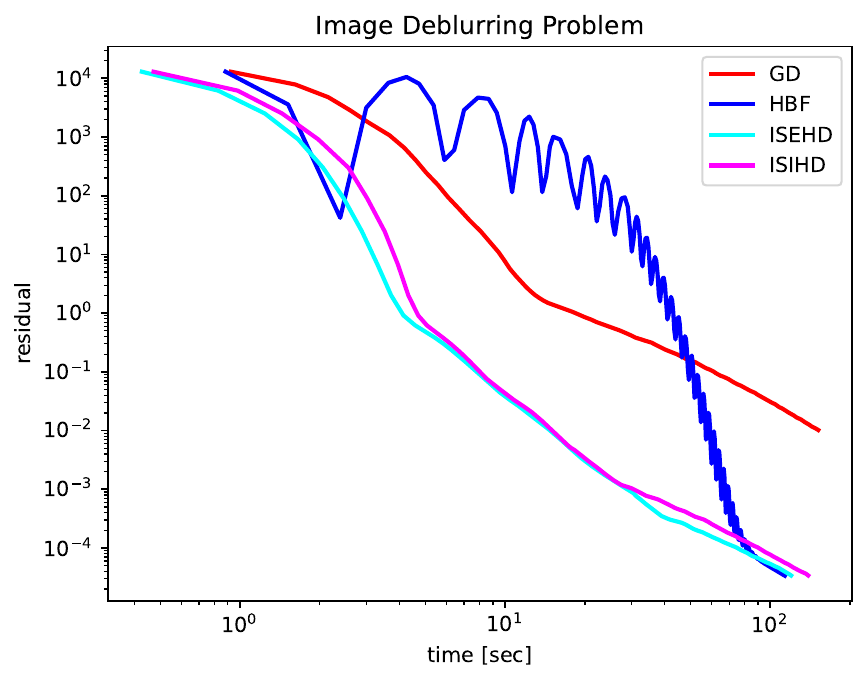}
         \caption{Residual vs Time.}
         \label{timeres}
    \end{subfigure}
    \caption{Results on the image deblurring problem.}\label{fig2}
\end{figure}

\smallskip

\begin{comment}
In terms of the resolution of the final image returned by \eqref{ngd} and \eqref{isehd-disc}, there are no significant differences to the eye. However, since the one generated by \eqref{isehd-disc} (and also by 
\eqref{isihd-disc} and \eqref{nhbf}) has a faster convergence rate than \eqref{ngd}, we will prefer the former method to obtain a lower residual (\ie a ``better'' solution) for practically any number of iterations.
\end{comment}

%\todo{The description of the experiments is not complete. The figures need more comments (what are you showing) and better discussion (why ISEHD is better).}

%%%%%%%%%%%%%%%%%%%%%%%%%%%%%%%%%%%%%%%%%%%%%%%%%%%%%%
\section{Conclusion and Perspectives}
%%%%%%%%%%%%%%%%%%%%%%%%%%%%%%%%%%%%%%%%%%%%%%%%%%%%%%
%\todo{I am not should we should keep this. The paper is already long.}
We conclude that:
\begin{itemize}
    \item  Under definability conditions on the objective and suitable conditions on $\gamma$, we obtain convergence of the trajectory of \eqref{eq:isehd} and \eqref{eq:isihd}. Besides, in the autonomous setting, and under a Morse condition, the trajectory almost surely converges to a local minimum of the objective.
    \item We obtain analogous properties for the respective proposed algorithmic schemes \eqref{isehd-disc} and \eqref{isihd-disc}.
    \item The inclusion of the term $\beta$ helps to reduce oscillations towards critical points, and, when chosen appropriately, without reducing substantially the speed of convergence of the case $\beta=0$, \ie the one for Heavy Ball with Friction method.
    \item The selection of $\beta$ is important. If it is chosen too close to zero, it may not significantly reduce oscillations. Conversely, if it is chosen too large (even within theoretical bounds), the trade-off for reduced oscillations might be a worse convergence rate.
\end{itemize}

Several open problems are worth investigating in the future:
\begin{comment}
    %\begin{itemize}
    \item 
 We could adapt \eqref{isehd-disc} and its properties to the Bregman case. For a differentiable function $g:\mathrm{dom}(g)\rightarrow\R$, we define $D_g:\mathrm{dom}(g)\times\mathrm{int}\phantom{.} \mathrm{dom}(g)\rightarrow\R$ as: $$D_g(x,y)\eqdef g(x)-g(y)-\langle \nabla g(y),x-y\rangle.$$  
    We consider a function $\phi:\mathrm{dom}(\phi)\rightarrow\R$ such that:
    \begin{itemize}
        \item is differentiable;  
        \item is $\mu-$ strongly convex with respect to some norm;
        \item \label{th} there exists $L_f>0$ such that $L_f\phi-f$ is a convex function. 
    \end{itemize} 
    And we consider the algorithm:
    \begin{equation}\label{isehd-disc-breg}\tag{ISEHD-Disc-Breg}
        \begin{aligned}
            \begin{cases}
                y_k&=x_k+\alpha_k(x_k-x_{k-1})-\beta_k(\nabla f(x_k)-\nabla f(x_{k-1})),\\
                x_{k+1}&=\argmin_{x\in\R^d}\langle \nabla f(x_k),x-y_k\rangle+\frac{1}{s_k}D_{\phi}(x,y_k).
            \end{cases}
        \end{aligned}
    \end{equation}
    We notice that we recover \eqref{isehd-disc} when $\phi(x)=\frac{\Vert x\Vert_2^2}{2}$.
    
    \smallskip
 \end{comment}   
 \begin{itemize}
 \item Replacing the global Lipschitz continuity assumption on the gradient in Theorems~\ref{convisehddisc} and \ref{convisihddisc} with a local Lipschitz continuity assumption as proposed in \cite{panageas, cedric,jia}.
 %\item Proposing discrete schemes of \eqref{eq:isehd} and \eqref{eq:isihd} with a variable stepsize, which can be computed, for instance, by backtracking.
 \item Extending our results to (non-euclidian) Bregman geometry.
 \item Extending our results to the non-smooth setting as proposed in \cite{casterainertial}.
 \end{itemize}

\bibliographystyle{unsrt}
\smaller
\bibliography{citas}

\begin{thebibliography}{10}

\bibitem{palis}
J.~Palis and W.~de~Melo.
\newblock Geometric theory of dynamical systems: An introduction.
\newblock {\em Springer-Verlag}, 1982.

\bibitem{loj1}
S.~{\L}ojasiewicz.
\newblock Une propri\'et\'e topologique des sous-ensembles analytiques r\'eels.
\newblock {\em Les \'equations aux d\'eriv\'ees partielles}, 117:87--89, 1963.

\bibitem{loj2}
S.~{\L}ojasiewicz.
\newblock Ensembles semi-analytiques.
\newblock {\em Lectures Notes IHES (Bures-sur-Yvette)}, 1965.

\bibitem{kur}
K.~Kurdyka.
\newblock On gradients of functions definable in o-minimal structures.
\newblock {\em Ann. Inst. Fourier}, 48 (3):769--783, 1998.

\bibitem{chill}
Ralph Chill and Alberto Fiorenza.
\newblock Convergence and decay rate to equilibrium of bounded solution of quasilinear parabolic equations.
\newblock {\em Jounral of Differential Equations}, 228:611--632, 2006.

\bibitem{aris}
J\'{e}r\^{o}me Bolte, Aris Daniilidis, and Adrian Lewis.
\newblock The {{\L}}ojasiewicz inequality for nonsmooth subanalytic functions with applications to subgradient dynamical systems.
\newblock {\em SIAM J. Opt.}, 17 (4):1205--1223, 2007.

\bibitem{boltearis}
J\'{e}r\^{o}me Bolte, Aris Daniilidis, Olivier Ley, and Laurent Mazet.
\newblock Characterizations of {{\L}}ojasiewicz inequalities and applications.
\newblock {\em Transactions of the American Mathematical Society}, 362 (6):3319--3363, 2008.

\bibitem{boltequasi}
P.~B\'egout, J.~Bolte, and M.~Jendoubi.
\newblock On damped second-order gradient systems.
\newblock {\em Journal of Differential Equations}, 259:3115--3143, 2015.

\bibitem{bolte}
J\'{e}r\^{o}me Bolte, Trong~Phong Nguyen, Juan Peypouquet, and Bruce~W. Suter.
\newblock From error bounds to the complexity of first-order descent methods for convex functions.
\newblock In Jon Lee and Sven Leyffer, editors, {\em Mathematical Programming}, volume 165, pages 471--507. Springer, 2016.

\bibitem{frankel}
Pierre Frankel, Guillaume Garrigos, and Juan Peypouquet.
\newblock Splitting methods with variable metric for {K}urdyka- {{\L}}ojasiewicz functions and general convergence rates.
\newblock {\em Journal of Optimization Theory and Applications}, 165(3):874--900, 2014.

\bibitem{noll}
Dominikus Noll.
\newblock Convergence of non-smooth descent methods using the {K}urdyka{\^a}-{\l}ojasiewicz inequality.
\newblock {\em Journal of Optimization Theory and Applications}, 160:553--572, 2014.

\bibitem{aujol}
J-F. Aujol, Ch. Dossal, and A.~Rondepierre.
\newblock Convergence rates of the {H}eavy-ball method with {{\L}}ojasiewicz property.
\newblock {\em HAL-02928958}, 2020.

\bibitem{pol}
Hamed Karimi, Julie Nutini, and Mark Schmidt.
\newblock Linear convergence of gradient and proximal-gradient methods under the polyak-lojasiewicz condition.
\newblock In {\em Machine Learning and Knowledge Discovery in Databases}. Springer, 2016.

\bibitem{absil}
P.A. Absil, R.~Mahony, and B.~Andrews.
\newblock Convergence of the iterates of descent methods for analytic cost functions.
\newblock {\em Siam Journal on Optimization}, 16 (2):531--547, 2005.

\bibitem{nonsm}
J\'{e}r\^{o}me Bolte, Aris Daniilidis, and Adrian Lewis.
\newblock The {{\L}}ojasiewicz inequality for nonsmooth subanalytic functions with applications to subgradient dynamical systems.
\newblock {\em SIAM Journal on Optimization}, 17 (4):1205--1223, 2007.

\bibitem{tam}
Hedy Attouch, J\'{e}r\^{o}me Bolte, and Benar~Fux Svaiter.
\newblock Convergence of descent methods for semi-algebraic and tame problems: proximal algorithms, forward-backward splitting, and regularized {G}auss-{S}eidel methods.
\newblock {\em Mathematical Programming}, Ser. A, 137:91--129, 2013.

\bibitem{ipiano}
Peter Ochs, Yunjin Chen, Thomas Brox, and Thomas Pock.
\newblock ipiano: Inertial proximal algorithm for nonconvex optimization.
\newblock {\em SIAM Journal on Imaging Sciences}, 7 (2):1388--1419, 2014.

\bibitem{clarke}
J\'{e}r\^{o}me Bolte, Aris Daniilidis, and Masahiro Shiota.
\newblock Clarke subgradients of stratifiable functions.
\newblock {\em SIAM Journal on Optimization}, 18 (2):556--572, 2007.

\bibitem{AC1}
Hedy Attouch and Alexandre Cabot.
\newblock Asymptotic stabilization of inertial gradient dynamics with time-dependent viscosity.
\newblock {\em J. Differential Equations}, 263(9):5412--5458, 2017.

\bibitem{su}
Weijie Su, Stephen Boyd, and Emmanuel~J. Cand\`es.
\newblock A differential equation for modeling {N}esterov's accelerated gradient method: Theory and insights.
\newblock {\em Journal of Machine Learning Research}, 17:1--43, 2016.

\bibitem{1983}
Y.E. {N}esterov.
\newblock A method of solving a convex programming problem with convergence rate $\mathcal{O}(1/k^2)$.
\newblock {\em Doklady Akademii Nauk SSSR}, 269(3):543--547, 1983.

\bibitem{cabot}
H.~Attouch and A.~Cabot.
\newblock Asymptotic stabilization of inertial gradient dynamics with time-dependent viscosity.
\newblock {\em Journal of Differential Equations}, 263-9:5412--5458, 2017.

\bibitem{faster1k2}
Hedy Attouch and Juan Peypouquet.
\newblock The rate of convergence of {N}esterov's accelerated forward-backward method is actually faster than $\frac{1}{k^2}$.
\newblock {\em SIAM Journal on Optimization}, 26(3):1824--1834, 2016.

\bibitem{CD}
A.~Chambolle and Ch. Dossal.
\newblock On the convergence of the iterates of the ``fast iterative shrinkage/thresholding algorithm''.
\newblock {\em J. Optim. Theory Appl.}, 166(3):968--982, 2015.

\bibitem{speedingup}
Boris Polyak.
\newblock Some methods of speeding up the convergence of iteration methods.
\newblock {\em USSR Computational Mathematics and Mathematical Physics}, 1964.

\bibitem{Alv}
Felipe Alvarez.
\newblock On the minimizing property of a second order dissipative system in {H}ilbert spaces.
\newblock {\em SIAM J. Control Optim.}, 38(4):1102--1119, 2000.

\bibitem{27}
A.~Haraux and Jendoubi M.A.
\newblock On a second order dissipative {ODE} in {H}ilbert space with an integrable source term.
\newblock {\em Acta Math. Sci.}, 32:155--163, 2012.

\bibitem{heavyb}
Hedy Attouch, Xavier Goudou, and Patrick Redont.
\newblock The heavy ball with friction method. i- the continuous dynamical system.
\newblock {\em Communications in Contemporary Mathematics}, 2 (1), 2011.

\bibitem{goudou}
X.~Goudou and J.~Munier.
\newblock The gradient and heavy ball with friction dynamical systems: the quasiconvex case.
\newblock {\em Mathematical Programming}, Ser B, 116:173--191, 2009.

\bibitem{apido}
Vassilis Apidopoulos, Nicolo Ginatta, and Silvia Villa.
\newblock Convergence rates for the heavy-ball continuous dynamics for non-convex optimization, under {P}olyak-{{\L}}ojasiewicz condition.
\newblock {\em Journal of Global Optimization}, 84:563--589, 2022.

\bibitem{noncon}
S.K. Zavriev and F.V. Kostyuk.
\newblock Heavy-ball method in nonconvex optimization problems.
\newblock {\em Models of Ecologic and Economic Systems}, 4:336--341, 1993.

\bibitem{peter}
Peter Ochs.
\newblock Local convergence of the {H}eavy-ball method and i{P}iano for non-convex optimization.
\newblock {\em arXiv:1606.09070}, 2016.

\bibitem{AABR}
F.~Alvarez, H.~Attouch, J.~Bolte, and P.~Redont.
\newblock A second-order gradient-like dissipative dynamical system with {H}essian-driven damping. {A}pplication to optimization and mechanics.
\newblock {\em J. Math. Pures Appl. (9)}, 81(8):747--779, 2002.

\bibitem{APR2}
Hedy Attouch, Juan Peypouquet, and Patrick Redont.
\newblock Fast convex optimization via inertial dynamics with {H}essian driven damping.
\newblock {\em J. Differential Equations}, 261(10):5734--5783, 2016.

\bibitem{9}
H.~Attouch, Z.~Chbani, J.~Fadili, and H.~Riahi.
\newblock First order optimization algorithms via inertial systems with {H}essian driven damping.
\newblock {\em Mathematical Programming}, 193:113--155, 2020.

\bibitem{alecsa}
C.~Alecsa, S.~L{\'a}szl{\'o}, and T.~Pinta.
\newblock An extension of the second order dynamical system that models {N}esterov's convex gradient method.
\newblock {\em Appl. Math. Optim.}, 84:1687--1716, 2021.

\bibitem{MJ}
Michael Muehlebach and Michael~I. Jordan.
\newblock Optimization with momentum: dynamical, control-theoretic, and symplectic perspectives.
\newblock {\em J. Mach. Learn. Res.}, 22:Paper No. 73, 50, 2021.

\bibitem{Muehlebach19}
M.~Muehlebach and M.~I. Jordan.
\newblock A dynamical systems perspective on {N}esterov acceleration.
\newblock In {\em Proceedings of the 36 th International Conference on Machine Learning}, volume~97. PMLR, 2019.

\bibitem{hessianpert}
Hedy Attouch, Jalal Fadili, and Vyacheslav Kungurtsev.
\newblock On the effect of perturbations in first-order optimization methods with inertia and {H}essian driven damping.
\newblock {\em Evolution equations and Control}, 12(1):71--117, 2023.

\bibitem{casterainertial}
Camille Castera, Jerome Bolte, Cédric Févotte, and Edouard Pauwels.
\newblock An inertial {N}ewton algorithm for deep learning.
\newblock {\em Journal of Machine Learning}, 22:1--31, 2021.

\bibitem{casteraavoid}
Camille Castera.
\newblock Inertial newton algorithms avoiding strict saddle points.
\newblock {\em arXiv:2111.04596}, 2021.

\bibitem{35}
B.~Shi, S.S. Du, M.I. Jordan, and Su~W.J.
\newblock Understanding the acceleration phenomenon via high resolution differential equations.
\newblock {\em Math. Program.}, 2021.

\bibitem{8}
H.~Attouch, A.~Cabot, Chbani Z., and H.~Riahi.
\newblock Accelerated forward-backward algorithms with perturbations: Application to {T}ikhonov regularization.
\newblock {\em J. Optim. Theory Appl.}, 179:1--36, 2018.

\bibitem{11}
H.~Attouch, Z.~Chbani, J.~Peypouquet, and P.~Redont.
\newblock Fast convergence of inertial dynamics and algorithms with asymptotic vanishing viscosity.
\newblock {\em Math. Program. Ser. B}, 168:123--175, 2018.

\bibitem{20}
C.~Dossal and J.F. Aujol.
\newblock Stability of over-relaxations for the forward-backward algorithm, application to fista.
\newblock {\em SIAM J. Optim.}, 25:2408--2433, 2015.

\bibitem{34}
M.~Schmidt, N.~Le~Roux, and F.~Bach.
\newblock Convergence rates of inexact proximal-gradient methods for convex optimization.
\newblock {\em NIPS'11}, 25th Annual Conference, 2011.

\bibitem{37}
S.~Villa, S.~Salzo, and Baldassarres L.
\newblock Accelerated and inexact forward-backward.
\newblock {\em SIAM J. Optim.}, 23:1607--1633, 2013.

\bibitem{19}
H.~Attouch, J.~Peypouquet, and P.~Redont.
\newblock Fast convex minimization via inertial dynamics with {H}essian driven damping.
\newblock {\em J. Differential Equations}, 261:5734--5783, 2016.

\bibitem{10}
H.~Attouch, Z.~Chbani, J.~Fadili, and H.~Riahi.
\newblock Convergence of iterates for first-order optimization algorithms with inertia and {H}essian driven damping.
\newblock {\em Optimization}, 2021.

\bibitem{haraux}
A.~Haraux and M.A. Jendoubi.
\newblock Convergence of solution of second-order gradient-like systems with analytic nonlinearities.
\newblock {\em Journal of Differential Equations}, 144:313--320, 1998.

\bibitem{AttouchAlternating10}
H.~Attouch, J.~Bolte, P.~Redont, and A.~Soubeyran.
\newblock Proximal alternating minimization and projection methods for nonconvex problems: An approach based on the kurdyka-{{\L}}ojasiewicz inequality.
\newblock {\em Mathematics of Operations Research}, 35(2):438--457, 2010.

\bibitem{jingwei}
Jingwei Liang.
\newblock Convergence rates of first-order operator splitting methods.
\newblock {\em Th\`ese de doctorat}, 2016.

\bibitem{blockcoord}
Peter Ochs.
\newblock Unifying abstract inexact convergence theorems and block coordinate variable metric ipiano.
\newblock {\em SIAM Journal on Optimization}, 29 (1):541--570, 2019.

\bibitem{anabstract}
Silvia Bonettini, Peter Ochs, Marco Prato, and Simone Rebegoldi.
\newblock An abstract convergence framework with application to inertial inexact forward-backward methods.
\newblock {\em Computational Optimization and Applications}, 84:319--362, 2023.

\bibitem{bot1}
Radu~Ioan Bot and Erno~Robert Csetnek.
\newblock An inertial tseng's type proximal algorithm for nonsmooth and nonconvex optimization problems.
\newblock {\em Journal of Optimization Theory and Applications}, 171:600--616, 2015.

\bibitem{proximal}
J\'{e}r\^{o}me Bolte, Shoham Sabach, and Marc Teboulle.
\newblock Proximal alternating linearized minimization for nonconvex and nonsmooth problems.
\newblock {\em Mathematical Programming}, 146:459--494, 2014.

\bibitem{bot2}
Radu~Ioan Bot, Erno~Robert Csetnek, and Szilard~Csaba Laszlo.
\newblock An inertial forward-backward algorithm for the minimization of the sum of two nonconvex functions.
\newblock {\em EURO Journal on Computational Optimization}, 4 (1):3--25, 2016.

\bibitem{fba}
Szilárd~Csaba László.
\newblock Forward-backward algorithms with different inertial terms for structured non-convex minimization problems.
\newblock {\em Journal of Optimization Theory and Applications}, 198:387--427, 2023.

\bibitem{gip}
Zhongming Wu and Min Li.
\newblock General inertial proximal gradient method for a class of nonconvex nonsmooth optimization problems.
\newblock {\em Computation Optimization and Applications}, 73:129--158, 2019.

\bibitem{shub_global_1987}
M.~Shub.
\newblock {\em Global {Stability} of {Dynamical} {Systems}}.
\newblock Springer, 1987.

\bibitem{Brunovsky97}
P.~Brunovsk\'y and P.~Pol\'acak.
\newblock The {M}orse-{S}male structure of a generic reaction-diffusion equation in higher space dimension.
\newblock {\em Journal of Differential Equations}, 135(1):129--181, 1997.

\bibitem{BrandiereDuflo96}
Odile Brandi\`ere and Marie Duflo.
\newblock Les algorithmes stochastiques contournent-ils les pi\`eges ?
\newblock {\em Annales de l'I.H.P. Probabilit\'es et statistiques}, 32(3):395--427, 1996.

\bibitem{Pemantle90}
Robin Pemantle.
\newblock {Nonconvergence to Unstable Points in Urn Models and Stochastic Approximations}.
\newblock {\em The Annals of Probability}, 18(2):698 -- 712, 1990.

\bibitem{agd}
Chi Jin, Praneeth Netrapalli, and Michael~I. Jordan.
\newblock Accelerated gradient descent escapes saddle points faster than gradient descent.
\newblock {\em Proceedings of Machine Learning Research}, 75:1--44, 2018.

\bibitem{gadat2}
S\'ebastien Gadat, Fabien Panloup, and Saadane Sofiane.
\newblock Stochastic heavy ball.
\newblock {\em Electronic Journal of Statistics}, 12:461--529, 2018.

\bibitem{hbnoise}
Wenqing Hu, Chris Junchi~Li, and Xiang Zhou.
\newblock On the global convergence of continuous-time stochastic heavy ball method for nonconvex optimization.
\newblock {\em arXiv:1712.05733v4}, 2019.

\bibitem{Lee16}
Jason~D. Lee, Max Simchowitz, Michael~I. Jordan, and Benjamin Recht.
\newblock Gradient descent only converges to minimizers.
\newblock In Vitaly Feldman, Alexander Rakhlin, and Ohad Shamir, editors, {\em 29th Annual Conference on Learning Theory}, volume~49 of {\em Proceedings of Machine Learning Research}, pages 1246--1257, Columbia University, New York, New York, USA, 23--26 Jun 2016. PMLR.

\bibitem{localminima}
Jason~D. Lee, Ioannis Panageas, Georgios Piliouras, Max Simchowitz, Michael~I. Jordan, and Benjamin Recht.
\newblock First-order methods almost always avoid strict saddle points.
\newblock {\em Mathematical Programming}, 176:311--337, 2019.

\bibitem{panageas}
Ioannis Panageas and Georgios Piliouras.
\newblock Gradient descent only converges to minimizers: Non-isolated critical points and invariant regions.
\newblock {\em 8th Innovations in Theoretical Computer Science Conference}, 2012.

\bibitem{cedric}
C\'edric Josz.
\newblock Global convergence of the gradient method for functions definable in o-minimal structures.
\newblock {\em Mathematical Programming}, 202:355--383, 2023.

\bibitem{wright}
Michael O'Neill and Stephen~J. Wright.
\newblock Behavior of accelerated gradient methods near critical points of nonconvex functions.
\newblock {\em Mathematical Programming}, 176:403--427, 2019.

\bibitem{bloc}
J.R. Silvester.
\newblock Determinant of block matrices.
\newblock {\em Math. Gaz.}, 84 (501):460--467, 2000.

\bibitem{aubinekeland}
Jean-Pierre Aubin and Ivar Ekeland.
\newblock Applied nonlinear analysis.
\newblock {\em Wiley-Interscience}, 1984.

\bibitem{coste2002intro}
M.~Coste.
\newblock An introduction to semialgebraic geometry.
\newblock Technical report, Institut de Recherche Mathematiques de Rennes, October 2002.

\bibitem{vandenDriesMiller96}
L.~{van den Dries} and C.~Miller.
\newblock Geometric categories and o-minimal structures.
\newblock {\em Duke Mathematical Journal}, 84:497--540, 1996.

\bibitem{omin}
M.~Coste.
\newblock An introduction to o-minimal geometry.
\newblock {\em RAAG Notes, Institut de Recherche Math{\~A}{\copyright}matiques de Rennes}, 1999.

\bibitem{haraux0}
A.~Haraux.
\newblock Syst\'emes dynamiques dissipatifs et applications.
\newblock {\em RMA 17}, 1991.

\bibitem{Perko96}
L.~Perko.
\newblock {\em Differential equations and dynamical systems}.
\newblock Springer, 1996.

\bibitem{garrigos}
Guillaume Garrigos.
\newblock Descent dynamical systems and algorithms for tame optimization and multi-objective problems.
\newblock {\em Th\'ese de doctorat}, 2015.

\bibitem{calckl}
Guoyin Li and Ting~Kei Pong.
\newblock Calculus of the exponent of {K}urdyka-{{\L}}ojasiewicz inequality and its applications to linear convergence of first-order methods.
\newblock {\em Foundation of Computational Mathematics}, 18:1199--1232, 2017.

\bibitem{attouchbolte}
Hedy Attouch and J\'{e}r\^{o}me Bolte.
\newblock On the convergence of the proximal algorithm for nonsmooth functions involving analytic features.
\newblock {\em Mathematical Programming}, 116 (1):5--16, 2007.

\bibitem{implicitreg}
Nathan Buskulic, Jalal Fadili, and Yvain Quéau.
\newblock Implicit regularization of the deep inverse prior trained with inertia.
\newblock {\em To be published}, 2025.

\bibitem{jia}
Xioaxi Jia, Christian Kanzow, and Patrick Mehlitz.
\newblock Convergence analysis of the proximal gradient method in the presence of the {K}urdyka-{\l}ojasiewicz property without global {L}ipschitz assumptions.
\newblock {\em SIAM Journal on Optimization}, 33 (4), 2023.

\end{thebibliography}

\end{document}